\documentclass[letterpaper,oneside,notitlepage]{amsart}%
\pdfoutput=1
\usepackage{amsfonts}
\usepackage{amsmath}
\usepackage{amssymb}
\usepackage{graphicx}
\usepackage{geometry}%
\providecommand{\U}[1]{\protect\rule{.1in}{.1in}}
\newtheorem{theorem}{Theorem}
\theoremstyle{plain}

\newtheorem{corollary}{Corollary}

\newtheorem{definition}{Definition}
\newtheorem{example}{Example}

\newtheorem{lemma}{Lemma}

\newtheorem{proposition}{Proposition}
\newtheorem{remark}{Remark}

\numberwithin{equation}{section}
\numberwithin{theorem}{section}
\numberwithin{definition}{section}
\numberwithin{corollary}{section}
\numberwithin{lemma}{section}
\numberwithin{proposition}{section}
\numberwithin{example}{section}
\numberwithin{remark}{section}
\numberwithin{equation}{section}
\numberwithin{figure}{section}
\numberwithin{table}{section}
\begin{document}
\title[A $\operatorname{cosq}_{K}$-characterization of Aleksandrov spaces of
curvature $\leq K$]{A $K$-quadrilateral cosine characterization of Aleksandrov spaces of\\curvature bounded above}
\author{I.D. Berg}
\address{Department of Mathematics, University of Illinois at Urbana-Champaign, Urbana,
IL 61801, USA}
\email{berg@math.uiuc.edu}
\author{Igor G. Nikolaev}
\curraddr{Department of Mathematics, University of Illinois at \\
Urbana-Champaign, Urbana, IL 61801, USA}
\email{inik@illinois.edu}

\begin{abstract}
In this note, we extend the main results of our paper on quasilinearization
and curvature of Aleksandrov spaces of curvature $\leq0$ to curvature bounds
other than $0$. For non-zero $K$, we employ the previously introduced notion
of the $K$-quadrilateral cosine, which is the cosine under parallel transport
in model $K$-space, and which is denoted by $\operatorname{cosq}_{K}$. Our
principal result states that a geodesically connected metric space (of
diameter not greater than $\pi/\left(  2\sqrt{K}\right)  $ if $K>0$) is an
$\Re_{K}$ domain (otherwise known as a $\operatorname{CAT}\left(  K\right)  $
space) if and only if always $\operatorname{cosq}_{K}\leq1$ or always
$\operatorname{cosq}_{K}$ $\geq-1$. (We prove that in such spaces always
$\operatorname{cosq}_{K}\leq1$ is equivalent to always $\operatorname{cosq}%
_{K}$ $\geq-1$). As a corollary, we give necessary and sufficient conditions
for a Cauchy complete semimetric space to be a complete $\Re_{K}$ domain. We
show that in our theorem the diameter hypothesis for positive $K$ is sharp and
we prove an extremal theorem when $\left\vert \operatorname{cosq}%
_{K}\right\vert $ attains an upper bound of $1$. We derive from our main
theorem and our previous result for $K=0$ a complete solution of Gromov's
curvature problem in the context of Aleksandrov spaces of curvature bounded
above. Then we establish the $K$-Euler's inequality and the extremal theorem
for equality in the $K$-Euler's inequality in an $\Re_{K}$ domain.

\end{abstract}
\maketitle

\section{Introduction}

Classes of Riemannian metrics that satisfy uniform sectional curvature bounds
often arise in geometry. In his fundamental papers \cite{A1951} and
\cite{A1957a}, Aleksandrov presented the upper and lower curvature conditions
for a geodesically connected metric space, i.e., a metric space in which any
two points can be joined by a shortest. In particular, Aleksandrov introduced
the notion of an $\Re_{K}$ domain, also known as a $\operatorname{CAT}\left(
K\right)  $ space, a geodesically connected metric space of curvature $\leq K$
in the sense of Aleksandrov, in which shortests depend continuously on their
end points and in which the perimeter of every geodesic triangle is less than
$2\pi/\sqrt{K}$ if $K>0$.

In this note, we present a deeper metric analysis of Aleksandrov's upper
boundedness curvature condition by extending the main results of our paper
\cite{BergNik2008} for $K=0$ to the case of non-zero $K$. There are \ striking
differences in our approach to non-zero $K$ that require different methods.
Our results are not local; hence the lack of linearity in the model space
presents substantial conceptual and technical problems.

Our main result in \cite{BergNik2008} states that a geodesically connected
metric space $\left(  \mathcal{M},\rho\right)  $ is an $\Re_{0}$ domain if and
only if for every two ordered pairs of distinct points $\overrightarrow
{AP}=\left(  A,P\right)  $ and $\overrightarrow{BQ}=\left(  B,Q\right)  $ in
$\mathcal{M}$, called (non-zero) \emph{bound vectors}, their
\emph{quadrilateral cosine, }$\operatorname{cosq}\left(  \overrightarrow
{AP},\overrightarrow{BQ}\right)  $, satisfies the following inequality
\[
\operatorname{cosq}\left(  \overrightarrow{AP},\overrightarrow{BQ}\right)
\equiv\frac{\rho^{2}\left(  A,Q\right)  +\rho^{2}\left(  B,P\right)  -\rho
^{2}\left(  A,B\right)  -\rho^{2}\left(  P,Q\right)  }{2\rho\left(
A,P\right)  \rho\left(  B,Q\right)  }\leq1.
\]

The quadrilateral cosine \ was introduced in \cite{Nik1990} under the name of
function $h$ and was used to construct the generalized Sasaki metric on the
set of tangent elements of a metric space and to obtain a pure metric
characterization of Riemannian spaces \cite{Nik1990}, \cite{Nik1999}.

The generalization of $\operatorname{cosq}$ to non-zero $K$ is not
straightforward. Let $K\neq0$ and $\kappa=\sqrt{\left\vert K\right\vert }$. In
what follows, $\widehat{\kappa}=\kappa=\sqrt{K}$ if $K>0$ and $\widehat
{\kappa}=i\kappa=i\sqrt{-K}$ if $K<0$. The following definition is equivalent
to Definition 3.2 in \cite{BergNik2007a}.

\begin{definition}
Let $\left(  \mathcal{M},\rho\right)  $ be a metric space and $A,P,B,Q\in
\mathcal{M}$ be such that $A\neq P,$ $B\neq Q$. If $K>0$, we assume that
$\rho\left(  A,P\right)  ,$ $\rho\left(  B,Q\right)  $ and $\rho\left(
A,B\right)  <\pi/\sqrt{K}$. Set
\begin{align*}
\rho\left(  A,P\right)   &  =x,\text{ }\rho\left(  B,Q\right)  =y,\text{ }%
\rho\left(  A,B\right)  =a,\\
\rho\left(  P,Q\right)   &  =b,\text{ }\rho\left(  P,B\right)  =d\text{ and
}\rho\left(  A,Q\right)  =f,
\end{align*}
as shown in Fig. \ref{fig1}.
%TCIMACRO{\FRAME{ftphFU}{1.9925in}{1.4326in}{0pt}{\Qcb{Definition of
%$\operatorname{cosq}_{K}$}}{\Qlb{fig1}}{fig1.eps}%
%{\special{ language "Scientific Word";  type "GRAPHIC";
%maintain-aspect-ratio TRUE;  display "PICT";  valid_file "F";
%width 1.9925in;  height 1.4326in;  depth 0pt;  original-width 2.1069in;
%original-height 1.5064in;  cropleft "0";  croptop "1";  cropright "1";
%cropbottom "0";  filename 'Fig1.eps';file-properties "XNPEU";}} }%
%BeginExpansion
\begin{figure}
[pth]
\begin{center}
\includegraphics[
natheight=1.506400in,
natwidth=2.106900in,
height=1.4326in,
width=1.9925in
]%
{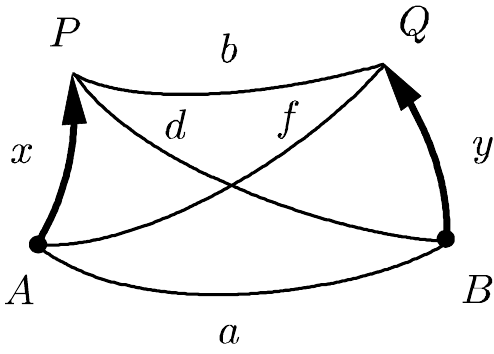}%
\caption{Definition of $\operatorname{cosq}_{K}$}%
\label{fig1}%
\end{center}
\end{figure}
%EndExpansion
Then the $K$\emph{-quadrilateral cosine} $\operatorname{cosq}_{K}\left(
\overrightarrow{AP},\overrightarrow{BQ}\right)  $ is defined by%
\begin{align*}
\operatorname{cosq}_{K}\left(  \overrightarrow{AP},\overrightarrow{BQ}\right)
&  =\frac{\cos\widehat{\kappa}b+\cos\widehat{\kappa}x\cos\widehat{\kappa}%
y}{\sin\widehat{\kappa}x\sin\widehat{\kappa}y}\\
&  -\frac{\left(  \cos\widehat{\kappa}x+\cos\widehat{\kappa}d\right)  \left(
\cos\widehat{\kappa}y+\cos\widehat{\kappa}f\right)  }{\left(  1+\cos
\widehat{\kappa}a\right)  \sin\widehat{\kappa}x\sin\widehat{\kappa}y}.
\end{align*}
In particular, if $K>0$, then%
\begin{align*}
\operatorname{cosq}_{K}\left(  \overrightarrow{AP},\overrightarrow{BQ}\right)
&  =\frac{\cos\kappa b+\cos\kappa x\cos\kappa y}{\sin\kappa x\sin\kappa y}\\
&  -\frac{\left(  \cos\kappa x+\cos\kappa d\right)  \left(  \cos\kappa
y+\cos\kappa f\right)  }{\left(  1+\cos\kappa a\right)  \sin\kappa x\sin\kappa
y}\text{,}%
\end{align*}
and if $K<0$, then%
\begin{align*}
\operatorname{cosq}_{K}\left(  \overrightarrow{AP},\overrightarrow{BQ}\right)
&  =\frac{\left(  \cosh\kappa x+\cosh\kappa d\right)  \left(  \cosh\kappa
y+\cosh\kappa f\right)  }{\left(  1+\cosh\kappa a\right)  \sinh\kappa
x\sinh\kappa y}\\
&  -\frac{\cosh\kappa b+\cosh\kappa x\cosh\kappa y}{\sinh\kappa x\sinh\kappa
y}.
\end{align*}
If $K=0$, we set $\operatorname{cosq}_{0}\left(  \overrightarrow
{AP},\overrightarrow{BQ}\right)  =\operatorname{cosq}\left(  \overrightarrow
{AP},\overrightarrow{BQ}\right)  $.
\end{definition}

We introduce the following conditions for a metric space $\left(
\mathcal{M},\rho\right)  $:

(i) The \emph{upper} \emph{four point }$\operatorname{cosq}_{K}$%
\emph{\ condition: }$\operatorname{cosq}_{K}\left(  \overrightarrow
{AP},\overrightarrow{BQ}\right)  \leq1$ for every pair of non-zero bound
vectors $\overrightarrow{AP}$ and $\overrightarrow{BQ}$ in $\mathcal{M}$ and
such that $\rho\left(  A,P\right)  , $ $\rho\left(  B,Q\right)  $ and
$\rho\left(  A,B\right)  <\pi/\sqrt{K}$ when $K>0$.

(ii) The \emph{lower} \emph{four point }$\operatorname{cosq}_{K}%
$\emph{\ condition: }$\operatorname{cosq}_{K}\left(  \overrightarrow
{AP},\overrightarrow{BQ}\right)  \geq-1 $ for every pair of non-zero bound
vectors $\overrightarrow{AP}$ and $\overrightarrow{BQ}$ in $\mathcal{M}$ and
such that $\rho\left(  A,P\right)  , $ $\rho\left(  B,Q\right)  $ and
$\rho\left(  A,B\right)  <\pi/\sqrt{K}$ when $K>0$.

We say that $\left(  \mathcal{M},\rho\right)  $ satisfies the \emph{one-sided
four point} $\operatorname{cosq}_{K}$\emph{\ condition }if it satisfies either
the upper four point $\operatorname{cosq}_{K}$ condition or the lower four
point $\operatorname{cosq}_{K}$ condition.

Our present main result is given by the following

\begin{theorem}
\label{MainTh}Let $K\neq0$ and let $\left(  \mathcal{M},\rho\right)  $ be a
geodesically connected metric space such that $\operatorname{diam}\left(
\mathcal{M}\right)  \leq\pi/\left(  2\sqrt{K}\right)  $ when $K>0$. Then
$\left(  \mathcal{M},\rho\right)  $ is an $\Re_{K}$ domain with the same
diameter restriction if and only if $\left(  \mathcal{M},\rho\right)  $
satisfies the one-sided $\operatorname{cosq}_{K}$\ condition.
\end{theorem}

\begin{remark}
As Example \ref{Ex_counter_1} shows, the restriction on the diameter of
$\left(  \mathcal{M},\rho\right)  $ for positive $K$ cannot be dropped and the
diameter bound in the hypothesis of Theorem \ref{MainTh} is sharp.
\end{remark}

\begin{remark}
A normed vector space of curvature $\leq K$ in the sense of Aleksandrov is an
inner product space \cite[p. 7]{A1957b}. Hence, we can complement the results
of the paper by Schoenberg \cite{Sch1952} by deriving from Theorem
\ref{MainTh} that a normed vector space is an inner product space if and only
if it satisfies the one-sided $\operatorname{cosq}_{K}$\ condition for some
positive $K$.
\end{remark}

Recall that the $K$-plane $\mathbb{S}_{K}$ is the Euclidean plane if $K=0$,
the open hemisphere of radius $1/\sqrt{K}$ if $K>0$ and the hyperbolic plane
of curvature $K$ if $K<0$. The definition of $K$-space $\mathbb{S}_{K}^{3}$ is similar.

\begin{remark}
Notice that in a domain of $\mathbb{S}_{K}^{3}$ of diameter less than
$\pi/\left(  2\sqrt{K}\right)  $ if $K>0$, the $K$-quadrilateral cosine of a
pair of non-zero bound vectors $\overrightarrow{AP}$ and $\overrightarrow{BQ}
$ equals the cosine of the angle between the vector $\exp_{B}^{-1}\left(
Q\right)  $ and the vector $\exp_{A}^{-1}\left(  P\right)  $ under the
Levi-Civita parallel translation from the point $A$ to the point $B$ along the
shortest joining these two points (see, Sec. \ref{cosqk_in_SK}). Moreover, in
$\mathbb{S}_{K}^{3}$, for every $K$, $\operatorname{cosq}_{K}$ always can be
interpreted as a cosine of an angle (Corollary \ref{CorcosqKbound}). By
Theorems \ref{MainTh} and \ref{ThRKBound}, in a geodesically connected metric
space satisfying the diameter restriction of Theorem \ref{MainTh},
$\operatorname{cosq}_{K}$ can be interpreted as the cosine of an angle when
and only when the metric space is an $\Re_{K}$ domain.
\end{remark}

If $K=0$, the upper four point $\operatorname{cosq}_{K}$ condition is
immediately equivalent to the lower four point $\operatorname{cosq}_{K}$
condition \cite[Introduction]{BergNik2008}. According to Examples
\ref{ExFPC_PosK} and \ref{ExFPC_NegK}, in a general metric space, this is not
true anymore for non-zero $K$. However, we derive from Theorem \ref{MainTh}
and Theorem \ref{ThRKBound} of Sec. \ref{Sec_RK} the following:

\begin{corollary}
Let $K\neq0$ and let $\left(  \mathcal{M},\rho\right)  $ be a geodesically
connected metric space such that $\operatorname{diam}\left(  \mathcal{M}%
\right)  \leq\pi/\left(  2\sqrt{K}\right)  $ when $K>0$. Then $\left(
\mathcal{M},\rho\right)  $ satisfies the upper four point $\operatorname{cosq}%
_{K}$ condition if and only if $\left(  \mathcal{M},\rho\right)  $ satisfies
the lower four point $\operatorname{cosq}_{K}$ condition.\newline
\end{corollary}

Recall that a polygonal curve $\mathcal{APQBA}$ in a Riemannian space is
called a Levi-Civita parallelogramoid \cite{Car1983} if the distances between
$A$ and $P$ and $B$ and $Q$ are equal, and the vectors $\exp_{A}^{-1}\left(
P\right)  $ and $\exp_{B}^{-1}\left(  Q\right)  $ are parallel along a
shortest joining $A$ to $B$. We say that a polygonal curve $\mathcal{APQBA}$
in $\mathbb{S}_{K}$ is a Levi-Civita trapezoid if either the vectors $\exp
_{A}^{-1}\left(  P\right)  $ and $\exp_{B}^{-1}\left(  Q\right)  $ are
parallel along the shortest $\mathcal{AB}$ or the vectors $\exp_{A}%
^{-1}\left(  P\right)  $ and $-\exp_{B}^{-1}\left(  Q\right)  $ are parallel
along the shortest $\mathcal{AB}$. A convex domain in $\mathbb{S}_{K}$
enclosed by a Levi-Civita trapezoid is called a Levi-Civita trapezoidal
domain. In particular, the set of points of a shortest in $\mathbb{S}_{K}$ is
a degenerate Levi-Civita trapezoidal domain. The following theorem
generalizing \cite[Theorem 15]{BergNik1998} and \cite[Theorem 6.2]%
{BergNik2007a} describes the extremal cases when $\operatorname{cosq}_{K} $
takes values $1$ or $-1$.

\begin{theorem}
\label{ThExtr}Let $K\neq0$ and let $\left(  \mathcal{M},\rho\right)  $ be a
geodesically connected metric space such that $\operatorname{diam}\left(
\mathcal{M}\right)  <\pi/\left(  2\sqrt{K}\right)  $ when $K>0$. If $\left(
\mathcal{M},\rho\right)  $ satisfies the one-sided four point
$\operatorname{cosq}_{K}$ condition, and for a pair of non-zero bound vectors
$\overrightarrow{AP}$ and $\overrightarrow{BQ}$ in $\mathcal{M}$, $\left\vert
\operatorname{cosq}_{K}\left(  \overrightarrow{AP},\overrightarrow{BQ}\right)
\right\vert $ $=1$, then the convex hull of the quadruple $\left\{
A,P,Q,B\right\}  $ is isometric to a Levi-Civita trapezoidal domain in
$\mathbb{S}_{K}$.\newline
\end{theorem}

By Example \ref{Ex_toExtr_th}, Theorem \ref{ThExtr} need not be true if
$\operatorname{diam}\left(  \mathcal{M}\right)  =\pi/\left(  2\sqrt{K}\right)
$ when $K>0 $.

Recall that a semimetric space is a distance space with a positive definite
and symmetric distance. A semimetric space $\left(  \mathcal{M},\rho\right)  $
is said to be \emph{weakly convex} if, for every $A,$ $B\in\mathcal{M}$, there
is $\lambda\in\left(  0,1\right)  $, such that, for every $\varepsilon>0$,
there is $C_{\varepsilon}\in\mathcal{M}$ satisfying the inequalities
$\left\vert \rho\left(  A,C_{\varepsilon}\right)  -\lambda\rho\left(
A,B\right)  \right\vert <\varepsilon$ and $\left\vert \rho\left(
B,C_{\varepsilon}\right)  -\left(  1-\lambda\right)  \rho\left(  A,B\right)
\right\vert <\varepsilon$. Cauchy sequences in a semimetric space and the
diameter of a semimetric space are defined in the same way as in a metric
space. Finally, notice that the upper and the lower four point
$\operatorname{cosq}_{K}$ conditions can also be stated for semimetric spaces.
We derive from Theorem \ref{MainTh} and Menger's theorem \cite[Theorem
14.1]{Bl1970} the following extension of \cite[Theorem 5]{BergNik2008} to
non-zero $K$:

\begin{theorem}
\label{Thsemimetr}Let $K\neq0$ and let $\left(  \mathcal{M},\rho\right)  $ be
a semimetric space such that $\operatorname{diam}\left(  \mathcal{M}\right)
\leq\pi/\left(  2\sqrt{K}\right)  $ when $K>0$. Then $\left(  \mathcal{M}%
,\rho\right)  $ is a complete $\Re_{K}$ domain with the same diameter
restriction if and only if the following conditions are satisfied:\newline(a)
$\left(  \mathcal{M},\rho\right)  $ is weakly convex.\newline(b) Each Cauchy
sequence in $\left(  \mathcal{M},\rho\right)  $ has a limit.\newline(c)
$\left(  \mathcal{M},\rho\right)  $ satisfies the one-sided four point
$\operatorname{cosq}_{K}$ condition.\newline
\end{theorem}

In his book \cite{Gr1999}, Gromov offered a method to define classes of metric
spaces corresponding to Riemannian manifolds with prescribed curvature
restrictions by introducing global and local $\mathcal{K}$-curvature classes.
Let $r\in\mathbb{N}$ and $M_{r}$ denote the set of all symmetric $r\times r$
matrices with zero diagonal entries and non-negative entries otherwise. Let
$\mathcal{X}$ be a set and $d:\mathcal{X}\times\mathcal{X}\rightarrow
\mathbb{R}$ be a non-negative function such that $d\left(  P,Q\right)
=d\left(  Q,P\right)  $ and $d\left(  P,Q\right)  =0$ if and only if $P=Q$,
for all $P,Q\in\mathcal{X}$. Then $K_{r}\left(  \mathcal{X}\right)  $ consists
of all matrices $A=\left(  a_{ij}\right)  $ in $M_{r}$ such that for every
$A\in K_{r}\left(  \mathcal{X}\right)  $ there is an $r$-tuple $\left\{
P_{1},P_{2},...,P_{r}\right\}  \subseteq\mathcal{X}$ satisfying $a_{ij}%
=d\left(  P_{i},P_{j}\right)  $, $i,j=1,2,...,r$. A subset
$\mathcal{K\subseteq}M_{r}$ defines the (global) $\mathcal{K}$-curvature class
as follows. The $\mathcal{K}$-curvature class consists of all $\left(
\mathcal{X},d\right)  $ such that $K_{r}\left(  \mathcal{X}\right)
\mathcal{\subseteq K}$. Gromov's curvature problem is the problem of a
meaningful geometric description of $\mathcal{K}$-curvature classes
(\cite{Gr1999}, Section 1.19$_{+}$, Curvature Problem).

In \cite[Theorem 8]{BergNik2008} we gave a solution of Gromov's curvature
problem in the context of $\Re_{0}$ domains and therefore for Aleksandrov
spaces of non-positive curvature. In this note, we obtain a complete solution
of Gromov's curvature problem in the context of $\Re_{K}$ domains and
Aleksandrov spaces of curvature $\leq K$ by solving Gromov's curvature problem
for non-zero $K$ as a corollary of Theorems \ref{MainTh} and \ref{Thsemimetr}.

Let $\mathfrak{M}_{G}$ be the set of all geodesically connected metric spaces
and $\mathfrak{M}_{S}$ denote the set of all semimetric spaces satisfying
conditions (a) and (b) of Theorem \ref{Thsemimetr}. For $\kappa>0$, let
$\mathcal{K}^{+}\left(  \kappa^{2}\right)  $ denote the set of matrices
$A=\left(  a_{ij}\right)  \in M_{4}$ such that%

\begin{align*}
&  \left(  \cos\kappa a_{23}+\cos\kappa a_{12}\cos\kappa a_{34}\right)
\left(  1+\cos\kappa a_{14}\right)  -\\
&  \left(  \cos\kappa a_{12}+\cos\kappa a_{24}\right)  \left(  \cos\kappa
a_{34}+\cos\kappa a_{13}\right)  \leq\\
&  \text{ }\sin\kappa a_{12}\sin\kappa a_{34}\left(  1+\cos\kappa
a_{14}\right)
\end{align*}
and $a_{12},a_{13},a_{14},a_{23},a_{24},a_{34}\leq\pi/\left(  2\kappa\right)
$. For $\mathcal{K}^{-}\left(  \kappa^{2}\right)  $, multiply the left-hand
side of the above inequality by $\left(  -1\right)  $. In a similar way, we
define $\mathcal{K}^{+}\left(  -\kappa^{2}\right)  $ as the set of all
matrices $A=\left(  a_{ij}\right)  \in M_{4}$ such that
\[%
\begin{array}
[c]{c}%
\left(  \cosh\kappa a_{12}+\cosh\kappa a_{24}\right)  \left(  \cosh\kappa
a_{34}+\cosh\kappa a_{13}\right)  -\\
\left(  \cosh\kappa a_{23}+\cosh\kappa a_{12}\cosh\kappa a_{34}\right)
\left(  1+\cosh\kappa a_{14}\right)  \leq\\
\text{ }\sinh\kappa a_{12}\sinh\kappa a_{34}\left(  1+\cosh\kappa
a_{14}\right)
\end{array}
\]
and for $\mathcal{K}^{-}\left(  -\kappa^{2}\right)  $, multiply the left-hand
side of the above inequality by $\left(  -1\right)  $.

\begin{theorem}
\label{ThCurvPr}Let $\kappa>0$ and $K=\kappa^{2}$ if $K>0$ and $K=-\kappa^{2}
$ if $K<0$. Then\newline(i) $\left(  \mathcal{X},\rho\right)  \in
\mathfrak{M}_{G}$ (respectively $\left(  \mathcal{X},\rho\right)
\in\mathfrak{M}_{S}$) is in the global $\mathcal{K}^{\pm}\left(  \kappa
^{2}\right)  $-curvature class if and only if $\left(  \mathcal{X}%
,\rho\right)  $ is an $\Re_{K}$ domain (respectively complete $\Re_{K}$
domain) of diameter not greater than $\pi/\left(  2\kappa\right)  $%
.\newline(ii) $\left(  \mathcal{X},\rho\right)  \in\mathfrak{M}_{G}$
(respectively $\left(  \mathcal{X},\rho\right)  \in\mathfrak{M}_{S}$) is in
the global $\mathcal{K}^{\pm}\left(  -\kappa^{2}\right)  $-curvature class if
and only if $\left(  \mathcal{X},\rho\right)  $ is an $\Re_{K}$ domain
(respectively complete $\Re_{K}$ domain). \newline
\end{theorem}

\begin{remark}
In particular, $\left(  \mathcal{X},\rho\right)  \in\mathfrak{M}_{G}$ is in
the local $\mathcal{K}^{\pm}\left(  \pm\kappa^{2}\right)  $-curvature class if
and only if $\left(  \mathcal{X},\rho\right)  $ is an Aleksandrov space of
curvature $\leq K$ where $K=\pm\kappa^{2}$.
\end{remark}

\begin{remark}
For an alternative proof of one of our main theorems \cite[Theorem
6]{BergNik2008} solving Gromov's curvature problem in the context of $\Re_{0}%
$-domains, see \cite{Sa2009}.
\end{remark}

In Sec. \ref{euler_ineq}, we generalize the familiar Euler's equality
\cite[Corollary 4]{Eul750} to non-zero $K$. Hence, we can extend the
quadrilateral inequality condition (also known as Enflo's $2$-roundness
condition \cite{Enf1969}) to the case of non-zero $K$. \newline

\textit{The }$K$\textit{-quadrilateral (or }$K$\textit{-Euler) inequality
condition: for every quadruple of points }$\left\{  A,\text{ }B,\text{
}C,\text{ }D\right\}  $ \textit{in a metric space} $\left(  \mathcal{M}%
,\mathcal{\rho}\right)  $, \textit{\ }\newline\textit{(i) if }$K>0$\textit{,
then}%
\begin{align*}
&  \cos\kappa\rho\left(  A,B\right)  +\cos\kappa\rho\left(  B,C\right)
+\cos\kappa\rho\left(  C,D\right)  +\cos\kappa\rho\left(  D,A\right) \\
&  \leq4\cos\kappa\frac{\rho\left(  A,C\right)  }{2}\cos\kappa\frac
{\rho\left(  B,D\right)  }{2},
\end{align*}
\textit{(ii) if }$K<0$\textit{, then}%
\begin{align*}
&  \cosh\kappa\rho\left(  A,B\right)  +\cosh\kappa\rho\left(  B,C\right)
+\cosh\kappa\rho\left(  C,D\right)  +\cosh\kappa\rho\left(  D,A\right) \\
&  \geq4\cosh\kappa\frac{\rho\left(  A,C\right)  }{2}\cosh\kappa\frac
{\rho\left(  B,D\right)  }{2}.
\end{align*}
According to Theorem 6 in \cite{BergNik2008}, a geodesically connected metric
space is an $\Re_{0}$ domain if and only if it satisfies the $0$-quadrilateral
inequality condition. In Sec. \ref{euler_ineq}, we prove that the
$K$-quadrilateral inequality condition holds in an $\Re_{K}$ domain for
non-zero $K$. We do not know if the converse is true.

In \cite{LafPrass2006}, Lafont and Prassidis established the $0$-quadrilateral
inequality in $\Re_{0}$ domains. In \cite{FLS2007} (also, see the correction
in \cite{FLS2007er}) Foertsch, Lytchak and Schroeder considered a weaker
Ptolemaic condition and showed that while each $\Re_{0}$ domain is Ptolemaic,
the converse may not be true.

Sec. \ref{Al_upper_curv_cond} is a short review of Aleksandrov spaces of
curvature bounded above. In Sec. \ref{cosqk_in_SK}, we prove that $\left\vert
\operatorname{cosq}_{K}\right\vert \leq1$ in $K$-space. Sec. \ref{Sec_RK}
presents the proof of $\left\vert \operatorname{cosq}_{K}\right\vert \leq1$ in
an $\Re_{K}$ domain of diameter not greater than $\pi/2\sqrt{K}$ if $K>0$. We
show that, in contrast to $\mathbb{S}_{K}^{3}$, the diameter restriction
cannot be dropped for an $\Re_{K}$ domain. In Sec. \ref{Counter_Examples}, we
present counterexamples showing that in a non-geodesically connected metric
space the upper four point $\operatorname{cosq}_{K}$ condition need not be
equivalent to the lower four point $\operatorname{cosq}_{K}$ condition. Sec.
\ref{Proof_Main_Th} contains the proof of our main result--Theorem
\ref{MainTh}. In this section, we assume that $\left(  \mathcal{M}%
,\rho\right)  $ is a geodesically connected metric space (of diameter not
greater than $\pi/2\sqrt{K}$ if $K>0$) satisfying the one-sided four point
$\operatorname{cosq}_{K}$ condition. In Sec. \ref{Cont_geod}, we prove that in
$\left(  \mathcal{M},\rho\right)  $ shortests depend continuously on their end
points; in particular, any pair of points can be joined by a unique shortest.
Hence, by Theorem 9 in \cite[\S \ 3]{A1957a}, the global angle comparison in
$\left(  \mathcal{M},\rho\right)  $ will follow from the local angle
comparison, i.e., locally, each vertex angle of a geodesic triangle
$\mathcal{T}$ is not greater than the corresponding angle of the isometric
copy of $\mathcal{T}$ $\ $in the $K$-plane. In Section
\ref{Sec_cross_diag_est}, we derive the main auxiliary estimate--the
cross-diagonal estimate. In Section \ref{Growth_Est}, the cross-diagonal
estimate lemma is used to derive our major estimate of Sec.
\ref{Proof_Main_Th}--the growth estimate lemma. In Sec. \ref{Prop_Angles}, we
show that the growth estimate lemma implies that in $\left(  \mathcal{M}%
,\rho\right)  $, between any pair of shortests starting at a common point $A$,
the proportional angle exists, that is, the limit of $\measuredangle_{K}%
X_{t}AY_{t}$ as $t\rightarrow0+$ exists if $\rho\left(  X_{t},A\right)
/\rho\left(  Y_{t},A\right)  =\operatorname{const}$ (for the notation, see
Sec. \ref{Al_upper_curv_cond} and Fig. \ref{fig12}). In Sec. \ref{Exist_Angle}%
, following the method of our proof of Proposition 20 in \cite{BergNik2008},
we derive from existence of proportional angles and growth estimate lemma that
in $\left(  \mathcal{M},\rho\right)  $, between any pair of shortests
emanating from a common point, Aleksandrov's angle exists. Existence of
Aleksandrov's angle and growth estimate lemma enables us to prove the local
angle comparison and thereby the global angle comparison (Sec.
\ref{Angle_Comparison}). In Sec. \ref{SecThExtr}, we consider an extremal case
when $\left\vert \operatorname{cosq}_{K}\right\vert =1$. In Sec.
\ref{SecThsemimetr}, we extend our main result to complete weakly convex
semimetric spaces satisfying the one-sided four point $\operatorname{cosq}%
_{K}$ condition. In Sec. \ref{euler_ineq}, we derive $K$-Euler's inequality
for $\Re_{K}$ domains and discuss the extremal case of equality in $K$-Euler's
inequality. In Sec. \ref{Remarks}, we show that for an individual quadruple in
a metric space, the one-sided four point $\operatorname{cosq}_{K}$ conditions
are weaker than previously introduced curvature conditions.

\section{Aleksandrov's upper curvature condition\label{Al_upper_curv_cond}}

In this section, we recall some basic definitions of Aleksandrov geometry.

Let $\left(  \mathcal{M},\rho\right)  $ be a metric space and $\mathcal{L}$ be
a curve in $\mathcal{M}$. We denote by $\ell_{\rho}\left(  \mathcal{L}\right)
$ the length of $\mathcal{L}$ in the metric $\rho$. A rectifiable curve
$\mathcal{L}$ joining $P$ to $Q$ is called a \emph{shortest}, or \emph{minimal
geodesic} (joining\emph{\ }$P$ to $Q$) if $\rho\left(  P,Q\right)  =\ell
_{\rho}\left(  \mathcal{L}\right)  $. If $\mathcal{L}$ is a shortest joining
$P$ to $Q$, then often we denote the shortest $\mathcal{L}$ by $\mathcal{PQ}$
if there is no possible ambiguity, and the distance between its end points
(or, in general, between a pair of points in $\mathcal{M}$) $P $ and $Q$ by
$PQ$. A subset $\mathcal{U}$ of a metric space$\mathcal{\,}$\ is\emph{\ }said
to be\emph{\ convex} if every pair of points $P,Q\in\mathcal{U}$ can be joined
by a shortest and all shortests joining $P$ to $Q$ are contained in
$\mathcal{U}$.

A configuration consisting of three distinct points $A,B,C\in\mathcal{M}$
(\emph{vertices}) and three shortests $\mathcal{AB},$ $\mathcal{BC}$ and
$\mathcal{AC}$ (\emph{sides}) is called a \emph{(geodesic) triangle}
$\mathcal{T=}$ $ABC$. The \emph{perimeter} $p\left(  \mathcal{T}\right)  $ of
a triangle $\mathcal{T}=ABC$ (or, in general, of a triple of points
$\mathcal{T=}\left\{  A,B,C\right\}  $ in $\mathcal{M}$) is the sum
$AB+BC+AC$. The \emph{isometric copy} in the $K$-plane of the triangle
$\mathcal{T}$ is the triangle $\mathcal{T}^{K}=A^{K}B^{K}C^{K}$ in
$\mathbb{S}_{K}$ having the same side lengths as $\mathcal{T}$: $AB=A^{K}%
B^{K},$ $AC=A^{K}C^{K}$ and $BC=B^{K}C^{K}$ (if $K>0$ we require that
$p\left(  \mathcal{T}\right)  <2\pi/\sqrt{K}$). We let $\measuredangle_{K}BAC$
denote the angle $\measuredangle B^{K}A^{K}C^{K}$. The area $\sigma\left(
ABC\right)  $ of the triangle $ABC$ is the area of the euclidean triangle
$A^{0}B^{0}C^{0}$.

Let $\mathcal{L}$ and $\mathcal{N}$ be two shortest arcs with a common
starting point $O$ in a metric space $(\mathcal{M},\rho)$. Let $X\in
\mathcal{L}\backslash\left\{  O\right\}  $ and $Y\in\mathcal{N}\backslash
\left\{  O\right\}  $. Set $x=OX$, $y=OY$ and $\measuredangle_{K}\left(
x,y\right)  =\measuredangle_{K}XOY$. The \emph{upper }and\emph{\ lower angles}
between the curves $\mathcal{L}$ and $\mathcal{N}$ are defined by%

\[
\overline{\measuredangle}(\mathcal{L},\mathcal{N})=\underset{x\rightarrow
0+,y\rightarrow0+}{\overline{\lim}}\measuredangle_{K}\left(  x,y\right)
\text{ and }\underline{\measuredangle}(\mathcal{L},\mathcal{N})=\underset
{x\rightarrow0+,y\rightarrow0+}{\underline{\lim}}\measuredangle_{K}\left(
x,y\right)  .
\]
It is known that the above definitions do not depend on $K$. We say that the
angle $\measuredangle\left(  \mathcal{L},\mathcal{N}\right)  $ between
$\mathcal{L}$ and $\mathcal{N}$ exists if $\overline{\measuredangle
}(\mathcal{L},\mathcal{N})=\underline{\measuredangle}(\mathcal{L},\mathcal{N})
$.

The (upper) $K$\emph{-excess} $\delta_{K}(\mathcal{T})$ of the triangle
$\mathcal{T}$ is defined by
\[
\delta_{K}(\mathcal{T})=(\overline{\measuredangle}ABC+\overline{\measuredangle
}ACB+\overline{\measuredangle}BAC)-(\measuredangle_{K}ABC+\measuredangle
_{K}ACB+\measuredangle_{K}BAC).
\]

An $\Re_{K}$ $\ $\emph{domain }(otherwise known as a $\operatorname{CAT}%
\left(  K\right)  $ space) is a metric space with the following properties:

(i) $\Re_{K}$ is convex (that is, $\Re_{K}$ is geodesically connected).

(ii) If $K>0$, then the perimeter of every triangle in $\Re_{K}$ is less than
$2\pi/\sqrt{K}$.

(iii) Each triangle $\mathcal{T}$ in $\Re_{K}$ has non-positive $K$-excess
$\delta_{K}(\mathcal{T})$.

We remark that by (ii), $\operatorname{diam}\left(  \Re_{K}\right)  <\pi
/\sqrt{K}$ when $K>0$.

Another name for an $\Re_{K}$ domain is a $\operatorname{CAT}\left(  K\right)
$ space; we will use Aleksandrov's original notation (see, \cite{A1951} and
\cite{A1957a}). A metric space $\left(  \mathcal{M},\rho\right)  $ is
a\emph{\ space} \emph{of curvature }$\leq K$ in the sense of Aleksandrov if
each point of $\mathcal{M}$ is contained in some neighborhood that is an
$\Re_{K} $ domain. For more information on Aleksandrov spaces of curvature
$\leq K$, see \cite{A1951}, \cite{A1957a}, \cite{BerNik1993} and
\cite{BrH1990}.

We will find useful the following theorem of Reshetnyak \cite{Resh1968}.

Let $\mathcal{L}$ be a closed rectifiable curve in a metric space $\left(
\mathcal{M},\rho\right)  $ such that $\ell_{\rho}\left(  \mathcal{L}\right)
<2\pi/\sqrt{K}$ if $K>0$. Let $\mathcal{V}$ be a convex domain in
$\mathbb{S}_{K}$ with the bounding curve $\mathcal{N}$. We say that
$\mathcal{V}$ \emph{majorizes} the curve $\mathcal{L}$ if there is a
non-expanding mapping of the domain $\mathcal{V}$ into $\mathcal{M}$ that maps
$\mathcal{N}$ onto $\mathcal{L}$ and preserves arc length. The domain
$\mathcal{V}$ is called the \emph{majorant} for $\mathcal{L}$.

\textbf{Reshetnyak's majorization theorem:} \textit{In an }$\Re_{K}%
$\textit{\ domain}, \textit{for every rectifiable closed curve }$\mathcal{L}%
$\textit{\ (whose length is less than }$2\pi/\sqrt{K}$\textit{\ when }%
$K>0$)\textit{, there is a convex domain in }$\mathbb{S}_{K}$\textit{\ that
majorizes }$\mathcal{L}$\textit{.}

Let $\left(  A_{1},A_{2},...,A_{n}\right)  $ be an $n$-tuple of distinct
points in $\left(  \mathcal{M},\rho\right)  $. Suppose that for every
$j\in\left\{  1,2,...,n-1\right\}  $, the points $A_{j}$ and $A_{j+1}$ can be
joined by a shortest $\mathcal{L}_{j}=\mathcal{A}_{j}\mathcal{A}_{j+1}$. Then
we call the curve $\mathcal{L=A}_{1}\mathcal{A}_{2}...\mathcal{A}_{n}$ formed
by the consecutive shortests $\mathcal{L}_{j}$, a \emph{polygonal curve} (with
vertices at $A_{1},A_{2},...,A_{n}$ in $\mathcal{M}$). It is not difficult to
see that in Reshetnyak's theorem if $\mathcal{L=A}_{1}\mathcal{A}%
_{2}...\mathcal{A}_{n}\mathcal{A}_{1}$ is a closed polygonal curve, then
$\mathcal{N}$ is also a closed polygonal curve $\mathcal{A}_{1}^{\prime
}\mathcal{A}_{2}^{\prime}...\mathcal{A}_{n}^{\prime}\mathcal{A}_{1}^{\prime}$
in $\mathbb{S}_{K}$. In our notation, we always assume that the vertices of
$\mathcal{N}$ are labeled so that $A_{j}A_{j+1}=A_{j}^{\prime}A_{j+1}^{\prime
}$ for every $j=1,2,..,n$, where $A_{n+1}=A_{1}$ and $A_{n+1}^{\prime}%
=A_{1}^{\prime}$.

If $\mathcal{L}$ is a polygonal curve $\mathcal{A}_{1}\mathcal{A}%
_{2}...\mathcal{A}_{n}$ of length $l$ in a metric space, then the \emph{arc
length parametrization of }$\mathcal{L}$\emph{\ relative to} $A_{1}$ is an arc
length parametrization of $\mathcal{L}$, $\mathfrak{g}_{al}=\mathfrak{g}%
_{al,\mathcal{L}}:\left[  0,l\right]  \rightarrow\mathcal{M}$, such that the
length of the arc of $\mathcal{L}$ with the end points at $A_{1}$ and
$\mathfrak{g}_{al}\left(  s\right)  $ is equal to $s\in\left[  0,l\right]  $.
The \emph{reduced parametrization of }$\mathcal{L}$\emph{\ relative to }$A$ is
the mapping $\mathfrak{g}_{r}=\mathfrak{g}_{r,\mathcal{L}}:\left[  0,1\right]
\rightarrow\mathcal{M~}$given by $\mathfrak{g}_{r}\left(  t\right)  =$
$\mathfrak{g}_{al}\left(  tl\right)  $ for every $t\in\left[  0,1\right]  $.
If $l_{0}>0$, then the $l_{0}$-arc length proportional parametrization of
$\mathcal{L}$ is the mapping $\mathfrak{g}_{l_{0},pr}=\mathfrak{g}%
_{l_{0},pr,\mathcal{L}}:\left[  0,l_{0}\right]  \rightarrow\mathcal{M}$ given
by $\mathfrak{g}_{l_{0},pr}\left(  u\right)  =\mathfrak{g}_{al}\left(
ul/l_{0}\right)  $.

Let $\left(  \mathcal{M},\rho\right)  $ be a geodesically connected metric
space and $\mathcal{F}\subseteq\mathcal{M}$ be a non-empty set. For a pair of
points $P,Q\in\left(  \mathcal{M},\rho\right)  $, we let $\mathcal{G}\left[
P,Q\right]  $ denote the set of points each of which belongs to a shortest
joining the points $P$ and $Q$. We define $\mathcal{G}\left[  \mathcal{F}%
\right]  $ by $\mathcal{G}\left[  \mathcal{F}\right]  =\cup_{P,Q\in
\mathcal{F}}\mathcal{G}\left[  P,Q\right]  $. Next, denote $\mathcal{F\,}$ by
$\mathcal{G}^{0}\left[  \mathcal{F}\right]  $ and $\underset{n\text{ times}%
}{\underbrace{\mathcal{G}\left[  \mathcal{G}\left[  ...\mathcal{G}\left[
\mathcal{F}\right]  \right]  \right]  }}$ by $\mathcal{G}^{n}\left[
\mathcal{F}\right]  $. Then the \emph{geodesic} \emph{convex hull} of
$\mathcal{F}$ is defined as $\mathcal{GC}\left[  \mathcal{F}\right]
=\cup_{n=0}^{\infty}\mathcal{G}^{n}\left[  \mathcal{F}\right]  $.

\section{$K$-quadrilateral cosine in $K$-space\label{cosqk_in_SK}}

In this section, we prove that $\left\vert \operatorname{cosq}_{K}\right\vert
\leq1$ in $\mathbb{S}_{K}^{3}$.

Let $K\neq0$. Let $\left\{  A,B,P,Q\right\}  $ be a quadruple of distinct
points in $\mathbb{S}_{K}^{3}$. Let $O$ be the midpoint of the shortest arc
$\mathcal{AB}$. If $PO<\pi/\left(  2\sqrt{K}\right)  $ when $K>0$, we can use
the following constructive interpretation of $\operatorname{cosq}_{K}$ in
$\mathbb{S}_{K}^{3}$. Indeed, let $P^{\prime}$ be the point symmetric to the
point $P$ relative to $O$, that is, $O$ is the midpoint of the shortest arc
$\mathcal{PP}^{\prime}$, as illustrated in Fig. \ref{fig2}. Then $\xi=\exp
_{A}^{-1}\left(  P\right)  $ is (Levi-Civita) parallel along $\mathcal{AB}$ to
the vector $\xi^{\prime\prime}=-\xi^{\prime}$, where $\xi^{\prime}=\exp
_{B}^{-1}\left(  P^{\prime}\right)  $.
%TCIMACRO{\FRAME{ftphFU}{2.0202in}{1.7398in}{0pt}{\Qcb{$\operatorname{cosq}%
%_{K}$ in $\QTR{Bbb}{S}_{K}^{3}$}}{\Qlb{fig2}}{fig2.eps}%
%{\special{ language "Scientific Word";  type "GRAPHIC";
%maintain-aspect-ratio TRUE;  display "PICT";  valid_file "F";
%width 2.0202in;  height 1.7398in;  depth 0pt;  original-width 2.1364in;
%original-height 1.8357in;  cropleft "0";  croptop "1";  cropright "1";
%cropbottom "0";  filename 'Fig2.eps';file-properties "XNPEU";}} }%
%BeginExpansion
\begin{figure}
[pth]
\begin{center}
\includegraphics[
natheight=1.835700in,
natwidth=2.136400in,
height=1.7398in,
width=2.0202in
]%
{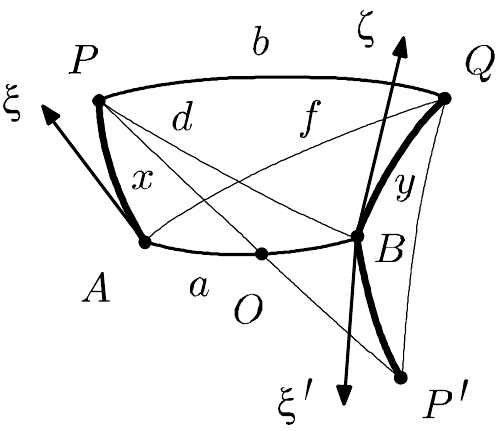}%
\caption{$\operatorname{cosq}_{K}$ in $\mathbb{S}_{K}^{3}$}%
\label{fig2}%
\end{center}
\end{figure}
%EndExpansion
Let $\zeta=\exp_{B}^{-1}\left(  Q\right)  $. In \cite[Lemma 3.1]%
{BergNik2007a}, we showed that $\operatorname{cosq}_{K}\left(  \overrightarrow
{AP},\overrightarrow{BQ}\right)  =-\cos\measuredangle P^{\prime}%
BQ=\cos\measuredangle\left(  \zeta,\xi^{\prime\prime}\right)  $. Hence, for
the $K$-quadrilateral cosine in $\mathbb{S}_{K}^{3}$, we always have
\[
\left\vert \operatorname{cosq}_{K}\left(  \overrightarrow{AP},\overrightarrow
{BQ}\right)  \right\vert \leq1
\]
as long as $PO<\pi/\left(  2\sqrt{K}\right)  $ when $K>0$.

Next, we show that the restriction $PO<\pi/\left(  2\sqrt{K}\right)  $ for
positive $K$ can be dropped for $\mathbb{S}_{K}^{3}$ itself. We begin with
\ the following simple corollary of the spherical cosine formula.

\begin{lemma}
\label{LemmaBruhat}Let $K>0$ and $\mathcal{T}=ABC$ be a non-degenerate
triangle in $\mathbb{S}_{K}$. Let $M\in\mathcal{AB}\backslash\left\{
A,B\right\}  $. Set $a=BC,$ $b=AC,$ $c=AB,$ $l=MC$ and $t=AM/c$, as shown in
the Fig. \ref{fig3}.%
%TCIMACRO{\FRAME{ftpFU}{1.3828in}{1.3828in}{0pt}{\Qcb{Sketch for Lemma
%\ref{LemmaBruhat}}}{\Qlb{fig3}}{fig3.eps}%
%{\special{ language "Scientific Word";  type "GRAPHIC";
%maintain-aspect-ratio TRUE;  display "PICT";  valid_file "F";
%width 1.3828in;  height 1.3828in;  depth 0pt;  original-width 1.4529in;
%original-height 1.4529in;  cropleft "0";  croptop "1";  cropright "1";
%cropbottom "0";  filename 'Fig3.eps';file-properties "XNPEU";}} }%
%BeginExpansion
\begin{figure}
[pt]
\begin{center}
\includegraphics[
natheight=1.452900in,
natwidth=1.452900in,
height=1.3828in,
width=1.3828in
]%
{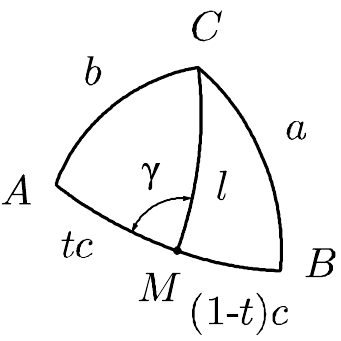}%
\caption{Sketch for Lemma \ref{LemmaBruhat}}%
\label{fig3}%
\end{center}
\end{figure}
%EndExpansion
Then
\[
\cos\kappa l=\frac{\cos\kappa a\sin\kappa tc+\cos\kappa b\sin\kappa\left(
1-t\right)  c}{\sin\kappa c}.
\]
In particular, if $M$ is the midpoint of the shortest $\mathcal{AB}$, we
obtain a familiar spherical Bruhat-Tits equality:
\[
\cos\kappa l=\frac{\cos\kappa a+\cos\kappa b}{2\cos\frac{\kappa c}{2}}%
\]
(for $K=0$, see Bruhat-Tits inequality in \cite{BrT1972}, Lemma 3.2.1).
\end{lemma}

By $K$-concavity in $\Re_{K}$ \cite[\S 3, Theorem 2]{A1957a}, we also have the following

\begin{corollary}
Let $K>0$ and $\mathcal{T}=ABC$ be a non-degenerate triangle in $\Re_{K}$ and
$M\in\mathcal{AB}\backslash\left\{  A,B\right\}  $. Set $a=BC,$ $b=AC,$
$c=AB,$ $l=MC$ and $t=AM/c$. Then%
\[
\cos\kappa l\geq\frac{\cos\kappa a\sin\kappa tc+\cos\kappa b\sin\kappa\left(
1-t\right)  c}{\sin\kappa c}.
\]

\end{corollary}

\begin{corollary}
\label{Cor_Pi_by_2}Let $K>0$ and $\mathcal{T}=ABC$ be a non-degenerate
triangle in $\Re_{K}$ and $M\in\mathcal{AB}\backslash\left\{  A,B\right\}  $.
Let $AC,BC\leq\pi/\left(  2\kappa\right)  $. Then $CM\leq\pi/\left(
2\kappa\right)  $. In addition, if either $AC$ or $BC$ is less than
$\pi/\left(  2\kappa\right)  $, then $CM<\pi/\left(  2\kappa\right)  $.
\end{corollary}

Finally, we show that $\operatorname{cosq}_{K}$ remains the same in the
half-sphere after cutting the lengths of bound vectors in half.

\begin{lemma}
\label{Lemma_In_Half}Let $K>0$ and $\overrightarrow{AP},$ $\overrightarrow
{BQ}$ be a pair of non-zero bound vectors in $\mathbb{S}_{K}^{3}$. Let $M_{1}$
and $M_{2}$ be the midpoints of the shortests $\mathcal{AP}$ and
$\mathcal{BQ}$, respectively. Then%
\[
\operatorname{cosq}_{K}\left(  \overrightarrow{AP},\overrightarrow{BQ}\right)
=\operatorname{cosq}_{K}\left(  \overrightarrow{AM_{1}},\overrightarrow
{BM_{2}}\right)  .
\]

\end{lemma}

\begin{proof}
We have:%
\begin{align*}
\operatorname{cosq}_{K}\left(  \overrightarrow{AM_{1}},\overrightarrow{BM_{2}%
}\right)   &  =\frac{\cos\kappa b^{\prime}+\cos\frac{\kappa x}{2}\cos
\frac{\kappa y}{2}}{\sin\frac{\kappa x}{2}\sin\frac{\kappa y}{2}}-\\
&  \frac{\left(  \cos\frac{\kappa x}{2}+\cos\kappa d^{\prime}\right)  \left(
\cos\frac{\kappa y}{2}+\cos\kappa f^{\prime}\right)  }{\left(  1+\cos\kappa
a\right)  \sin\frac{\kappa x}{2}\sin\frac{\kappa y}{2}},
\end{align*}
where the notation is given in Fig. \ref{fig4}.%
%TCIMACRO{\FRAME{ftphFU}{2.5137in}{1.6392in}{0pt}{\Qcb{Sketch for Lemma
%\ref{Lemma_In_Half}}}{\Qlb{fig4}}{fig4.eps}%
%{\special{ language "Scientific Word";  type "GRAPHIC";
%maintain-aspect-ratio TRUE;  display "PICT";  valid_file "F";
%width 2.5137in;  height 1.6392in;  depth 0pt;  original-width 2.6659in;
%original-height 1.7269in;  cropleft "0";  croptop "1";  cropright "1";
%cropbottom "0";  filename 'Fig4.eps';file-properties "XNPEU";}} }%
%BeginExpansion
\begin{figure}
[pth]
\begin{center}
\includegraphics[
natheight=1.726900in,
natwidth=2.665900in,
height=1.6392in,
width=2.5137in
]%
{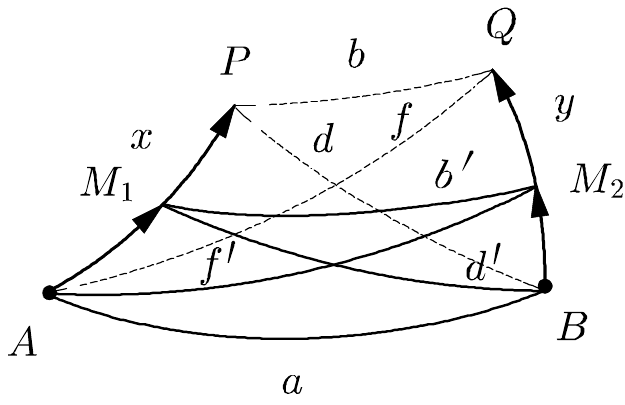}%
\caption{Sketch for Lemma \ref{Lemma_In_Half}}%
\label{fig4}%
\end{center}
\end{figure}
%EndExpansion
By the Bruhat-Tits equality (Lemma \ref{LemmaBruhat}),%
\begin{align*}
\cos\kappa d^{\prime}  &  =\frac{\cos\kappa a+\cos\kappa d}{2\cos\frac{\kappa
x}{2}}\text{ (triangle }ABP\text{), }\\
\cos\kappa f^{\prime}  &  =\frac{\cos\kappa a+\cos\kappa f}{2\cos\frac{\kappa
y}{2}}\text{ (triangle }ABQ\text{),}\\
\cos\kappa g  &  =\frac{\cos\kappa b+\cos\kappa d}{2\cos\frac{\kappa y}{2}%
}\text{, }g=PM_{2}\text{ (triangle }PQB\text{),}\\
\cos\kappa b^{\prime}  &  =\frac{\cos\kappa g+\cos\kappa f^{\prime}}%
{2\cos\frac{\kappa x}{2}}\text{ (triangle }APM_{2}\text{),}%
\end{align*}
whence $\cos\kappa b^{\prime}=\left(  \cos\kappa a+\cos\kappa b+\cos\kappa
d+\cos\kappa f\right)  /\left(  4\cos\frac{\kappa x}{2}\cos\frac{\kappa y}%
{2}\right)  $. Hence,
\begin{align*}
&  \operatorname{cosq}_{K}\left(  \overrightarrow{AM_{1}},\overrightarrow
{BM_{2}}\right) \\
&  =\left(  1+\cos\kappa a\right)  \left[  \cos\kappa a+\cos\kappa
b+\cos\kappa d+\cos\kappa f\right.  +\\
&  \left.  4\cos^{2}\frac{\kappa x}{2}\cos^{2}\frac{\kappa y}{2}\right]
-\left(  2\cos^{2}\frac{\kappa x}{2}+\cos\kappa a+\cos\kappa d\right)
\times\\
&  \left(  2\cos^{2}\frac{\kappa y}{2}+\cos\kappa a+\cos\kappa f\right)
\left[  \left(  1+\cos\kappa a\right)  \sin\kappa x\sin\kappa y\right]
^{-1}\\
&  =\left[  \left(  1+\cos\kappa a\right)  \left(  \cos\kappa a+\cos\kappa
b+\cos\kappa d+\cos\kappa f\right.  +\right. \\
&  \left.  1+\cos\kappa x+\cos\kappa y+\cos\kappa x\cos\kappa y\right)
-\left(  1+\cos\kappa x+\cos\kappa a+\cos\kappa d\right)  \times\\
&  \left.  \left(  1+\cos\kappa y+\cos\kappa a+\cos\kappa f\right)  \right]
\left/  \left[  \left(  1+\cos\kappa a\right)  \sin\kappa x\sin\kappa
y\right]  \right.  .
\end{align*}
After elementary but tedious simplifications of the last expression, we get:
\begin{align*}
&  \operatorname{cosq}_{K}\left(  \overrightarrow{AM_{1}},\overrightarrow
{BM_{2}}\right) \\
&  =\left[  \left(  1+\cos\kappa a\right)  \sin\kappa x\sin\kappa y\right]
^{-1}\left(  \cos\kappa b+\cos\kappa a\cos\kappa b+\cos\kappa a\cos\kappa
x\cos\kappa y\right. \\
&  \left.  -\cos\kappa x\cos\kappa f-\cos\kappa y\cos\kappa d-\cos\kappa
d\cos\kappa f\right) \\
&  =\operatorname{cosq}_{K}\left(  \overrightarrow{AP},\overrightarrow
{BQ}\right)  \text{,}%
\end{align*}
as needed.
\end{proof}

Let $K>0$. Recall that $\operatorname{diam}\left(  \Re_{K}\right)  <\pi
/\sqrt{K}$. By Lemma \ref{Lemma_In_Half}, there is no restriction in assuming
that $AP$ and $BQ$ are as small as we wish. Hence, without loss of generality,
we can assume that $PO<\pi/\left(  2\sqrt{K}\right)  $ (see Fig. \ref{fig2}).
So, we get the following

\begin{corollary}
\label{CorcosqKbound}Let $K\neq0$. Then for every pair of non-zero bound
vectors $\overrightarrow{AP}$ and $\overrightarrow{BQ}$ in $\mathbb{S}_{K}%
^{3}$, $\left\vert \operatorname{cosq}_{K}\left(  \overrightarrow
{AP},\overrightarrow{BQ}\right)  \right\vert \leq1$.
\end{corollary}

\section{$K$-quadrilateral cosine in an $\Re_{K}$ domain\label{Sec_RK}}

The main goal of this section is to show that $\left\vert \operatorname{cosq}%
_{K}\right\vert \leq1$ in an $\Re_{K}$ domain of diameter not greater than
$\pi/2\sqrt{K}$ if $K>0$. In addition, for $K>0$ we present examples of
$\Re_{K}$ domains of diameter greater than $\pi/2\sqrt{K}$ and arbitrarily
close to $\pi/2\sqrt{K}$ for which $\left\vert \operatorname{cosq}%
_{K}\right\vert \leq1$ does not hold.

The following theorem is a minor generalization of Theorem 4.2 in
\cite{BergNik2007a}.

\begin{theorem}
\label{ThRKBound}Let $K\neq0$ and let $\mathfrak{Q}\mathcal{=}\left\{
A,P,B,Q\right\}  $ be a quadruple of points in an $\Re_{K}$ domain such that
$A\neq P,$ $B\neq Q$ and $\operatorname{diam}\left(  \mathfrak{Q}\right)
\leq\pi/\left(  2\sqrt{K}\right)  $ if $K>0$. Then
\[
\left\vert \operatorname{cosq}_{K}\left(  \overrightarrow{AP},\overrightarrow
{BQ}\right)  \right\vert \leq1.
\]

\end{theorem}

\begin{proof}
It is sufficient to consider the case of positive $K$. If $\operatorname{diam}%
\left(  \mathfrak{Q}\right)  <\pi/\left(  2\sqrt{K}\right)  $, then by
\cite[Theorem 4.2]{BergNik2007a}, $\operatorname{cosq}_{K}\left(
\overrightarrow{AP},\overrightarrow{BQ}\right)  \geq-1$. For the reader's
convenience, we include some omitted details in \cite{BergNik2007a} of the
proof of the inequality $\operatorname{cosq}_{K}\left(  \overrightarrow
{AP},\overrightarrow{BQ}\right)  \leq1$. Consider the closed polygonal curve
$\mathcal{L=APQBA}$, as shown in Fig. \ref{fig1}. We will follow the part of
the proof of Reshetnyak's Lemma 2 in \cite{Resh1968} corresponding to the case
of $K$-fans consisting of two triangles in $\mathbb{S}_{K}$ (a special case of
Reshetnyak's majorization theorem). Namely, under the hypothesis of Theorem
\ref{ThRKBound}, in addition to the existence of a convex domain
$\mathcal{V\subseteq}\mathbb{S}_{K}$ majorizing the polygonal curve
$\mathcal{L}$, Reshetnyak's proof also implies that the domain $\mathcal{V}$
can be selected so that%
\begin{equation}
d=PB\leq d^{\prime}=P^{\prime}B^{\prime}<\pi/\left(  2\sqrt{K}\right)  ,\text{
}f=AQ\leq f^{\prime}=A^{\prime}Q^{\prime}<\pi/\sqrt{K} \label{diag_cond}%
\end{equation}
where $\mathcal{L}^{\prime}\mathcal{=A}^{\prime}\mathcal{P}^{\prime
}\mathcal{Q}^{\prime}\mathcal{B}^{\prime}\mathcal{A}^{\prime}$ is the bounding
curve of $\mathcal{V}$. Indeed, as shown in the proof of Lemma 2 in
\cite{Resh1968}, there is a quadrangular domain $\mathcal{F}$ in
$\mathbb{S}_{K}$ bounded by a quadrangle $\widetilde{\mathcal{L}}^{\prime
}=\widetilde{A}^{\prime}\widetilde{P}^{\prime}\widetilde{Q}^{\prime}%
\widetilde{B}^{\prime}$ such that
\[
AP=\widetilde{A}^{\prime}\widetilde{P}^{\prime},\text{ }AB=\widetilde
{A}^{\prime}\widetilde{B}^{\prime},\text{ }PB=\widetilde{P}^{\prime}%
\widetilde{B}^{\prime}\text{ and }PQ=\widetilde{P}^{\prime}\widetilde
{Q}^{\prime},\text{ }BQ=\widetilde{B}^{\prime}\widetilde{Q}^{\prime}.
\]
If $\mathcal{F}$ is convex, then we put $\mathcal{F=V}$ and we have
$d=\widetilde{d}^{\prime}=\widetilde{P}^{\prime}\widetilde{B}^{\prime}%
<\pi/\left(  2\sqrt{K}\right)  $ and (as shown in Reshetnyak's proof)
$f=AQ\leq\widetilde{f}^{\prime}=\widetilde{A}^{\prime}\widetilde{Q}^{\prime
}<\pi/\sqrt{K}$. Now suppose that the quadrangular domain $\mathcal{F}$ is not
convex. Then either the angle of the quadrangle $\widetilde{\mathcal{L}%
}^{\prime}$ at its vertex $\widetilde{P}^{\prime}$ is greater than $\pi$ or
the angle at its vertex $\widetilde{B}^{\prime}$ is greater than $\pi$. For
definiteness, suppose that the angle of $\widetilde{\mathcal{L}}^{\prime}$ at
$\widetilde{P}^{\prime}$ is greater than $\pi$, as shown in Fig. \ref{fig5}.%
%TCIMACRO{\FRAME{ftphFU}{2.8753in}{1.3339in}{0pt}{\Qcb{Sketch for Theorem
%\ref{ThRKBound}}}{\Qlb{fig5}}{fig5.eps}{\special{ language "Scientific Word";
%type "GRAPHIC";  maintain-aspect-ratio TRUE;  display "PICT";
%valid_file "F";  width 2.8753in;  height 1.3339in;  depth 0pt;
%original-width 3.0524in;  original-height 1.4012in;  cropleft "0";
%croptop "1";  cropright "1";  cropbottom "0";
%filename 'Fig5.eps';file-properties "XNPEU";}} }%
%BeginExpansion
\begin{figure}
[pth]
\begin{center}
\includegraphics[
natheight=1.401200in,
natwidth=3.052400in,
height=1.3339in,
width=2.8753in
]%
{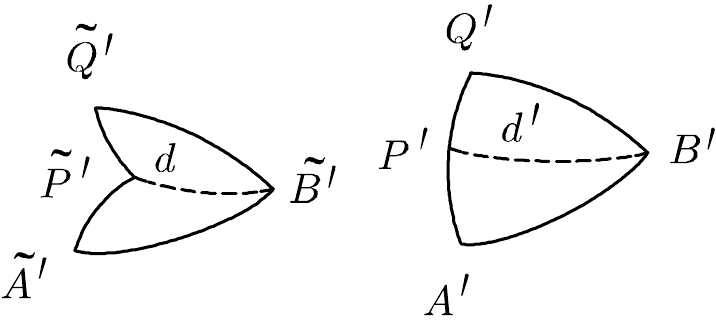}%
\caption{Sketch for Theorem \ref{ThRKBound}}%
\label{fig5}%
\end{center}
\end{figure}
%EndExpansion
Let $\mathcal{V}\subseteq\mathbb{S}_{K}$ be the domain bounded by the triangle
$A^{\prime}Q^{\prime}B^{\prime}$ obtained from the polygonal curve
$\widetilde{\mathcal{L}}^{\prime}$ by rectifying the arc $\widetilde
{A}^{\prime}\widetilde{P}^{\prime}\widetilde{Q}^{\prime}$. Then by \cite[Lemma
2, \S 3]{A1957a}, $d=\widetilde{d}^{\prime}<d^{\prime}=P^{\prime}B^{\prime}$
and $f\leq A^{\prime}Q^{\prime}<\pi/\sqrt{K}$ (as shown in Reshetnyak's
proof). By Corollary \ref{Cor_Pi_by_2}, $d^{\prime}<\pi/\left(  2\sqrt
{K}\right)  $; so, inequalities (\ref{diag_cond}) hold true.

By (\ref{diag_cond}) and because $\operatorname{diam}\left(  \Re_{K}\right)
<\pi/\left(  2\sqrt{K}\right)  $, we see that the difference of the products
\begin{align*}
&  \left(  \cos\kappa x+\cos\kappa d\right)  \left(  \cos\kappa y+\cos\kappa
f\right)  -\left(  \cos\kappa x+\cos\kappa d^{\prime}\right)  \left(
\cos\kappa y+\cos\kappa f^{\prime}\right) \\
&  =\cos\kappa x\left(  \cos\kappa f-\cos\kappa f^{\prime}\right)  +\cos\kappa
y\left(  \cos\kappa d-\cos\kappa d^{\prime}\right) \\
&  +\cos\kappa f\left(  \cos\kappa d-\cos\kappa d^{\prime}\right)  +\cos\kappa
d^{\prime}\left(  \cos\kappa f-\cos\kappa f^{\prime}\right)
\end{align*}
is non-negative. So, $\operatorname{cosq}_{K}\left(  \overrightarrow
{AP},\overrightarrow{BQ}\right)  \leq\operatorname{cosq}_{K}\left(
\overrightarrow{A^{\prime}P^{\prime}},\overrightarrow{B^{\prime}Q^{\prime}%
}\right)  $ follows. By Corollary \ref{CorcosqKbound},
\[
\operatorname{cosq}_{K}\left(  \overrightarrow{AP},\overrightarrow{BQ}\right)
\leq1,
\]
as needed.

Now, we consider the case when $\operatorname{diam}\left(  \mathfrak{Q}%
\right)  =\pi/\left(  2\sqrt{K}\right)  $. If $\varepsilon>0$ is sufficiently
small, by invoking Corollary \ref{Cor_Pi_by_2}, it is not difficult to select
points $A_{\varepsilon},P_{\varepsilon},B_{\varepsilon}$ and $Q_{\varepsilon}$
in $\Re_{K}$ such that the distances $AA_{\varepsilon},$ $PP_{\varepsilon},$
$BB_{\varepsilon}$ and $QQ_{\varepsilon}$ do not exceed $\varepsilon$ and such
that $\operatorname{diam}\left\{  A_{\varepsilon},P_{\varepsilon
},B_{\varepsilon},Q_{\varepsilon}\right\}  <\pi/\left(  2\kappa\right)  $. One
of such configurations is shown in Fig. \ref{fig6}.
%TCIMACRO{\FRAME{ftphFU}{1.7149in}{1.6973in}{0pt}{\Qcb{$\operatorname{diam}%
%\left(  \QTR{frak}{Q}\right)  =\pi/\left(  2\sqrt{K}\right)  $}}{\Qlb{fig6}%
%}{fig6.eps}{\special{ language "Scientific Word";  type "GRAPHIC";
%maintain-aspect-ratio TRUE;  display "PICT";  valid_file "F";
%width 1.7149in;  height 1.6973in;  depth 0pt;  original-width 1.809in;
%original-height 1.7905in;  cropleft "0";  croptop "1";  cropright "1";
%cropbottom "0";  filename 'Fig6.eps';file-properties "XNPEU";}} }%
%BeginExpansion
\begin{figure}
[pth]
\begin{center}
\includegraphics[
natheight=1.790500in,
natwidth=1.809000in,
height=1.6973in,
width=1.7149in
]%
{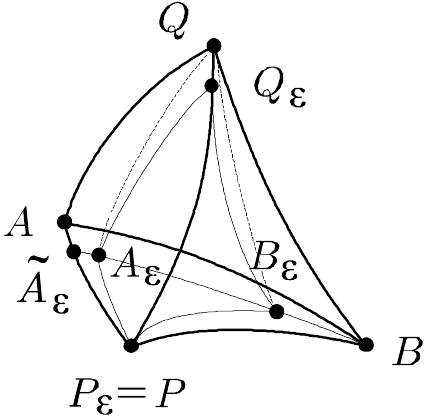}%
\caption{$\operatorname{diam}\left(  \mathfrak{Q}\right)  =\pi/\left(
2\sqrt{K}\right)  $}%
\label{fig6}%
\end{center}
\end{figure}
%EndExpansion
From the first part of the proof, we see that $\left\vert \operatorname{cosq}%
_{K}\left(  \overrightarrow{A_{\varepsilon}P},\overrightarrow{B_{\varepsilon
}Q_{\varepsilon}}\right)  \right\vert \leq1$ for every small positive
$\varepsilon$. Hence, by passing to the limit as $\varepsilon\rightarrow0+$,
we get $\left\vert \operatorname{cosq}_{K}\left(  \overrightarrow
{AP},\overrightarrow{BQ}\right)  \right\vert \leq1$, as claimed.
\end{proof}

The following example shows that for positive $K$, the restriction on the
diameter of $\Re_{K}$ cannot be dropped and the diameter bound in Theorem
\ref{ThRKBound} is sharp. For simplicity, we consider $K=1$.

\begin{example}
\label{Ex_counter_1}Let $\varepsilon>0$. Consider the $T$-shaped graph
$\left(  \mathcal{M}_{\varepsilon},\rho_{\varepsilon}\right)  $ obtained by
gluing a segment of straight line $AO$ of length $\pi/4+\varepsilon$ to the
middle $O$ of another segment of straight line $BQ$ of length $\pi
/2+2\varepsilon$, as shown in Fig. \ref{fig7}.%
%TCIMACRO{\FRAME{ftphFU}{2.0848in}{1.5682in}{0pt}{\Qcb{Sketch for Example
%\ref{Ex_counter_1}}}{\Qlb{fig7}}{fig7.eps}%
%{\special{ language "Scientific Word";  type "GRAPHIC";
%maintain-aspect-ratio TRUE;  display "PICT";  valid_file "F";
%width 2.0848in;  height 1.5682in;  depth 0pt;  original-width 2.2056in;
%original-height 1.6521in;  cropleft "0";  croptop "1";  cropright "1";
%cropbottom "0";  filename 'Fig7.eps';file-properties "XNPEU";}} }%
%BeginExpansion
\begin{figure}
[pth]
\begin{center}
\includegraphics[
natheight=1.652100in,
natwidth=2.205600in,
height=1.5682in,
width=2.0848in
]%
{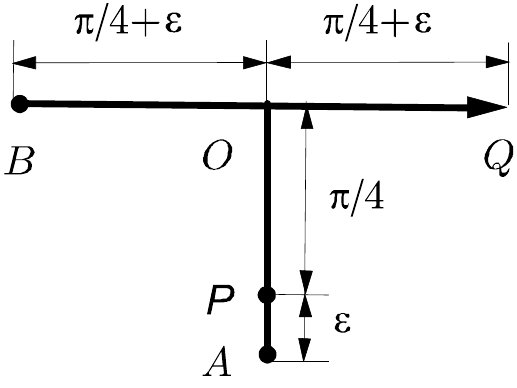}%
\caption{Sketch for Example \ref{Ex_counter_1}}%
\label{fig7}%
\end{center}
\end{figure}
%EndExpansion
It is readily seen that $\left(  \mathcal{M}_{\varepsilon},\rho_{\varepsilon
}\right)  $ is an $\Re_{0}$ domain and that the perimeter of every triangle in
$\left(  \mathcal{M}_{\varepsilon},\rho_{\varepsilon}\right)  $ is less than
$2\pi$ for small positive $\varepsilon$. Hence, $\left(  \mathcal{M}%
_{\varepsilon},\rho_{\varepsilon}\right)  $ is also an $\Re_{1}$ domain.
Notice that $\operatorname{diam}\left(  \mathcal{M}_{\varepsilon}\right)
=\pi/2+2\varepsilon$. Let $P\in\mathcal{AO}\backslash\left\{  A,O\right\}  $
be such that $AP=\varepsilon$. \newline(a)
\begin{align*}
\operatorname{cosq}_{1}\left(  \overrightarrow{BQ},\overrightarrow{AP}\right)
&  =\frac{\cos\left(  \frac{\pi}{2}+\varepsilon\right)  +\cos\left(  \frac
{\pi}{2}+2\varepsilon\right)  \cos\varepsilon}{\sin\left(  \frac{\pi}%
{2}+2\varepsilon\right)  \sin\varepsilon}-\\
&  \frac{2\cos\left(  \frac{\pi}{2}+2\varepsilon\right)  \left[
\cos\varepsilon+\cos\left(  \frac{\pi}{2}+\varepsilon\right)  \right]
}{\left[  1+\cos\left(  \frac{\pi}{2}+2\varepsilon\right)  \right]
\sin\left(  \frac{\pi}{2}+2\varepsilon\right)  \sin\varepsilon}\\
&  =\frac{1+\sin2\varepsilon}{1-\sin2\varepsilon}>1
\end{align*}
for every $\varepsilon\in\left(  0,\pi/4\right)  $ and therefore for small
positive $\varepsilon$. \newline(b) In a similar way,%
\begin{align*}
\operatorname{cosq}_{1}\left(  \overrightarrow{BQ},\overrightarrow{PA}\right)
&  =\frac{\cos\left(  \frac{\pi}{2}+2\varepsilon\right)  +\cos\left(
\frac{\pi}{2}+2\varepsilon\right)  \cos\varepsilon}{\sin\left(  \frac{\pi}%
{2}+2\varepsilon\right)  \sin\varepsilon}-\\
&  \frac{\left[  \cos\left(  \frac{\pi}{2}+2\varepsilon\right)  +\cos\left(
\frac{\pi}{2}+\varepsilon\right)  \right]  \left[  \cos\varepsilon+\cos\left(
\frac{\pi}{2}+2\varepsilon\right)  \right]  }{\left[  1+\cos\left(  \frac{\pi
}{2}+\varepsilon\right)  \right]  \sin\left(  \frac{\pi}{2}+2\varepsilon
\right)  \sin\varepsilon}\\
&  =-\frac{\left(  1+\sin2\varepsilon\right)  \cos\varepsilon}{\left(
1-\sin\varepsilon\right)  \cos2\varepsilon}<-1
\end{align*}
for every $\varepsilon\in\left(  0,\pi/4\right)  $ and therefore for small
positive $\varepsilon$. \newline So, for small positive $\varepsilon$, the
metric space $\left(  \mathcal{M}_{\varepsilon},\rho_{\varepsilon}\right)  $
is an $\Re_{1}$ domain, the diameter of $\left(  \mathcal{M}_{\varepsilon
},\rho_{\varepsilon}\right)  $ is greater than $\pi/2$ and
\[
\lim_{\varepsilon\rightarrow0+}\operatorname{diam}\left(  \mathcal{M}%
_{\varepsilon}\right)  =\pi/2,
\]
whereas $\operatorname{cosq}_{1}$ takes values greater than $1$ and less than
$-1$.\newline
\end{example}

\section{Testing $\operatorname{cosq}_{K}$.
Counterexamples\label{Counter_Examples}}

We begin with the discussion of testing a metric space for the one-sided four
point $\operatorname{cosq}_{K}$ condition. We present counterexamples showing
that in general the upper four point $\operatorname{cosq}_{K}$ condition is
different from the lower four point $\operatorname{cosq}_{K}$ condition.

Let $K\in%
%TCIMACRO{\U{211d} }%
%BeginExpansion
\mathbb{R}
%EndExpansion
$ and let $\mathfrak{Q}\mathcal{=}\left\{  A,P,B,Q\right\}  $ be a quadruple
of distinct points in a metric space $\left(  \mathcal{M},\rho\right)  $ such
that the perimeter of every triple $\left\{  A,B,C\right\}  $ in
$\mathfrak{Q}$ is less than $2\pi/\sqrt{K}$ when $K>0$. For every triple
$X,Y,Z\in\mathfrak{Q}$, the absolute value of the $K$-quadrilateral cosine
between any pair of non-zero bound vectors with heads and tails in the triple
$\left\{  X,Y,Z\right\}  $ always does not exceed one. Indeed, each such
triple can be embedded isometrically into $\mathbb{S}_{K}$; hence, by
Corollary \ref{CorcosqKbound}, $\left\vert \operatorname{cosq}_{K}\right\vert
$ does not exceed $1$ for every pair of such bound vectors. So, by recalling
that $\operatorname{cosq}_{K}$ is symmetric, we need consider only the
following 12 main cases given in Table \ref{table1} where the two non-zero
bound vectors have no point in common.

\begin{center}%
%TCIMACRO{\TeXButton{TeX field}{\begin{table}[h]
%\begin{center}}}%
%BeginExpansion
\begin{table}[h]
\begin{center}%
%EndExpansion
$%
\begin{tabular}
[c]{|l|l|l|l|l|l|l|}\hline
Case & I & II & III & IV & V & VI\\\hline
$\operatorname{cosq}_{1}$ & $\overrightarrow{AP},\overrightarrow{BQ}$ &
$\overrightarrow{AP},\overrightarrow{QB}$ & $\overrightarrow{AB}%
,\overrightarrow{PQ}$ & $\overrightarrow{AB},\overrightarrow{QP}$ &
$\overrightarrow{AQ},\overrightarrow{PB}$ & $\overrightarrow{AQ}%
,\overrightarrow{BP}$\\\hline
Case & VII & VIII & IX & X & XI & XII\\\hline
$\operatorname{cosq}_{1}$ & $\overrightarrow{PA},\overrightarrow{BQ}$ &
$\overrightarrow{PA},\overrightarrow{QB}$ & $\overrightarrow{PB}%
,\overrightarrow{QA}$ & $\overrightarrow{PQ},\overrightarrow{BA}$ &
$\overrightarrow{BA},\overrightarrow{QP}$ & $\overrightarrow{BP}%
,\overrightarrow{QA}$\\\hline
\end{tabular}
\ $%
%TCIMACRO{\TeXButton{TeX field}{\end{center}
%\caption{Twelve main cases}
%\label{table1}
%\end{table}}}%
%BeginExpansion
\end{center}
\caption{Twelve main cases}
\label{table1}
\end{table}%
%EndExpansion

\end{center}

The following examples show that the upper and the lower four point
$\operatorname{cosq}_{K}$ conditions are not equivalent for non-zero $K$. For
simplicity, we consider $K=\pm1$. Adjustment for arbitrary non-zero $K$ is straightforward.

\begin{example}
[$K=1$]\label{ExFPC_PosK}(a) \textbf{The lower four point }%
$\operatorname{cosq}_{1}$\textbf{\ condition holds, whereas the upper four
point }$\operatorname{cosq}_{1}$\textbf{\ condition fails. }Consider the
$T$-shaped graph obtained by gluing a segment of straight line $AP$ of length
$\pi/4+0.1$ to the middle $P$ of another segment of straight line $BQ$ of
length $\pi/2+0.2$, as shown in Fig. \ref{fig8}.
%TCIMACRO{\FRAME{ftphFU}{2.082in}{1.5433in}{0pt}{\Qcb{Sketch for Example
%\ref{ExFPC_PosK}, part (a)}}{\Qlb{fig8}}{fig8.eps}%
%{\special{ language "Scientific Word";  type "GRAPHIC";
%maintain-aspect-ratio TRUE;  display "PICT";  valid_file "F";  width 2.082in;
%height 1.5433in;  depth 0pt;  original-width 2.2029in;
%original-height 1.6254in;  cropleft "0";  croptop "1";  cropright "1";
%cropbottom "0";  filename 'Fig8.eps';file-properties "XNPEU";}} }%
%BeginExpansion
\begin{figure}
[pth]
\begin{center}
\includegraphics[
natheight=1.625400in,
natwidth=2.202900in,
height=1.5433in,
width=2.082in
]%
{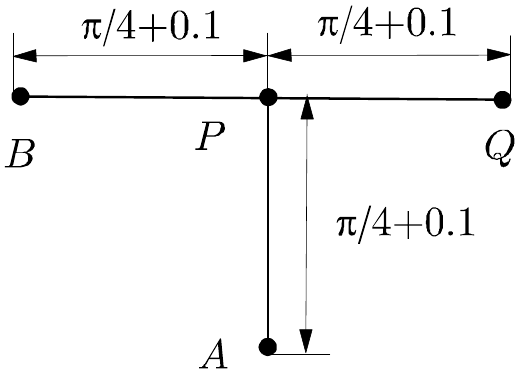}%
\caption{Sketch for Example \ref{ExFPC_PosK}, part (a)}%
\label{fig8}%
\end{center}
\end{figure}
%EndExpansion
Let $\mathcal{M=}\left\{  A,P,B,Q\right\}  $ with the induced metric $\rho$.
All $12$ main (approximate) values of $\operatorname{cosq}_{1}$ for the four
point metric space $\left(  \mathcal{M},\rho\right)  $ are given in Table
\ref{table2}.\newline%
%TCIMACRO{\TeXButton{TeX field}{\begin{table}[h]
%\begin{center}}}%
%BeginExpansion
\begin{table}[h]
\begin{center}%
%EndExpansion
$%
\begin{tabular}
[c]{|l|l|l|l|l|l|l|}\hline
Case & I & II & III & IV & V & VI\\\hline
$\operatorname{cosq}_{1}$ & $1.496$ & $1.496$ & $-0.58$ & $1.496$ & $-0.58$ &
$1.496$\\\hline
Case & VII & VIII & IX & X & XI & XII\\\hline
$\operatorname{cosq}_{1}$ & $-0.58$ & $-0.58$ & $-0.58$ & $-0.58$ & $1.496$ &
$1.496$\\\hline
\end{tabular}
\ \ \ $%
%TCIMACRO{\TeXButton{TeX field}{\end{center}
%\caption{Example \ref{ExFPC_PosK}, part (a)}
%\label{table2}
%\end{table}}}%
%BeginExpansion
\end{center}
\caption{Example \ref{ExFPC_PosK}, part (a)}
\label{table2}
\end{table}%
%EndExpansion
\newline(b) \textbf{The upper four point }$\operatorname{cosq}_{1}%
$\textbf{\ condition holds, whereas the lower four point }$\operatorname{cosq}%
_{1}$\textbf{\ condition fails.} Consider the quadruple $\mathfrak{Q}=\left\{
A,P,B,Q\right\}  $ in $\mathbb{S}_{1}$ with the metric $\rho_{\mathbb{S}_{1}}$
such that the point $P$ is symmetric to the point $Q$ w.r.t. the midpoint of
the shortest $\mathcal{AB}$. All $6$ distances between the pairs of points of
$\mathfrak{Q} $ are shown in Fig. \ref{fig9} with $\varepsilon=0$. Then
$\operatorname{cosq}_{1}\left(  \overrightarrow{AP},\overrightarrow
{BQ}\right)  =-1$.
%TCIMACRO{\FRAME{ftphFU}{1.8246in}{1.2822in}{0pt}{\Qcb{Sketch for Example
%\ref{ExFPC_PosK}, part (b)}}{\Qlb{fig9}}{fig9.eps}%
%{\special{ language "Scientific Word";  type "GRAPHIC";
%maintain-aspect-ratio TRUE;  display "PICT";  valid_file "F";
%width 1.8246in;  height 1.2822in;  depth 0pt;  original-width 1.927in;
%original-height 1.345in;  cropleft "0";  croptop "1";  cropright "1";
%cropbottom "0";  filename 'Fig9.eps';file-properties "XNPEU";}} }%
%BeginExpansion
\begin{figure}
[pth]
\begin{center}
\includegraphics[
natheight=1.345000in,
natwidth=1.927000in,
height=1.2822in,
width=1.8246in
]%
{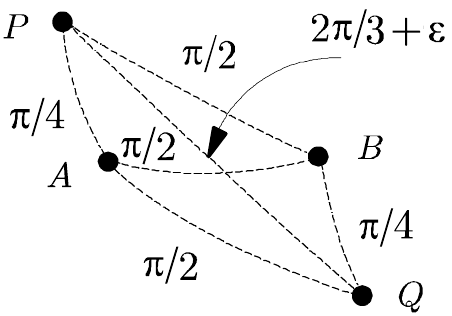}%
\caption{Sketch for Example \ref{ExFPC_PosK}, part (b)}%
\label{fig9}%
\end{center}
\end{figure}
%EndExpansion
Now we change the metric $\rho_{\mathbb{S}_{1}}$ by increasing the distance
between $P$ and $Q$ by a positive $\varepsilon$ and leaving all other
distances the same. If $\varepsilon$ is sufficiently small, then the new
distance $\rho_{\varepsilon}$ is a metric. For $\varepsilon=0.1$, all $12$
main (approximate) values of $\operatorname{cosq}_{1}$ for the four point
metric space $\left(  \mathfrak{Q},\rho_{0.1}\right)  $ are given in Table
\ref{table3}.\newline%
%TCIMACRO{\TeXButton{TeX field}{\begin{table}[h]
%\begin{center}}}%
%BeginExpansion
\begin{table}[h]
\begin{center}%
%EndExpansion
$%
\begin{tabular}
[c]{|l|l|l|l|l|l|l|}\hline
Case & I & II & III & IV & V & VI\\\hline
$\operatorname{cosq}_{1}$ & $-1.168$ & $0.826$ & $0.871$ & $-0.107$ & $0.707$
& $-1.084$\\\hline
Case & VII & VIII & IX & X & XI & XII\\\hline
$\operatorname{cosq}_{1}$ & $0.826$ & $-1.404$ & $-1.202$ & $-0.107$ & $0.871$
& $0.707$\\\hline
\end{tabular}
\ \ \ \ $%
%TCIMACRO{\TeXButton{TeX field}{\end{center}
%\caption{Example \ref{ExFPC_PosK}, part (b)}
%\label{table3}
%\end{table}}}%
%BeginExpansion
\end{center}
\caption{Example \ref{ExFPC_PosK}, part (b)}
\label{table3}
\end{table}%
%EndExpansion
\newline
\end{example}

\begin{example}
[$K=-1$]\label{ExFPC_NegK}We use the same approach to construction of
counterexamples for $K=-1$ as in part (b) of Example \ref{ExFPC_PosK}. Let
$\mathfrak{Q}\mathcal{=}\left\{  A,P,B,Q\right\}  $ be a four element
set.\newline(a) \textbf{The lower four point }$\operatorname{cosq}_{-1}%
$\textbf{\ condition holds, whereas the upper four point }$\operatorname{cosq}%
_{-1}$\textbf{\ condition fails. }The $6$ (symmetric) distances between the
pairs of points in $\mathfrak{Q}$ are given by
\begin{align*}
\rho\left(  A,P\right)   &  =\rho\left(  B,Q\right)  =1,\text{ }\rho\left(
A,B\right)  =2,\\
\rho\left(  P,Q\right)   &  =2.697\text{ and }\rho\left(  A,Q\right)
=\rho\left(  B,P\right)  =2.44.
\end{align*}
All $12$ main (approximate) values of $\operatorname{cosq}_{-1}$ for the four
point metric space $\left(  \mathfrak{Q},\rho\right)  $ are given in Table
\ref{table4}.\newline%
%TCIMACRO{\TeXButton{TeX field}{\begin{table}[h]
%\begin{center}}}%
%BeginExpansion
\begin{table}[h]
\begin{center}%
%EndExpansion
$%
\begin{tabular}
[c]{|l|l|l|l|l|l|l|}\hline
Case & I & II & III & IV & V & VI\\\hline
$\operatorname{cosq}_{-1}$ & $1.0347$ & $-0.8133$ & $0.7495$ & $-0.9998$ &
$0.4534$ & $-0.9133$\\\hline
Case & VII & VIII & IX & X & XI & XII\\\hline
$\operatorname{cosq}_{-1}$ & $-0.8133$ & $0.1465$ & $-0.9511$ & $-0.9998$ &
$0.7495$ & $0.4534$\\\hline
\end{tabular}
\ $%
%TCIMACRO{\TeXButton{TeX field}{\end{center}
%\caption{Example \ref{ExFPC_NegK}, part (a)}
%\label{table4}
%\end{table}}}%
%BeginExpansion
\end{center}
\caption{Example \ref{ExFPC_NegK}, part (a)}
\label{table4}
\end{table}%
%EndExpansion
\newline(b) \textbf{The upper four point }$\operatorname{cosq}_{-1}%
$\textbf{\ condition holds, whereas the lower four point }$\operatorname{cosq}%
_{-1}$\textbf{\ condition fails. }The $6$ distances between the pairs of
points in $\mathfrak{Q}$ are given by
\begin{align*}
\rho\left(  A,P\right)   &  =\rho\left(  B,Q\right)  =1,\text{ }\rho\left(
A,B\right)  =2,\\
\rho\left(  P,Q\right)   &  =3.027\text{ and }\rho\left(  A,Q\right)
=\rho\left(  B,P\right)  =2.43.
\end{align*}
All $12$ main (approximate) values of $\operatorname{cosq}_{-1}$ for the four
point metric space $\left(  \mathfrak{Q},\rho\right)  $ are given in Table
\ref{table5}.\newline%
%TCIMACRO{\TeXButton{TeX field}{\begin{table}[h]
%\begin{center}}}%
%BeginExpansion
\begin{table}[h]
\begin{center}%
%EndExpansion
$%
\begin{tabular}
[c]{|l|l|l|l|l|l|l|}\hline
Case & I & II & III & IV & V & VI\\\hline
$\operatorname{cosq}_{-1}$ & $-1.184$ & $0.922$ & $0.522$ & $-0.944$ & $0.807$
& $-1.008$\\\hline
Case & VII & VIII & IX & X & XI & XII\\\hline
$\operatorname{cosq}_{-1}$ & $0.922$ & $-1.077$ & $-1.003$ & $-0.944$ &
$0.522$ & $0.807$\\\hline
\end{tabular}
$%
%TCIMACRO{\TeXButton{TeX field}{\end{center}
%\caption{Example \ref{ExFPC_NegK}, part (b)}
%\label{table5}
%\end{table}}}%
%BeginExpansion
\end{center}
\caption{Example \ref{ExFPC_NegK}, part (b)}
\label{table5}
\end{table}%
%EndExpansion
\newline
\end{example}

\section{Proof of Theorem \ref{MainTh}\label{Proof_Main_Th}}

\subsection{Sketch of the proof}

Let $\left(  \mathcal{M},\rho\right)  $ be a geodesically connected metric
space (of diameter not greater than $\pi/2\sqrt{K}$ for positive $K$)
satisfying the one-sided four point $\operatorname{cosq}_{K}$ condition for
non-zero $K$. Theorem \ref{MainTh} is proved once we establish the angle
comparison: for every geodesic triangle $\mathcal{T}=ABC$ in $\left(
\mathcal{M},\rho\right)  $, $\measuredangle ABC\leq\measuredangle_{K}ABC,$
$\measuredangle BAC\leq\measuredangle_{K}BAC$ and $\measuredangle
ACB\leq\measuredangle_{K}ACB$. We begin by proving Lemma \ref{Lemma_cont_geod}
stating that shortests in $\left(  \mathcal{M},\rho\right)  $ depend
continuously on their end points. One of Aleksandrov's theorem and Lemma
\ref{Lemma_cont_geod} enable us to reduce the derivation of the global angle
comparison estimate to the proof of the local angle comparison. The
cross-diagonal estimate lemma (Lemma \ref{Lemma_cross_diag}) is one of the
main steps in the proof of the major growth estimate lemma (Lemma
\ref{Lemma_Growth_est}). Both of these estimates are derived from the
one-sided four point $\operatorname{cosq}_{K}$ condition. We employ the growth
estimate to prove \textquotedblleft almost monotonicity\textquotedblright\ of
the angles $\alpha_{0}\left(  t\right)  $ (Corollary \ref{Cor_cos_tau}) and
existence of proportional angles (Corollary \ref{Cor_prop_angle}), an
important auxiliary step in proving the existence of Aleksandrov angles
(Proposition \ref{Prop_Exist_Angle}). Now we have all necessary means needed
for derivation of the local angle comparison inequality. We begin with the
identity corresponding to the growth estimate in $\mathbb{S}_{K}$ (Proposition
\ref{Prop_sper_ident}). We consider a sufficiently small geodesic triangle
$\mathcal{T}=ABC$ in $\left(  \mathcal{M},\rho\right)  $. Existence of
Aleksandrov angles gives us the freedom of selecting the points in shortests
$\mathcal{AB}$ and $\mathcal{AC}$ respectively approaching to the vertex $A$
in a special way. For every small positive $t$, we select $\widehat{X}_{t}%
\in\mathcal{AB}$ and $\widehat{Y}_{t}\in\mathcal{AC}$, $\widehat{X}%
_{t},\widehat{Y}_{t}$ $\rightarrow A$ as $t\rightarrow0+$ (see Sec.
\ref{Angle_Comparison}) so that $\measuredangle A^{K}\widehat{X}_{t}%
^{K}\widehat{Y}_{t}^{K}$ and $\measuredangle\widehat{X}_{t}^{K}A^{K}%
\widehat{Y}_{t}^{K}$ converge as $t\rightarrow0+$ (Lemma
\ref{Angle_conv_comp_th}). Hence, it is possible to pass to the limit in the
growth estimate as $t\rightarrow0+$. The limit form of the growth estimate and
the identity of Proposition \ref{Prop_sper_ident} enables us to derive the
local angle comparison estimate (Proposition \ref{Prop_angle_comp}).

\subsection{Continuity and uniqueness of shortests\label{Cont_geod}}

The main result of this section is the following

\begin{lemma}
\label{Lemma_cont_geod}Let $K\neq0$ and let $\left(  \mathcal{M},\rho\right)
$ be a metric space such that $\operatorname{diam}\left(  \mathcal{M}\right)
\leq\pi/\left(  2\sqrt{K}\right)  $ when $K>0$. Let $\mathcal{L=AB}$ be a
shortest and $\left(  \mathcal{L}_{n}=\mathcal{A}_{n}\mathcal{B}_{n}\right)
_{n=1}^{\infty}$ be a sequence of shortests in $\left(  \mathcal{M}%
,\rho\right)  $ such that $\lim_{n\rightarrow\infty}A_{n}=A$ and
$\lim_{n\rightarrow\infty}B_{n}=B$. Let $\mathfrak{g}_{r}$ be the reduced
parametrization of $\mathcal{L}$ relative to $A$ and $\mathfrak{g}_{r,n}$ be
the reduced parametrization of $\mathcal{L}_{n}$ relative to $A_{n}$,
$n=1,2,...$ (see, Sec. \ref{Al_upper_curv_cond}) . If $\left(  \mathcal{M}%
,\rho\right)  $ satisfies the one-sided four point $\operatorname{cosq}_{K}$
condition, then the sequence $\left(  \mathfrak{g}_{r,n}\right)
_{n=1}^{\infty}$ converges uniformly to $\mathfrak{g}_{r}$ on the closed
interval $\left[  0,1\right]  $.
\end{lemma}

\begin{proof}
Let $\mathcal{L}=\mathcal{AB}$ and $\mathcal{L}_{n}=\mathcal{A}_{n}%
\mathcal{B}_{n}$ $n=1,2,...\,$\ . We can assume that $l=\ell_{\rho}\left(
\mathcal{L}\right)  >0$ and $l_{n}=\ell_{\rho}\left(  \mathcal{L}_{n}\right)
>0$ for every $n$. For $t\in\left(  0,1\right)  $, set $P=\mathfrak{g}%
_{r}\left(  t\right)  $, $P_{n}=\mathfrak{g}_{r,n}\left(  t\right)  $ and
$\overline{\delta}=\overline{\lim}_{n\rightarrow\infty}PP_{n}$, see Fig.
\ref{fig10}.%
%TCIMACRO{\FRAME{ftphFU}{2.5829in}{1.4741in}{0pt}{\Qcb{Sketch for Lemma
%\ref{Lemma_cont_geod}}}{\Qlb{fig10}}{fig10.eps}%
%{\special{ language "Scientific Word";  type "GRAPHIC";
%maintain-aspect-ratio TRUE;  display "PICT";  valid_file "F";
%width 2.5829in;  height 1.4741in;  depth 0pt;  original-width 2.7397in;
%original-height 1.5507in;  cropleft "0";  croptop "1";  cropright "1";
%cropbottom "0";  filename 'Fig10.eps';file-properties "XNPEU";}} }%
%BeginExpansion
\begin{figure}
[pth]
\begin{center}
\includegraphics[
natheight=1.550700in,
natwidth=2.739700in,
height=1.4741in,
width=2.5829in
]%
{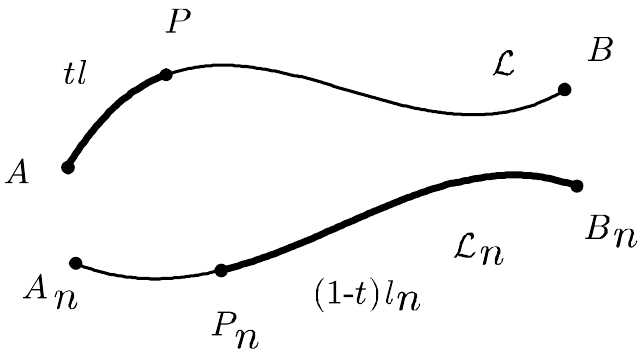}%
\caption{Sketch for Lemma \ref{Lemma_cont_geod}}%
\label{fig10}%
\end{center}
\end{figure}
%EndExpansion

\textbf{I}. \textbf{Let }$\left(  \mathcal{M},\rho\right)  $\textbf{\ satisfy
the upper four point }$\operatorname{cosq}_{K}$\textbf{\ condition.} Consider
the non-zero bound vectors $\overrightarrow{AP}$ and $\overrightarrow
{P_{n}B_{n}}$. By the upper four point $\operatorname{cosq}_{K}$ condition,%
\begin{align*}
&  \operatorname{cosq}_{K}\left(  \overrightarrow{AP},\overrightarrow
{P_{n}B_{n}}\right)  =\frac{\cos\widehat{\kappa}PB_{n}+\cos\widehat{\kappa
}AP\cos\widehat{\kappa}P_{n}B_{n}}{\sin\widehat{\kappa}AP\sin\widehat{\kappa
}P_{n}B_{n}}\\
&  -\frac{\left(  \cos\widehat{\kappa}AP+\cos\widehat{\kappa}PP_{n}\right)
\left(  \cos\widehat{\kappa}P_{n}B_{n}+\cos\widehat{\kappa}B_{n}A\right)
}{\left(  1+\cos\widehat{\kappa}AP_{n}\right)  \sin\widehat{\kappa}%
AP\sin\widehat{\kappa}P_{n}B_{n}}%
\end{align*}
does not exceed $1$. By letting $n\rightarrow\infty$, we get%
\begin{align}
&  \overline{\lim}_{n\rightarrow\infty}\operatorname{cosq}_{K}\left(
\overrightarrow{AP},\overrightarrow{P_{n}B_{n}}\right)  =\frac{\cos
\widehat{\kappa}\left(  1-t\right)  l+\cos\widehat{\kappa}tl\cos
\widehat{\kappa}\left(  1-t\right)  l}{\sin\widehat{\kappa}tl\sin
\widehat{\kappa}\left(  1-t\right)  l}\nonumber\\
&  -\frac{\left(  \cos\widehat{\kappa}tl+\cos\widehat{\kappa}\overline{\delta
}\right)  \left[  \cos\widehat{\kappa}\left(  1-t\right)  l+\cos
\widehat{\kappa}l\right]  }{\left(  1+\cos\widehat{\kappa}tl\right)
\sin\widehat{\kappa}tl\sin\widehat{\kappa}\left(  1-t\right)  l}=1+\nonumber\\
&  \frac{\cos\widehat{\kappa}\left(  1-t\right)  l+\cos\widehat{\kappa}l}%
{\sin\widehat{\kappa}tl\sin\widehat{\kappa}\left(  1-t\right)  l}%
-\frac{\left(  \cos\widehat{\kappa}tl+\cos\widehat{\kappa}\overline{\delta
}\right)  \left[  \cos\widehat{\kappa}\left(  1-t\right)  l+\cos
\widehat{\kappa}l\right]  }{\left(  1+\cos\widehat{\kappa}tl\right)
\sin\widehat{\kappa}tl\sin\widehat{\kappa}\left(  1-t\right)  l}\nonumber\\
&  =1+\frac{\left[  1-\cos\left(  \widehat{\kappa}\overline{\delta}\right)
\right]  \left[  \cos\left(  \widehat{\kappa}\left(  1-t\right)  l\right)
+\cos\left(  \widehat{\kappa}l\right)  \right]  }{\left[  1+\cos\left(
\widehat{\kappa}tl\right)  \right]  \sin\left(  \widehat{\kappa}tl\right)
\sin\left(  \widehat{\kappa}\left(  1-t\right)  l\right)  }\leq1.
\label{cont_g_c}%
\end{align}

If $K>0$, then%
\[
\frac{\left(  1-\cos\kappa\overline{\delta}\right)  \left[  \cos\kappa\left(
1-t\right)  l+\cos\kappa l\right]  }{\left(  1+\cos\kappa tl\right)
\sin\kappa tl\sin\kappa\left(  1-t\right)  l}\leq0.
\]
Because $\operatorname{diam}\left(  \mathcal{M}\right)  \leq\pi/\left(
2\kappa\right)  $, $\overline{\delta}=0$ follows.

If $K<0$, then%
\[
\frac{\left(  \cosh\kappa\overline{\delta}-1\right)  \left[  \cosh
\kappa\left(  1-t\right)  l+\cosh\kappa l\right]  }{\left(  1+\cosh\kappa
tl\right)  \sinh\kappa tl\sinh\kappa\left(  1-t\right)  l}\leq0,
\]
whence $\overline{\delta}=0$ follows.

\textbf{II. Let }$\left(  \mathcal{M},\rho\right)  $\textbf{\ satisfy the
lower four point }$\operatorname{cosq}_{K}$\textbf{\ condition. }In a manner
similar to I, we get%
\begin{align*}
&  \underline{\lim}_{n\rightarrow\infty}\operatorname{cosq}_{K}\left(
\overrightarrow{AP},\overrightarrow{B_{n}P_{n}}\right) \\
&  =\frac{\cos\widehat{\kappa}\overline{\delta}+\cos\widehat{\kappa}%
tl\cos\widehat{\kappa}\left(  1-t\right)  l}{\sin\widehat{\kappa}%
tl\sin\widehat{\kappa}\left(  1-t\right)  l}-\frac{\left[  \cos\widehat
{\kappa}tl+\cos\widehat{\kappa}\left(  1-t\right)  l\right]  ^{2}}{\left(
1+\cos\widehat{\kappa}l\right)  \sin\widehat{\kappa}tl\sin\widehat{\kappa
}\left(  1-t\right)  l}\geq-1,
\end{align*}
whence%
\[
\frac{\left[  \cos\widehat{\kappa}\overline{\delta}+\cos\widehat{\kappa
}\left(  1-2t\right)  l\right]  \left(  1+\cos\widehat{\kappa}l\right)
}{\left(  1+\cos\widehat{\kappa}l\right)  \sin\widehat{\kappa}tl\sin
\widehat{\kappa}\left(  1-t\right)  l}-\frac{\left[  \cos\widehat{\kappa
}tl+\cos\widehat{\kappa}\left(  1-t\right)  l\right]  ^{2}}{\left(
1+\cos\widehat{\kappa}l\right)  \sin\widehat{\kappa}tl\sin\widehat{\kappa
}\left(  1-t\right)  l}%
\]
is non-negative. Notice that%
\begin{align*}
\left[  \cos\widehat{\kappa}tl+\cos\widehat{\kappa}\left(  1-t\right)
l\right]  ^{2}  &  =4\cos^{2}\frac{\widehat{\kappa}l}{2}\cos^{2}\frac
{\widehat{\kappa}\left(  1-2t\right)  l}{2}=\\
&  \left(  1+\cos\widehat{\kappa}l\right)  \left(  1+\cos\widehat{\kappa
}\left(  1-2t\right)  l\right)  .
\end{align*}
Hence,
\begin{align*}
&  \left[  \cos\widehat{\kappa}\overline{\delta}+\cos\widehat{\kappa}\left(
1-2t\right)  l\right]  \left(  1+\cos\widehat{\kappa}l\right)  -\left[
\cos\widehat{\kappa}tl+\cos\widehat{\kappa}\left(  1-t\right)  l\right]
^{2}\\
&  =\left(  \cos\widehat{\kappa}\overline{\delta}-1\right)  \left(
1+\cos\widehat{\kappa}l\right)  .
\end{align*}
So,
\[
\underline{\lim}_{n\rightarrow\infty}\operatorname{cosq}_{K}\left(
\overrightarrow{AP},\overrightarrow{B_{n}P_{n}}\right)  +1=\frac{\cos
\widehat{\kappa}\overline{\delta}-1}{\sin\widehat{\kappa}tl\sin\widehat
{\kappa}\left(  1-t\right)  l}\geq0\text{.}%
\]

If $K>0$, then%
\[
\frac{\cos\kappa\overline{\delta}-1}{\sin\kappa tl\sin\kappa\left(
1-t\right)  l}\geq0,
\]
whence $\overline{\delta}=0$ follows.

If $K<0$, then%
\[
\frac{\cosh\widehat{\kappa}\overline{\delta}-1}{\sinh\kappa tl\sinh
\kappa\left(  1-t\right)  l}\leq0
\]
whence $\overline{\delta}=0$.

By I and II, $\mathfrak{g}_{r,n}\left(  t\right)  $ converges pointwise to
$\mathfrak{g}_{r}\left(  t\right)  $ for every $t\in\left[  0,1\right]  $ as
$n\rightarrow\infty$. It is not difficult to see that sequence $\left(
\mathfrak{g}_{r,n}\right)  _{n=1}^{\infty}$ also converges uniformly to
$\mathfrak{g}_{r}$ on the closed interval $\left[  0,1\right]  $.

The proof of Lemma \ref{Lemma_cont_geod} is complete.
\end{proof}

\begin{corollary}
\label{Cor_uniqueness}Let $K\neq0$ and let $\left(  \mathcal{M},\rho\right)  $
be a metric space such that $\operatorname{diam}\left(  \mathcal{M}\right)  $
is not greater than $\pi/\left(  2\sqrt{K}\right)  $ when $K>0$. If $\left(
\mathcal{M},\rho\right)  $ satisfies the one-sided four point
$\operatorname{cosq}_{K}$ condition, then every pair of points in
$\mathcal{M}$ can be joined by at most one shortest.
\end{corollary}

\subsection{Cross-diagonal estimate lemma\label{Sec_cross_diag_est}}

Let $\left(  \mathcal{M},\rho\right)  $ be a metric space. Let $A,B,C$ be
three distinct points in $\mathcal{M}$, $0<\underline{m}\leq\overline
{m}<+\infty$ and $s,t\in(0,1]$ satisfying the following conditions:

M1. The points $A$ and $B$ can be joined by a shortest $\mathcal{L}$ and the
points $A$ and $C$ can be joined by a shortest $\mathcal{N}$.

M2. If $K>0$, then $AB,AC\leq\pi/2\sqrt{K}$.

M3. $\underline{m}\leq s/t\leq\overline{m}$.

From now on, we will use the following notation:
\begin{align*}
X_{s}  &  =\mathfrak{g}_{r,\mathcal{L}}\left(  s\right)  \text{, }%
Y_{t}=\mathfrak{g}_{r,\mathcal{N}}\left(  t\right)  ,\text{ }s,\text{ }%
t\in(0,1].\\
x  &  =AB,\text{ }y=AC,\text{ }z=BC,\text{ }d_{s,t}=BY_{t},\text{ }%
f_{s,t}=CX_{s},\text{ }z_{s,t}=X_{s}Y_{t}\text{,}%
\end{align*}
as illustrated in Fig. \ref{fig11},
%TCIMACRO{\FRAME{ftphFU}{2.0202in}{1.7674in}{0pt}{\Qcb{Sketch for the
%cross-diagonal lemma}}{\Qlb{fig11}}{fig11.eps}%
%{\special{ language "Scientific Word";  type "GRAPHIC";
%maintain-aspect-ratio TRUE;  display "PICT";  valid_file "F";
%width 2.0202in;  height 1.7674in;  depth 0pt;  original-width 2.1364in;
%original-height 1.8661in;  cropleft "0";  croptop "1";  cropright "1";
%cropbottom "0";  filename 'Fig11.eps';file-properties "XNPEU";}} }%
%BeginExpansion
\begin{figure}
[pth]
\begin{center}
\includegraphics[
natheight=1.866100in,
natwidth=2.136400in,
height=1.7674in,
width=2.0202in
]%
{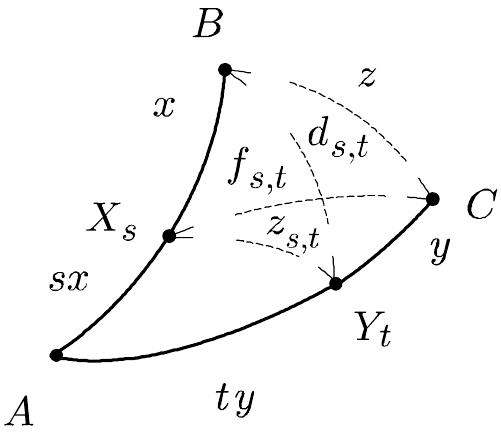}%
\caption{Sketch for the cross-diagonal lemma}%
\label{fig11}%
\end{center}
\end{figure}
%EndExpansion
and we put $\lambda=\max\left\{  x,y\right\}  ,\eta=x/y$, $\xi=\lambda
\max\left\{  s,t\right\}  $. Also, for $K\in%
%TCIMACRO{\U{211d} }%
%BeginExpansion
\mathbb{R}
%EndExpansion
$, set
\[
\alpha_{K}\left(  s,t\right)  =\measuredangle_{K}X_{s}AY_{t},\text{ }\beta
_{K}\left(  s,t\right)  =\measuredangle_{K}AX_{s}Y_{t},\text{ }\gamma
_{K}\left(  s,t\right)  =\measuredangle_{K}AY_{t}X_{s}.
\]

If $p>0$, we write $\varphi\left(  s,t\right)  =\mathcal{O}\left(  \xi
^{p}\right)  $ when there is a constant $C>0$ such that $\left\vert
\varphi\left(  s,t\right)  \right\vert \leq C\xi^{p}$ for sufficiently small
$s $ and $t$. If $C$ is a constant depending on $M_{1},M_{2},...,M_{k}$, i.e.,
$C=C\left(  M_{1},M_{2},...,M_{k}\right)  $, then we write $\varphi\left(
s,t\right)  =\mathcal{O}_{M_{1},M_{2},...,M_{k}}\left(  \xi^{p}\right)  $.

If $\mathcal{T}=ABC$ is a triangle in $\mathbb{S}_{1}$, then $f_{s,t}$ is not
less than the length of the orthogonal projection of the shortest
$\mathcal{X}_{s}\mathcal{C}$ onto the shortest $\mathcal{AC}$. So, if $O_{s}$
is the orthogonal projection of the point $X_{s}$ onto the shortest
$\mathcal{AC}$, then $f_{s,t}\geq y-AO_{s}$. It is not difficult to see that
$AO_{s}$ approximately equals $\left(  sx\right)  \cos\alpha_{0}\left(
s,t\right)  $. Hence, approximately, $f_{s,t}$ is bounded below by $y-\left(
sx\right)  \cos\alpha_{0}\left(  s,t\right)  $. The following lemma states a
similar estimate for a triangle $\mathcal{T}=ABC$ in a metric space satisfying
the one-sided four point $\operatorname{cosq}_{K}$ condition.

\begin{lemma}
\label{Lemma_cross_diag}Let $K\neq0$ and $0<\underline{m}\leq\overline
{m}<+\infty$. Let $A,B,C$ be three distinct points in a metric space $\left(
\mathcal{M},\rho\right)  $ and $s,t\in(0,1]$ satisfying M1-M3. Suppose that
$\left(  \mathcal{M},\rho\right)  $ satisfies the one-sided four point
$\operatorname{cosq}_{K}$ condition.\newline(i) If $K>0$, then
\begin{align*}
\cos\kappa f_{s,t}  &  \leq\cos\kappa y+\kappa\left(  sx\right)  \sin\kappa
y\cos\alpha_{0}\left(  s,t\right)  +\mathcal{O}\left(  \xi^{2}\right)  ,\\
\cos\kappa d_{s,t}  &  \leq\cos\kappa x+\kappa\left(  ty\right)  \sin\kappa
x\cos\alpha_{0}\left(  s,t\right)  +\mathcal{O}\left(  \xi^{2}\right)  .
\end{align*}
\newline(ii) If $K<0$, then
\begin{align*}
\cosh\kappa f_{s,t}  &  \geq\cosh\kappa y-\kappa\left(  sx\right)  \sinh\kappa
y\cos\alpha_{0}\left(  s,t\right)  +\mathcal{O}\left(  \xi^{2}\right)
,\newline\\
\cosh\kappa d_{s,t}  &  \geq\cosh\kappa x-\kappa\left(  ty\right)  \sinh\kappa
x\cos\alpha_{0}\left(  s,t\right)  +\mathcal{O}\left(  \xi^{2}\right)  ,
\end{align*}
where $\mathcal{O}\left(  \xi^{2}\right)  =\mathcal{O}_{\lambda,\eta
,\underline{m},\overline{m},K}\left(  \xi^{2}\right)  $.
\end{lemma}

\begin{proof}
\textbf{I}. \textbf{Let }$\left(  \mathcal{M},\rho\right)  $\textbf{\ satisfy
the upper four point }$\operatorname{cosq}_{K}$\textbf{\ condition}. For the
sake of brevity, set $h_{s,t}=\operatorname{cosq}_{K}\left(  \overrightarrow
{X_{s}C},\overrightarrow{AY_{t}}\right)  $. Then\newline%
\begin{align*}
h_{s,t}  &  =\frac{\cos\widehat{\kappa}\left(  1-t\right)  y+\cos
\widehat{\kappa}f_{s,t}\cos\widehat{\kappa}ty}{\sin\widehat{\kappa}f_{s,t}%
\sin\widehat{\kappa}ty}\\
&  -\frac{\left(  \cos\widehat{\kappa}f_{s,t}+\cos\widehat{\kappa}y\right)
\left(  \cos\widehat{\kappa}ty+\cos\widehat{\kappa}z_{s,t}\right)  }{\left(
1+\cos\widehat{\kappa}sx\right)  \sin\widehat{\kappa}f_{s,t}\sin
\widehat{\kappa}ty}\\
&  =\frac{\cos\widehat{\kappa}y+\widehat{\kappa}ty\sin\widehat{\kappa}%
y-\frac{1}{2}\widehat{\kappa}^{2}\left(  ty\right)  ^{2}\cos\widehat{\kappa
}y+\mathcal{O}\left(  \xi^{3}\right)  }{\sin\widehat{\kappa}f_{s,t}%
\sin\widehat{\kappa}ty}\\
&  +\frac{\cos\widehat{\kappa}f_{s,t}\left[  1-\frac{1}{2}\widehat{\kappa}%
^{2}\left(  ty\right)  ^{2}+\mathcal{O}\left(  \xi^{4}\right)  \right]  }%
{\sin\widehat{\kappa}f_{s,t}\sin\widehat{\kappa}ty}\\
&  -\frac{\cos\widehat{\kappa}f_{s,t}+\cos\widehat{\kappa}y}{\sin
\widehat{\kappa}f_{s,t}\sin\widehat{\kappa}ty}\left[  \frac{1}{2}+\frac{1}%
{8}\widehat{\kappa}^{2}\left(  sx\right)  ^{2}+\mathcal{O}\left(  \xi
^{4}\right)  \right]  \times\\
&  \left[  2-\frac{1}{2}\widehat{\kappa}^{2}\left(  ty\right)  ^{2}-\frac
{1}{2}\widehat{\kappa}^{2}z_{s,t}^{2}+\mathcal{O}\left(  \xi^{4}\right)
\right]  .
\end{align*}
After lengthy but routine simplifications and using the upper four point
$\operatorname{cosq}_{K}$ condition, we get:
\begin{align*}
h_{s,t}  &  =\frac{\widehat{\kappa}\left(  ty\right)  \sin\widehat{\kappa
}y-\widehat{\kappa}^{2}\left(  sx\right)  \left(  ty\right)  \frac
{\cos\widehat{\kappa}y+\cos\widehat{\kappa}f_{s,t}}{2}\cos\alpha_{0}\left(
s,t\right)  +\mathcal{O}\left(  \xi^{3}\right)  }{\widehat{\kappa}\left(
ty\right)  \left[  1+\mathcal{O}\left(  \xi^{2}\right)  \right]  \sin
\widehat{\kappa}f_{s,t}}=\\
&  =\frac{\sin\widehat{\kappa}y-\widehat{\kappa}\left(  sx\right)  \frac
{\cos\widehat{\kappa}y+\cos\widehat{\kappa}f_{s,t}}{2}\cos\alpha_{0}\left(
s,t\right)  }{\sin\widehat{\kappa}f_{s,t}}+\mathcal{O}\left(  \xi^{2}\right)
\leq1.
\end{align*}
Set $\mu=\left(  \cos\widehat{\kappa}y+\cos\widehat{\kappa}f_{s,t}\right)
/2$. By the triangle inequality, $\left\vert f_{s,t}-y\right\vert \leq sx$.
Hence, $\mu=\cos\widehat{\kappa}y+\mathcal{O}\left(  \xi\right)  $ follows and
we have:%
\begin{equation}
h_{s,t}=\frac{\sin\widehat{\kappa}y-\widehat{\kappa}\left(  sx\right)
\cos\widehat{\kappa}y\cos\alpha_{0}\left(  s,t\right)  }{\sin\widehat{\kappa
}f_{s,t}}+\mathcal{O}\left(  \xi^{2}\right)  \leq1. \label{hst_ineq}%
\end{equation}

So, if $K>0$, we get:%
\[
\sin\kappa y-\kappa\left(  sx\right)  \cos\kappa y\cos\alpha_{0}\left(
s,t\right)  \leq\sin\kappa f_{s,t}+\mathcal{O}\left(  \xi^{2}\right)
\]
and if $K<0$, we get:%
\[
\sinh\kappa y-\kappa\left(  sx\right)  \cosh\kappa y\cos\alpha_{0}\left(
s,t\right)  \leq\sinh\kappa f_{s,t}+\mathcal{O}\left(  \xi^{2}\right)  .
\]
Now, by writing $\cos\kappa f_{s,t}=\sqrt{1-\sin^{2}\kappa f_{s,t}}$ if $K>0$
and $\cosh\kappa f_{s,t}=\sqrt{1+\sinh^{2}\kappa f_{s,t}}$ if $K<0$, it is not
difficult to derive the inequalities of (i) and (ii) of the lemma for
$f_{s,t}$.

\textbf{II}. \textbf{Let }$\left(  \mathcal{M},\rho\right)  $\textbf{\ satisfy
the lower four point }$\operatorname{cosq}_{K}$\textbf{\ condition}. Set
\[
g_{s,t}=\operatorname{cosq}_{K}\left(  \overrightarrow{X_{s}C},\overrightarrow
{Y_{t}A}\right)  .
\]
Then\newline%
\begin{align*}
g_{s,t}  &  =\frac{\left(  1+\cos\widehat{\kappa}z_{s,t}\right)  \left(
\cos\widehat{\kappa}y+\cos\widehat{\kappa}f_{s,t}\cos\widehat{\kappa
}ty\right)  }{\left(  1+\cos\widehat{\kappa}z_{s,t}\right)  \sin
\widehat{\kappa}f_{s,t}\sin\widehat{\kappa}ty}\\
&  -\frac{\left[  \cos\widehat{\kappa}f_{s,t}+\cos\widehat{\kappa}\left(
1-t\right)  y\right]  \left(  \cos\widehat{\kappa}ty+\cos\widehat{\kappa
}sx\right)  }{\left(  1+\cos\widehat{\kappa}z_{s,t}\right)  \sin
\widehat{\kappa}f_{s,t}\sin\widehat{\kappa}ty}\geq-1.
\end{align*}
Let $I$ denote the numerator of $g_{s,t}$. We have:%
\begin{align*}
I  &  =\left[  2-\frac{1}{2}\widehat{\kappa}^{2}z_{s,t}^{2}+\mathcal{O}\left(
\xi^{4}\right)  \right]  \left\{  \cos\widehat{\kappa}y+\cos\widehat{\kappa
}f_{s,t}\left[  1-\frac{1}{2}\widehat{\kappa}^{2}\left(  ty\right)
^{2}+\mathcal{O}\left(  \xi^{4}\right)  \right]  \right\} \\
&  -\left[  \cos\widehat{\kappa}f_{s,t}+\cos\widehat{\kappa}y+\widehat{\kappa
}\left(  ty\right)  \sin\widehat{\kappa}y-\frac{1}{2}\widehat{\kappa}%
^{2}\left(  ty\right)  ^{2}\cos\widehat{\kappa}y+\mathcal{O}\left(  \xi
^{3}\right)  \right]  \times\\
&  \left[  2-\frac{1}{2}\widehat{\kappa}^{2}\left(  ty\right)  ^{2}-\frac
{1}{2}\widehat{\kappa}^{2}\left(  sx\right)  ^{2}+\mathcal{O}\left(  \xi
^{4}\right)  \right]  .
\end{align*}
After elementary simplifications, we get%
\begin{align*}
I  &  =-2\widehat{\kappa}\left(  ty\right)  \sin\widehat{\kappa}%
y+\widehat{\kappa}^{2}\left(  \cos\widehat{\kappa}y+\cos\widehat{\kappa
}f_{s,t}\right)  \left(  sx\right)  \left(  ty\right)  \cos\alpha_{0}\left(
s,t\right) \\
&  +\widehat{\kappa}\left(  ty\right)  ^{2}\left(  \cos\widehat{\kappa}%
y-\cos\widehat{\kappa}f_{s,t}\right)  +\mathcal{O}\left(  \xi^{3}\right)
\text{.}%
\end{align*}
By the triangle inequality, $\left\vert y-f_{s,t}\right\vert \leq sx$. Hence,
$\cos\widehat{\kappa}y-\cos\widehat{\kappa}f_{s,t}=\mathcal{O}\left(
\xi\right)  $. So,
\[
I=-2\widehat{\kappa}\left(  ty\right)  \sin\widehat{\kappa}y+\widehat{\kappa
}^{2}\left(  \cos\widehat{\kappa}y+\cos\widehat{\kappa}f_{s,t}\right)  \left(
sx\right)  \left(  ty\right)  \cos\alpha_{0}\left(  s,t\right)  +\mathcal{O}%
\left(  \xi^{3}\right)  .
\]
Hence,
\begin{align*}
g_{s,t}  &  =\frac{-2\widehat{\kappa}\left(  ty\right)  \sin\widehat{\kappa
}y+\widehat{\kappa}^{2}\left(  \cos\widehat{\kappa}y+\cos\widehat{\kappa
}f_{s,t}\right)  \left(  sx\right)  \left(  ty\right)  \cos\alpha_{0}\left(
s,t\right)  +\mathcal{O}\left(  \xi^{3}\right)  }{2\left[  1+\mathcal{O}%
\left(  \xi^{2}\right)  \right]  \widehat{\kappa}\left(  ty\right)
\sin\widehat{\kappa}f_{s,t}}\\
&  =\frac{-\sin\widehat{\kappa}y+\widehat{\kappa}\frac{\cos\widehat{\kappa
}y+\cos\widehat{\kappa}f_{s,t}}{2}\left(  sx\right)  \cos\alpha_{0}\left(
s,t\right)  }{\sin\widehat{\kappa}f_{s,t}}+\mathcal{O}\left(  \xi^{2}\right)
\geq-1,
\end{align*}
which implies (\ref{hst_ineq}). Hence, the inequalities of (i) and (ii) for
$f_{s,t}$ follow.

Derivation of the inequalities of parts (i) and (ii) for $d_{s,t}$ is similar.

The proof of the cross-diagonal lemma is complete.
\end{proof}

\subsection{Growth estimate lemma\label{Growth_Est}}

We keep the notation of Sec. \ref{Sec_cross_diag_est}. To illustrate the
estimates of Lemma \ref{Lemma_Growth_est}, consider a geodesic triangle
$\mathcal{T}=ABC$ in $\mathbb{S}_{1}$ (for the notation, see Fig.
\ref{fig11}). Let $z_{\perp}$denote the length of the orthogonal projection of
the shortest $\mathcal{BC}$ onto the (possibly extended) shortest
$\mathcal{X}_{s}\mathcal{Y}_{t}$. For small $x$ and $y$, we can treat the
triangle $\mathcal{T}$ as approximately Euclidean triangle. Then it is not
difficult to see that $z_{\perp}$ is approximately equal to $x\cos\beta
_{0}\left(  s,t\right)  +y\cos\gamma_{0}\left(  s,t\right)  $. So, for small
$x$ and $y$, the length $z$ is approximately bounded below by $x\cos\beta
_{0}\left(  s,t\right)  +y\cos\gamma_{0}\left(  s,t\right)  $. Lemma
\ref{Lemma_Growth_est} establishes similar estimates for metric spaces
satisfying the one-sided four point $\operatorname{cosq}_{K}$ condition.

\begin{lemma}
\label{Lemma_Growth_est}Let $K\neq0$ and $0<\underline{m}\leq\overline
{m}<+\infty$. Let $A,B,C$ be three distinct points in a metric space $\left(
\mathcal{M},\rho\right)  $ and $s,t\in(0,1]$ satisfying M1-M3 of Sec.
\ref{Sec_cross_diag_est}. In addition, suppose that $\left(  \mathcal{M}%
,\rho\right)  $ satisfies the one-sided four point $\operatorname{cosq}_{K}$
condition. Let $\mathcal{A}\subseteq(0,1]\times(0,1]$ be such that $\left(
0,0\right)  $ is an accumulation point of the set $\mathcal{A}$ and $0<$
$\underline{m}\leq z_{s,t}/\left(  sx\right)  $ for every $\left(  s,t\right)
\in\mathcal{A}$.\newline(i) If $K>0$, then for every $\left(  s,t\right)
\in\mathcal{A}$,
\[
\sin\kappa y\cos\gamma_{0}\left(  s,t\right)  +\frac{\cos\kappa y+\cos\kappa
z}{1+\cos\kappa x}\sin\kappa x\cos\beta_{0}\left(  s,t\right)  \leq\sin\kappa
z+\mathcal{O}\left(  \xi\right)  ,
\]
\newline(ii) If $K<0$, then for every $\left(  s,t\right)  \in\mathcal{A}$,
\[
\sinh\kappa y\cos\gamma_{0}\left(  s,t\right)  +\frac{\cosh\kappa
y+\cosh\kappa z}{1+\cosh\kappa x}\sinh\kappa x\cos\beta_{0}\left(  s,t\right)
\leq\sinh\kappa z+\mathcal{O}\left(  \xi\right)  ,
\]
where $\mathcal{O}\left(  \xi\right)  =\mathcal{O}_{\lambda,\eta,\underline
{m},\overline{m},K}\left(  \xi\right)  $.
\end{lemma}

\begin{proof}
We consider $\left(  s,t\right)  \in\mathcal{A}$.

\textbf{I}. \textbf{Let }$\left(  \mathcal{M},\rho\right)  $\textbf{\ satisfy
the upper four point }$\operatorname{cosq}_{K}$\textbf{\ condition}. Set%
\[
p_{s,t}=\operatorname{cosq}_{K}\left(  \overrightarrow{X_{s}Y_{t}%
},\overrightarrow{BC}\right)  ,
\]
see Fig. \ref{fig11}. Then\newline%
\begin{align*}
p_{s,t}  &  =\frac{\cos\widehat{\kappa}\left(  1-t\right)  y+\cos
\widehat{\kappa}z_{s,t}\cos\widehat{\kappa}z}{\sin\widehat{\kappa}z_{s,t}%
\sin\widehat{\kappa}z}\\
&  -\frac{\left(  \cos\widehat{\kappa}z_{s,t}+\cos\widehat{\kappa}%
d_{s,t}\right)  \left(  \cos\widehat{\kappa}z+\cos\widehat{\kappa}%
f_{s,t}\right)  }{\left[  1+\cos\widehat{\kappa}\left(  1-s\right)  x\right]
\sin\widehat{\kappa}z_{s,t}\sin\widehat{\kappa}z}\\
&  =\frac{\cos\widehat{\kappa}y+\widehat{\kappa}\left(  ty\right)
\sin\widehat{\kappa}y+\mathcal{O}\left(  \xi^{2}\right)  +\cos\widehat{\kappa
}z\left[  1+\mathcal{O}\left(  \xi^{2}\right)  \right]  }{\sin\widehat{\kappa
}z_{s,t}\sin\widehat{\kappa}z}-\\
&  \frac{\left(  1+\cos\widehat{\kappa}d_{s,t}\right)  \left(  \cos
\widehat{\kappa}z+\cos\widehat{\kappa}f_{s,t}\right)  +\mathcal{O}\left(
\xi^{2}\right)  }{\kappa z_{s,t}\sin\widehat{\kappa}z}\times\\
&  \left[  1+\mathcal{O}\left(  \xi^{2}\right)  \right]  \times\left[
\frac{1}{1+\cos\widehat{\kappa}x}-\frac{\widehat{\kappa}\left(  sx\right)
\sin\widehat{\kappa}x}{\left(  1+\cos\widehat{\kappa}x\right)  ^{2}%
}+\mathcal{O}\left(  \xi^{2}\right)  \right]  .
\end{align*}

For the sake of brevity, set $\mu=\cos\widehat{\kappa}z+\cos\widehat{\kappa}y
$ and $\nu=1+\cos\widehat{\kappa}x$. Let $K>0$. By part (i) of the
cross-diagonal estimate lemma (Lemma \ref{Lemma_cross_diag}),
\begin{align*}
p_{s,t}  &  \geq\frac{\mu+\kappa\left(  ty\right)  \sin\kappa y+\mathcal{O}%
\left(  \xi^{2}\right)  }{\kappa z_{s,t}\sin\kappa z}-\\
&  \frac{\left[  \nu+\kappa\left(  ty\right)  \sin\kappa x\cos\alpha
_{0}\left(  s,t\right)  \right]  \left[  \mu+\kappa\left(  sx\right)
\sin\kappa y\cos\alpha_{0}\left(  s,t\right)  \right]  }{\nu\kappa z_{s,t}%
\sin\kappa z}\times\\
&  \left[  1-\frac{\kappa\left(  sx\right)  \sin\kappa x}{\nu}+\mathcal{O}%
\left(  \xi^{2}\right)  \right]  .
\end{align*}
After elementary simplifications and using the upper four point
$\operatorname{cosq}_{K}$ condition, we get:%
\begin{align}
1  &  \geq p_{s,t}\geq\frac{\left(  ty\right)  \sin\kappa y-\left(  sx\right)
\sin\kappa y\cos\alpha_{0}\left(  s,t\right)  }{z_{s,t}\sin\kappa
z}\label{ineq_growth_upper}\\
&  +\frac{\mu}{\nu}\frac{\left(  sx\right)  \sin\kappa x-\left(  ty\right)
\sin\kappa x\cos\alpha_{0}\left(  s,t\right)  }{z_{s,t}\sin\kappa
z}+\mathcal{O}\left(  \xi\right)  .\nonumber
\end{align}
By recalling that $\cos\alpha_{0}\left(  s,t\right)  =\left[  \left(
sx\right)  ^{2}+\left(  ty\right)  ^{2}-z_{s,t}^{2}\right]  /\left[  2\left(
ty\right)  \left(  sx\right)  \right]  $, we readily see that
\begin{align*}
\left(  ty\right)  \sin\kappa y-\left(  sx\right)  \sin\kappa y\cos\alpha
_{0}\left(  s,t\right)   &  =z_{s,t}\sin\kappa y\cos\gamma_{0}\left(
s,t\right)  ,\\
\left(  sx\right)  \sin\kappa x-\left(  ty\right)  \sin\kappa x\cos\alpha
_{0}\left(  s,t\right)   &  =z_{s,t}\sin\kappa x\cos\beta_{0}\left(
s,t\right)  \text{.}%
\end{align*}
Finally, we get:
\[
\frac{\sin\kappa y\cos\gamma_{0}\left(  s,t\right)  +\frac{\cos\kappa
y+\cos\kappa z}{1+\cos\kappa x}\sin\kappa x\cos\beta_{0}\left(  s,t\right)
}{\sin\kappa z}\leq1+\mathcal{O}\left(  \xi\right)  ,
\]
and the inequality of part (i) follows. The case of negative $K$ is similar
and we leave it to the reader.

\textbf{II}. \textbf{Let }$\left(  \mathcal{M},\rho\right)  $\textbf{\ satisfy
the lower four point }$\operatorname{cosq}_{K}$\textbf{\ condition}. Set
\[
q_{s,t}=\operatorname{cosq}_{K}\left(  \overrightarrow{Y_{t}X_{s}%
},\overrightarrow{BC}\right)  .
\]
Then\newline%
\begin{align*}
q_{s,t}  &  =\frac{\cos\widehat{\kappa}f_{s,t}+\cos\widehat{\kappa}z_{s,t}%
\cos\widehat{\kappa}z}{\sin\widehat{\kappa}z_{s,t}\sin\widehat{\kappa}z}-\\
&  \frac{\left[  \cos\widehat{\kappa}z_{s,t}+\cos\widehat{\kappa}\left(
1-s\right)  x\right]  \left[  \cos\widehat{\kappa}z+\cos\widehat{\kappa
}\left(  1-t\right)  y\right]  }{\left(  1+\cos\widehat{\kappa}d_{s,t}\right)
\sin\widehat{\kappa}z_{s,t}\sin\widehat{\kappa}z}\\
&  =\left(  1+\cos\widehat{\kappa}d_{s,t}\right)  \left\{  \left[
\cos\widehat{\kappa}f_{s,t}+\cos\widehat{\kappa}z+\mathcal{O}\left(  \xi
^{2}\right)  \right]  -\right. \\
&  \left[  1+\cos\widehat{\kappa}x+\widehat{\kappa}\left(  sx\right)
\sin\widehat{\kappa}x+\mathcal{O}\left(  \xi^{2}\right)  \right]  \times\\
&  \left[  \cos\widehat{\kappa}z+\cos\widehat{\kappa}y+\widehat{\kappa}\left(
ty\right)  \sin\widehat{\kappa}y+\mathcal{O}\left(  \xi^{2}\right)  \right]
\times\\
&  \left.  \left[  \left(  1+\cos\widehat{\kappa}d_{s,t}\right)  \sin
\widehat{\kappa}z_{s,t}\sin\widehat{\kappa}z\right]  ^{-1}\right\} \\
&  =\left\{  \left(  1+\cos\widehat{\kappa}d_{s,t}\right)  \left[
\cos\widehat{\kappa}f_{s,t}+\cos\widehat{\kappa}z\right]  \right. \\
&  \left.  -\left[  \mu+\widehat{\kappa}\left(  ty\right)  \sin\widehat
{\kappa}y\right]  \left[  \nu+\widehat{\kappa}\left(  sx\right)  \sin
\widehat{\kappa}x\right]  +\mathcal{O}\left(  \xi^{2}\right)  \right\}
\times\\
&  \left[  \left(  1+\cos\widehat{\kappa}d_{s,t}\right)  \sin\widehat{\kappa
}z_{s,t}\sin\widehat{\kappa}z\right]  ^{-1},
\end{align*}
where we keep the notation $\mu=\cos\widehat{\kappa}y+\cos\widehat{\kappa}z$
and $\nu=1+\cos\widehat{\kappa}x$. By invoking the triangle inequality, we see
that $\cos\widehat{\kappa}d_{s,t}=\cos\widehat{\kappa}x+\mathcal{O}\left(
\xi\right)  $, whence $1/\left(  1+\cos\widehat{\kappa}d_{s,t}\right)
=1/\nu+\mathcal{O}\left(  \xi\right)  $. So, we get:%
\[
q_{s,t}=\frac{I}{\nu\widehat{\kappa}z_{s,t}\sin\widehat{\kappa}z}\left[
1+\mathcal{O}\left(  \xi\right)  \right]  ,
\]
where
\begin{align*}
I  &  =\left(  1+\cos\widehat{\kappa}d_{s,t}\right)  \left(  \cos
\widehat{\kappa}f_{s,t}+\cos\widehat{\kappa}z\right) \\
&  -\left[  \mu+\widehat{\kappa}\left(  ty\right)  \sin\widehat{\kappa
}y\right]  \left[  \nu+\widehat{\kappa}\left(  sx\right)  \sin\widehat{\kappa
}x\right]  +\mathcal{O}\left(  \xi^{2}\right)  .
\end{align*}
Let $K>0$. By the cross-diagonal estimate lemma,
\begin{align*}
I  &  \leq I^{\prime}=\left[  \nu+\kappa\left(  ty\right)  \sin\kappa
x\cos\alpha_{0}\left(  s,t\right)  \right]  \times\left[  \mu+\kappa\left(
sx\right)  \sin\kappa y\cos\alpha_{0}\left(  s,t\right)  \right] \\
&  -\left[  \mu+\kappa\left(  ty\right)  \sin\kappa y\right]  \left[
\nu+\kappa\left(  sx\right)  \sin\kappa x\right]  +\mathcal{O}\left(  \xi
^{2}\right) \\
&  =\kappa\left\{  -\nu\sin\kappa y\left[  \left(  ty\right)  -\left(
sx\right)  \cos\alpha_{0}\left(  s,t\right)  \right]  -\right. \\
&  \left.  \mu\sin\kappa x\left[  \left(  sx\right)  -\left(  ty\right)
\cos\alpha_{0}\left(  s,t\right)  \right]  +\mathcal{O}\left(  \xi^{2}\right)
\right\}  ,
\end{align*}
whence by invoking the lower four point $\operatorname{cosq}_{K}$ condition,
we have:%
\begin{align*}
&  \frac{-\nu\sin\kappa y\left[  \left(  ty\right)  -\left(  sx\right)
\cos\alpha_{0}\left(  s,t\right)  \right]  -\mu\sin\kappa x\left[  \left(
sx\right)  -\left(  ty\right)  \cos\alpha_{0}\left(  s,t\right)  \right]
}{\nu z_{s,t}\sin\kappa z}\\
&  \geq q_{s,t}\geq-1+\mathcal{O}\left(  \xi\right)  ,
\end{align*}
which is equivalent to inequality (\ref{ineq_growth_upper}). Hence, the
inequality of part (i) of the lemma follows. The case of negative $K$ is similar.

The proof of the growth estimate lemma is complete.
\end{proof}

It is well-known that $\alpha_{K}\left(  s,t\right)  -\alpha_{0}\left(
s,t\right)  ,\beta_{K}\left(  s,t\right)  -\beta_{0}\left(  s,t\right)  $ and
$\gamma_{K}\left(  s,t\right)  -\gamma_{0}\left(  s,t\right)  $ are
$\mathcal{O}\left(  \sigma\left(  AX_{s}Y_{t}\right)  \right)  =\mathcal{O}%
\left(  \xi^{2}\right)  $. Hence, by recalling that $\alpha_{0}\left(
s,t\right)  +\beta_{0}\left(  s,t\right)  +\gamma_{0}\left(  s,t\right)  =\pi
$, we get the following

\begin{corollary}
\label{Coroll_growth_K}Under the hypotheses of the growth estimate lemma
(Lemma \ref{Lemma_Growth_est}), the following inequalities hold:\newline(i) If
$K>0$, then for every $\left(  s,t\right)  \in\mathcal{A}$,
\begin{align*}
&  \frac{\cos\kappa y+\cos\kappa z}{1+\cos\kappa x}\sin\kappa x\cos\beta
_{K}\left(  s,t\right) \\
&  -\sin\kappa y\left[  \cos\left(  \alpha_{K}\left(  s,t\right)  +\beta
_{K}\left(  s,t\right)  \right)  \right] \\
&  \leq\sin\kappa z+\mathcal{O}\left(  \xi\right)  ,
\end{align*}
\newline(ii) If $K<0$, then for every $\left(  s,t\right)  \in\mathcal{A}$,
\begin{align*}
&  \frac{\cosh\kappa y+\cosh\kappa z}{1+\cosh\kappa x}\sinh\kappa x\cos
\beta_{K}\left(  s,t\right) \\
&  -\sinh\kappa y\left[  \cos\left(  \alpha_{K}\left(  s,t\right)  +\beta
_{K}\left(  s,t\right)  \right)  \right] \\
&  \leq\sinh\kappa z+\mathcal{O}\left(  \xi\right)  ,
\end{align*}
where $\mathcal{O}\left(  \xi\right)  =\mathcal{O}_{\lambda,\eta,\underline
{m},\overline{m},K}\left(  \xi\right)  $.
\end{corollary}

\subsection{Existence of proportional angles\label{Prop_Angles}}

Let $\left(  \mathcal{M},\rho\right)  $ be a metric space, let $\mathcal{L=AB}
$ and $\mathcal{N=AC}$ be shortests in $\left(  \mathcal{M},\rho\right)  $,
starting at a common point $A\in\mathcal{M}$. Let $K\in%
%TCIMACRO{\U{211d} }%
%BeginExpansion
\mathbb{R}
%EndExpansion
$ and $t\in(0,1]$. Set $\alpha_{K}\left(  t\right)  =\alpha_{K}\left(
t,t\right)  ,$ $\beta_{K}\left(  t\right)  =\beta_{K}\left(  t,t\right)  $,
$\gamma_{K}\left(  t\right)  =\gamma_{K}\left(  t,t\right)  $ and
$z_{t}=z_{t,t} $. In this section, we derive from the growth estimate lemma
that the proportional angle $\lim_{t\rightarrow0+}\alpha_{0}\left(  t\right)
$ exists. We begin with the following

\begin{lemma}
\label{Lemma_z_tau_1}Let $K\neq0$, $\underline{m}>0$ and $\left(
\mathcal{M},\rho\right)  $ be a metric space satisfying the one-sided four
point $\operatorname{cosq}_{K}$ condition. Also, suppose that
$\operatorname{diam}\left(  \mathcal{M}\right)  \leq\pi/\left(  2\sqrt
{K}\right)  $ when $K>0$. Let $\mathcal{L=AB}$, $\mathcal{N=AC}$ be shortests
in $\left(  \mathcal{M},\rho\right)  $ starting at a common point
$A\in\mathcal{M}$ and $t\in(0,1]$. If $0<$ $\underline{m}\leq z_{t}/t$ for
$0<t<\varepsilon$ for some $\varepsilon\in\left(  0,1\right)  $, then there is
$\varepsilon^{\prime}\in(0,\varepsilon]$ such that for every $\tau\in\left(
0,t^{2}\right)  \cap\left(  0,\varepsilon^{\prime}\right)  $, the following
inequality holds:
\[
z_{\tau}\leq\frac{\tau}{t}\left(  z_{t}+\mu t^{2}\right)  ,
\]
where $\mu=\mu\left(  \lambda,\eta,\underline{m},K\right)  >0$.
\end{lemma}

\begin{proof}
The notation of the lemma is illustrated in Fig. \ref{fig12}.
%TCIMACRO{\FRAME{ftphFU}{2.0202in}{1.8246in}{0pt}{\Qcb{Sketch for Lemma
%\ref{Lemma_z_tau_1}}}{\Qlb{fig12}}{fig12.eps}%
%{\special{ language "Scientific Word";  type "GRAPHIC";
%maintain-aspect-ratio TRUE;  display "PICT";  valid_file "F";
%width 2.0202in;  height 1.8246in;  depth 0pt;  original-width 2.1364in;
%original-height 1.927in;  cropleft "0";  croptop "1";  cropright "1";
%cropbottom "0";  filename 'Fig12.eps';file-properties "XNPEU";}} }%
%BeginExpansion
\begin{figure}
[pth]
\begin{center}
\includegraphics[
natheight=1.927000in,
natwidth=2.136400in,
height=1.8246in,
width=2.0202in
]%
{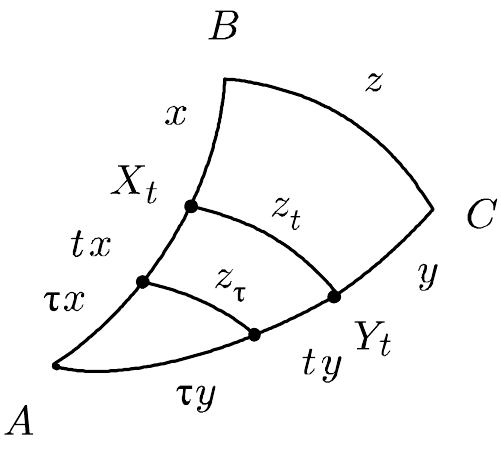}%
\caption{Sketch for Lemma \ref{Lemma_z_tau_1}}%
\label{fig12}%
\end{center}
\end{figure}
%EndExpansion
Let $0<\tau<t^{2}\leq t\leq1$. In the growth estimate lemma, take $x:=tx$ and
$y:=ty$. Then $\xi:=$ $\xi_{t}=\max\left\{  tx,ty\right\}  \frac{\tau}%
{t}=\frac{\tau}{t}t\lambda=\tau\lambda$. Hence, $\mathcal{O}\left(  \xi
_{t}\right)  =\mathcal{O}\left(  \tau\right)  $. By the growth estimate lemma
applied to the shortests $\mathcal{AX}_{t}$ and $\mathcal{AY}_{t}$, if $K>0$,
then
\begin{align}
&  \sin\kappa ty\cos\gamma_{0}\left(  \tau\right)  +\frac{\cos\kappa
ty+\cos\kappa z_{t}}{1+\cos\kappa tx}\sin\kappa tx\cos\beta_{0}\left(
\tau\right) \label{K_pos_tau}\\
&  \leq\sin\kappa z_{t}+\mathcal{O}\left(  \tau\right)  ,\nonumber
\end{align}
\newline and if $K<0$, then
\begin{align*}
&  \sinh\kappa ty\cos\gamma_{0}\left(  \tau\right)  +\frac{\cosh\kappa
ty+\cosh\kappa z_{t}}{1+\cosh\kappa tx}\sinh\kappa tx\cos\beta_{0}\left(
\tau\right) \\
&  \leq\sinh\kappa z_{t}+\mathcal{O}\left(  \tau\right)  ,
\end{align*}
for every $t\in\left(  0,\varepsilon\right)  $, where $\mathcal{O}\left(
\tau\right)  =\mathcal{O}_{\lambda,\eta,\underline{m},K}\left(  \tau\right)  $.

Let $K>0$. Then we can rewrite (\ref{K_pos_tau}) in the following form:%
\begin{align*}
&  \kappa\left(  ty\right)  \left[  1+\mathcal{O}\left(  t^{2}\right)
\right]  \cos\gamma_{0}\left(  \tau\right)  +\left[  1+\mathcal{O}\left(
t^{2}\right)  \right]  \kappa\left(  tx\right)  \cos\beta_{0}\left(
\tau\right) \\
&  \leq\kappa z_{t}+\mathcal{O}\left(  t^{3}\right)  +\mathcal{O}\left(
\tau\right)  =\kappa z_{t}+\mathcal{O}\left(  t^{2}\right)  \text{,}%
\end{align*}
whence,%
\begin{equation}
y\cos\gamma_{0}\left(  \tau\right)  +x\cos\beta_{0}\left(  \tau\right)
\leq\frac{z_{t}+\mathcal{O}\left(  t^{2}\right)  }{t}. \label{tau_ineq1}%
\end{equation}
Let $\eta_{\tau}=z_{\tau}/\tau$. Recall that
\begin{align*}
\cos\gamma_{0}\left(  \tau\right)   &  =\frac{\tau^{2}y^{2}+z_{\tau}^{2}%
-\tau^{2}x^{2}}{2\tau yz_{\tau}}=\frac{y^{2}+\eta_{\tau}^{2}-x^{2}}%
{2y\eta_{\tau}},\\
\cos\beta_{0}\left(  \tau\right)   &  =\frac{\tau^{2}x^{2}+z_{\tau}^{2}%
-\tau^{2}y^{2}}{2\tau xz_{\tau}}=\frac{x^{2}+\eta_{\tau}^{2}-y^{2}}%
{2x\eta_{\tau}}.
\end{align*}
Hence, by (\ref{tau_ineq1}),
\[
\eta_{\tau}\leq\frac{z_{t}+\mathcal{O}\left(  t^{2}\right)  }{t},
\]
and the claim of the lemma for positive $K$ follows. The case of negative $K$
is similar.

The proof of Lemma \ref{Lemma_z_tau_1} is complete.
\end{proof}

By Lemma \ref{Lemma_z_tau_1},
\begin{align*}
\cos\alpha_{0}\left(  \tau\right)   &  =\frac{t^{2}x^{2}+t^{2}y^{2}%
-\frac{t^{2}}{\tau^{2}}z_{\tau}^{2}}{2t^{2}xy}\geq\frac{t^{2}x^{2}+t^{2}%
y^{2}-\left(  z_{t}+\mu t^{2}\right)  ^{2}}{2t^{2}xy}\\
&  =\cos\alpha_{0}\left(  t\right)  -\frac{\mu z_{t}}{xy}-\frac{\mu^{2}t^{2}%
}{2xy}.
\end{align*}
By the triangle inequality, $z_{t}\leq\left(  x+y\right)  t$. So, we have the following

\begin{corollary}
\label{Cor_cos_tau}Under the hypothesis of Lemma \ref{Lemma_z_tau_1}, the
following inequality holds:%
\[
\cos\alpha_{0}\left(  \tau\right)  \geq\cos\alpha_{0}\left(  t\right)
-\mu^{\prime}t\text{,}%
\]
where $\mu^{\prime}=\mu^{\prime}\left(  \lambda,\eta,\underline{m},K\right)
>0 $.
\end{corollary}

\begin{corollary}
\label{Cor_prop_angle}Let $K\neq0$ and $\left(  \mathcal{M},\rho\right)  $ be
a metric space satisfying the one-sided four point $\operatorname{cosq}_{K}$
condition. Also, suppose that $\operatorname{diam}\left(  \mathcal{M}\right)
\leq\pi/\left(  2\sqrt{K}\right)  $ when $K>0$. Let $\mathcal{L=AB}$ and
$\mathcal{N=AC}$ be shortests in $\left(  \mathcal{M},\rho\right)  $ starting
at a common point $A\in\mathcal{M}$. Then $\lim_{t\rightarrow0+}\alpha
_{0}\left(  t\right)  $ exists.
\end{corollary}

\begin{proof}
Let $\overline{\alpha}_{0}=\overline{\lim}_{t\rightarrow0+}\alpha_{0}\left(
t\right)  $ and $\underline{\alpha}_{0}=\underline{\lim}_{t\rightarrow
0+}\alpha_{0}\left(  t\right)  $. Then there are sequences $\left(
t_{n}\right)  _{n=1}^{\infty}$ and $\left(  \tau_{n}\right)  _{n=1}^{\infty}$
in $(0,1]$ convergent to zero such that $\overline{\alpha}_{0}=\lim
_{n\rightarrow\infty}\alpha_{0}\left(  \tau_{n}\right)  $ and $\underline
{\alpha}_{0}=\underline{\lim}_{n\rightarrow\infty}\alpha_{0}\left(
t_{n}\right)  $. There is no restriction in assuming that $\tau_{n}<t_{n}^{2}$
for every $n\in\mathbb{N}$. We consider the following cases.

\textbf{I. }$\lim_{n\rightarrow\infty}z_{\tau_{n}}/\tau_{n}=0$. Then%
\[
\cos\alpha_{0}\left(  \tau_{n}\right)  =\frac{x^{2}+y^{2}-\left(  z_{\tau_{n}%
}/\tau_{n}\right)  ^{2}}{2xy}\rightarrow\frac{x^{2}+y^{2}}{2xy}\text{ as
}n\rightarrow\infty\text{. }%
\]
By the triangle inequality, $z_{\tau_{n}}/\tau_{n}\geq\left\vert
x-y\right\vert $, whence $x=y$, and we have:
\[
\lim_{n\rightarrow\infty}\cos\alpha_{0}\left(  \tau_{n}\right)  =1.
\]
Hence, $\overline{\alpha}_{0}=0$, and
\[
\lim_{t\rightarrow0+}\alpha_{0}\left(  t\right)  =0
\]
follows.

\textbf{II. }By Corollary \ref{Cor_cos_tau}, $\cos\alpha_{0}\left(  \tau
_{n}\right)  \geq\cos\alpha_{0}\left(  t_{n}\right)  +\mathcal{O}\left(
t_{n}\right)  $ for every $n\in\mathbb{N}$. So, by passing to the limit as
$n\rightarrow\infty$ in both sides of the last inequality, we get the
inequality $\overline{\alpha}_{0}\leq\underline{\alpha}_{0}$. This completes
the proof of Corollary \ref{Cor_prop_angle}.
\end{proof}

\subsection{Existence of angle\label{Exist_Angle}}

\begin{proposition}
\label{Prop_Exist_Angle}Let $\left(  \mathcal{M},\rho\right)  $ be a metric
space satisfying the one-sided four point $\operatorname{cosq}_{K}$ condition.
Then between any pair of shortests $\mathcal{L}$ and $\mathcal{N}$ in $\left(
\mathcal{M},\rho\right)  $, starting at a common point $P\in\mathcal{M}$,
there exists Aleksandrov's angle.
\end{proposition}

\begin{proof}
Set $\alpha_{av}=\left(  \overline{\measuredangle}(\mathcal{L},\mathcal{N}%
)+\underline{\measuredangle}\left(  \mathcal{L},\mathcal{N}\right)  \right)
/2$. If $\overline{\measuredangle}(\mathcal{L},\mathcal{N})=0$ or
$\underline{\measuredangle}(\mathcal{L},\mathcal{N})=\pi$, we are done. So, we
can assume that $\sin\alpha_{av}>0$. Contrary to the claim of the proposition,
suppose that $\overline{\measuredangle}(\mathcal{L},\mathcal{N})-\underline
{\measuredangle}(\mathcal{L},\mathcal{N})=\varepsilon_{0}>0$.

\textbf{I. }In Step 1 of the proof of Proposition 20 in \cite{BergNik2008}, we
showed that for every $0<\varepsilon<\varepsilon_{0}$, there are points
$\widetilde{X}$, $X\in\mathcal{L}\backslash\left\{  P\right\}  $ and $Y$,
$\widetilde{Y}\in\mathcal{N}\backslash\left\{  P\right\}  $, or $\widetilde{X}
$, $X\in\mathcal{N}\backslash\left\{  P\right\}  $ and $Y,$ $\widetilde{Y}%
\in\mathcal{L}\backslash\left\{  P\right\}  $ such that the following
conditions are satisfied (for simplicity, we drop $\varepsilon$ from our
notation for these points):

(i) $\widetilde{X}$\ is contained between $X$\ and $P$, and $Y$\ is contained
between $\widetilde{Y}$\ and $P$, as illustrated in Figure \ref{fig13}, and
the points $X,\widetilde{X}$, $\widetilde{Y}$ and $Y$ can be selected
arbitrary close to the point $P$.

(ii) $0\leq\gamma^{\prime\prime}=\measuredangle_{0}\widetilde{X}%
PY<\underline{\measuredangle}\left(  \mathcal{L},\mathcal{N}\right)
+\varepsilon/4.$

(iii) $\gamma^{\prime}=\measuredangle_{0}\widetilde{X}P\widetilde
{Y}>\underline{\measuredangle}\left(  \mathcal{L},\mathcal{N}\right)
-\varepsilon/4.$

(iv) $0\leq\underline{\gamma}=\measuredangle_{0}XP\widetilde{Y}<\underline
{\measuredangle}\left(  \mathcal{L},\mathcal{N}\right)  +\varepsilon/4.$

(v) $\overline{\gamma}=\measuredangle_{0}XPY>\overline{\measuredangle}\left(
\mathcal{L},\mathcal{N}\right)  -\varepsilon/4$\textit{.}

(vi) $x/\widetilde{x}=\widetilde{y}/y$, where $\widetilde{x}=P\widetilde{X}$
and $y=PY$.
%TCIMACRO{\FRAME{ftphFU}{1.2675in}{1.4104in}{0pt}{\Qcb{Sketch for Proposition
%\ref{Prop_Exist_Angle}}}{\Qlb{fig13}}{fig13.eps}%
%{\special{ language "Scientific Word";  type "GRAPHIC";
%maintain-aspect-ratio TRUE;  display "PICT";  valid_file "F";
%width 1.2675in;  height 1.4104in;  depth 0pt;  original-width 1.3293in;
%original-height 1.4824in;  cropleft "0";  croptop "1";  cropright "1";
%cropbottom "0";  filename 'Fig13.eps';file-properties "XNPEU";}} }%
%BeginExpansion
\begin{figure}
[pth]
\begin{center}
\includegraphics[
natheight=1.482400in,
natwidth=1.329300in,
height=1.4104in,
width=1.2675in
]%
{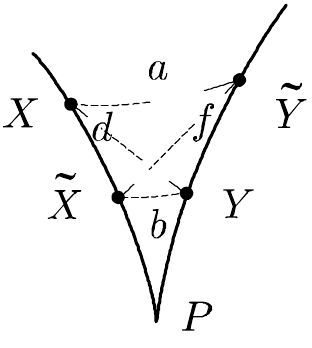}%
\caption{Sketch for Proposition \ref{Prop_Exist_Angle}}%
\label{fig13}%
\end{center}
\end{figure}
%EndExpansion

With little effort, the proof of (i)-(vi) for $K=0$ in \cite{BergNik2008} can
be extended to non-zero $K$. Indeed, by the definition of the lower angle, for
every $\eta>0$, there is $t_{\eta}\in\left(  0,1\right)  $ and $\xi,\zeta
\in\left(  0,t_{\eta}\right)  $ such that
\[
\measuredangle_{0}\left(  \xi,\zeta\right)  <\underline{\measuredangle}\left(
\mathcal{L},\mathcal{N}\right)  +\eta.
\]
By Corollary \ref{Cor_cos_tau},
\[
\cos\measuredangle_{0}\left(  \tau\xi,\tau\zeta\right)  \geq\cos
\measuredangle_{0}\left(  \xi,\zeta\right)  -\mu^{\prime}t,\text{ }%
\]
where $t>0$ is sufficiently small and $0<\tau<t^{2}$. So,
\[
\cos\measuredangle_{0}\left(  \tau\xi,\tau\zeta\right)  \geq\cos\left(
\underline{\measuredangle}\left(  \mathcal{L},\mathcal{N}\right)
+\eta\right)  -\mu^{\prime}t\text{.}%
\]
Hence, given $\varepsilon\in\left(  0,\varepsilon_{0}\right)  $, there is
$t^{\prime}\in\left(  0,1\right)  $ such that the following inequality holds:
\[
\measuredangle_{0}\left(  \tau\xi,\tau\zeta\right)  \leq\underline
{\measuredangle}\left(  \mathcal{L},\mathcal{N}\right)  +\frac{\varepsilon}%
{4},\text{ }0<\tau<t^{^{\prime}}.
\]
After this point, the proof of (i)-(vi) is the same as in Step I of the proof
of Proposition 20 in \cite{BergNik2008}.

Let $\widehat{\gamma}=\max\left\{  \gamma^{\prime\prime},\underline{\gamma
}\right\}  $. By (ii) and (iv), for sufficiently small~positive $\varepsilon$,
the following inequalities hold:

(vii) $\widehat{\gamma}\leq\underline{\measuredangle}\left(  \mathcal{L}%
,\mathcal{N}\right)  +\varepsilon/4<\pi.$

Now consider $I=2\cos\widehat{\gamma}-\left[  \cos\overline{\gamma}+\cos
\gamma^{\prime}\right]  $. By (iii) and (v),%
\begin{align*}
I  &  \geq\cos\left(  \underline{\measuredangle}\left(  \mathcal{L}%
,\mathcal{N}\right)  +\varepsilon/4\right)  -\cos\left(  \overline
{\measuredangle}\left(  \mathcal{L},\mathcal{N}\right)  -\varepsilon/4\right)
\\
&  +\cos\left(  \underline{\measuredangle}\left(  \mathcal{L},\mathcal{N}%
\right)  +\varepsilon/4\right)  -\cos\left(  \underline{\measuredangle}\left(
\mathcal{L},\mathcal{N}\right)  -\varepsilon/4\right)  =\\
&  2\sin\frac{\overline{\measuredangle}\left(  \mathcal{L},\mathcal{N}\right)
-\underline{\measuredangle}\left(  \mathcal{L},\mathcal{N}\right)
-\varepsilon/2}{2}\sin\alpha_{av}-2\sin\frac{\varepsilon}{4}\sin
\underline{\measuredangle}\left(  \mathcal{L},\mathcal{N}\right) \\
&  >2\sin\frac{\varepsilon_{0}}{4}\sin\alpha_{av}-2\sin\frac{\varepsilon}%
{4}\sin\underline{\measuredangle}\left(  \mathcal{L},\mathcal{N}\right)  .
\end{align*}
Hence, for small positive $\varepsilon$, the inequality
\begin{equation}
I=2\cos\widehat{\gamma}-\left[  \cos\overline{\gamma}+\cos\gamma^{\prime
}\right]  >\sin\frac{\varepsilon_{0}}{4}\sin\alpha_{av}>0 \label{Ineq_11}%
\end{equation}
follows.

By Corollary \ref{Cor_uniqueness}, there is no restriction in assuming that
$X\neq\widetilde{Y}$ and $\widetilde{X}\neq Y$. In what follows $t=\widetilde
{x}/x$.

\textbf{II. Let }$\left(  \mathcal{M},\rho\right)  $\textbf{\ satisfy the
upper four point }$\operatorname{cosq}_{K}$\textbf{\ condition}. Set
\[
p=\operatorname{cosq}_{K}\left(  \overrightarrow{X\widetilde{Y}}%
,\overrightarrow{\widetilde{X}Y}\right)  .
\]
Let $f=\widetilde{X}\widetilde{Y}$ and $d=XY$, as shown in Fig. \ref{fig13}.
Then
\begin{align*}
p  &  =\frac{\cos\widehat{\kappa}\left(  1-t\right)  \widetilde{y}%
+\cos\widehat{\kappa}a\cos\widehat{\kappa}b}{\sin\widehat{\kappa}a\sin
\widehat{\kappa}b}-\frac{\left(  \cos\widehat{\kappa}a+\cos\widehat{\kappa
}f\right)  \left(  \cos\widehat{\kappa}b+\cos\widehat{\kappa}d\right)
}{\left[  1+\cos\widehat{\kappa}\left(  1-t\right)  x\right]  \sin
\widehat{\kappa}a\sin\widehat{\kappa}b}\\
&  =\frac{\cos\widehat{\kappa}\widetilde{y}+\widehat{\kappa}\left(  ty\right)
\sin\widehat{\kappa}\widetilde{y}+\cos\widehat{\kappa}a+\mathcal{O}\left(
\lambda^{2}t^{2}\right)  }{\sin\widehat{\kappa}a\sin\widehat{\kappa}b}\\
&  -\frac{\left(  \cos\widehat{\kappa}a+\cos\widehat{\kappa}f\right)  \left[
1+\cos\widehat{\kappa}d+\mathcal{O}\left(  \lambda^{2}t^{2}\right)  \right]
}{\sin\widehat{\kappa}a\sin\widehat{\kappa}b}\times\\
&  \left[  \frac{1}{1+\cos\widehat{\kappa}x}-\frac{\widehat{\kappa}\left(
tx\right)  \sin\widehat{\kappa}}{\left(  1+\cos\widehat{\kappa}x\right)  ^{2}%
}\allowbreak+\mathcal{O}\left(  \lambda^{2}t^{2}\right)  \right]  .
\end{align*}

Let\textbf{\ }$K>0$\textbf{. }Set $\gamma_{K}^{\prime}=\measuredangle
_{K}\widetilde{X}P\widetilde{Y}$ and $\overline{\gamma}_{K}=\measuredangle
_{K}XPY$. By the spherical cosine formula, $\cos\kappa f=\cos\kappa
tx\cos\kappa\widetilde{y}+\sin\kappa tx\sin\kappa\widetilde{y}\cos\gamma
_{K}^{\prime}$. Recall that $\gamma_{K}^{\prime}-\gamma^{\prime}%
=\mathcal{O}\left(  \sigma\left(  \widetilde{X}P\widetilde{Y}\right)  \right)
=\mathcal{O}\left(  \lambda t\right)  $, whence $\cos\gamma_{K}^{\prime}%
=\cos\gamma^{\prime}+\mathcal{O}\left(  \lambda t\right)  $. So, we get:
\begin{equation}
\cos\kappa f=\cos\kappa\widetilde{y}+\kappa\left(  tx\right)  \sin
\kappa\widetilde{y}\cos\gamma^{\prime}+\mathcal{O}\left(  \lambda^{2}%
t^{2}\right)  . \label{diag1}%
\end{equation}
In a similar way,
\begin{equation}
\cos\kappa d=\cos\kappa x+\kappa\left(  t\widetilde{y}\right)  \sin\kappa
x\cos\overline{\gamma}+\mathcal{O}\left(  \lambda^{2}t^{2}\right)  .
\label{diag2}%
\end{equation}
For the sake of brevity, set $\mu=\cos\kappa a+\cos\kappa\widetilde{y}$ and
$\nu=1+\cos\kappa x$. By (\ref{diag1}), (\ref{diag2}) and by invoking the
upper four point $\operatorname{cosq}_{K}$ condition, we get:
\begin{align}
p  &  =\frac{\mu+\kappa\left(  t\widetilde{y}\right)  \sin\kappa\widetilde
{y}+\mathcal{O}\left(  \lambda^{2}t^{2}\right)  }{\sin\widehat{\kappa}%
a\sin\widehat{\kappa}b}-\left[  \mu+\kappa\left(  tx\right)  \sin
\kappa\widetilde{y}\cos\gamma^{\prime}+\mathcal{O}\left(  \lambda^{2}%
t^{2}\right)  \right]  \times\nonumber\\
&  \frac{\left[  \nu+\kappa\left(  t\widetilde{y}\right)  \sin\kappa
x\cos\overline{\gamma}+\mathcal{O}\left(  \lambda^{2}t^{2}\right)  \right]
\left[  1-\frac{\kappa\left(  tx\right)  \sin\kappa x}{\nu}\allowbreak
+\mathcal{O}\left(  \lambda^{2}t^{2}\right)  \right]  }{\nu\sin\kappa
a\sin\kappa b}\nonumber\\
&  =\kappa t\frac{\sin\kappa\widetilde{y}\left(  \widetilde{y}-x\cos
\gamma^{\prime}\right)  +\frac{\mu}{\nu}\sin\kappa x\left(  x-\widetilde
{y}\cos\overline{\gamma}\right)  +\mathcal{O}\left(  \lambda^{2}t\right)
}{\sin\kappa a\sin\kappa b}\leq1. \label{est_an_ex}%
\end{align}
Now we approximate (\ref{est_an_ex}) w.r.t. $x$ and $\widetilde{y}$:
\begin{align*}
p  &  =\kappa^{2}t\frac{\widetilde{y}\left(  \widetilde{y}-x\cos\gamma
^{\prime}\right)  +x\left(  x-\widetilde{y}\cos\overline{\gamma}\right)
+\mathcal{O}\left(  t\lambda^{2}\right)  +\mathcal{O}\left(  \lambda
^{4}\right)  }{\sin\kappa a\sin\kappa b}=\\
&  \kappa^{2}t\frac{x^{2}+\widetilde{y}^{2}-x\widetilde{y}\left(
\cos\overline{\gamma}+\cos\gamma^{\prime}\right)  +\mathcal{O}\left(
t\lambda^{2}\right)  +\mathcal{O}\left(  \lambda^{4}\right)  }{\sin\kappa
a\sin\kappa b}.
\end{align*}
Let $A=x^{2}+\widetilde{y}^{2}-x\widetilde{y}\left(  \cos\overline{\gamma
}+\cos\gamma^{\prime}\right)  +\mathcal{O}\left(  t\lambda^{2}\right)
+\mathcal{O}\left(  \lambda^{4}\right)  $ and $B=x^{2}+\widetilde{y}%
^{2}-2x\widetilde{y}\cos\widehat{\gamma}$. Notice that by (\ref{Ineq_11}),
\begin{align}
A  &  >B+x\widetilde{y}\sin\frac{\varepsilon_{0}}{4}\sin\alpha_{av}%
+\mathcal{O}\left(  t\lambda^{2}\right)  +\mathcal{O}\left(  \lambda
^{4}\right) \nonumber\\
&  >B+\frac{1}{2}x\widetilde{y}\sin\frac{\varepsilon_{0}}{4}\sin\alpha
_{av}\geq\frac{1}{2}x\widetilde{y}\sin\frac{\varepsilon_{0}}{4}\sin\alpha
_{av}>0 \label{length_esta}%
\end{align}
for sufficiently small $\lambda$ and $t$. Set
\[
a^{\prime}=\sqrt{x^{2}+\widetilde{y}^{2}-2x\widetilde{y}\cos\widehat{\gamma}%
}\text{ and }b^{\prime}=t\sqrt{x^{2}+\widetilde{y}^{2}-2x\widetilde{y}%
\cos\widehat{\gamma}}.
\]
Because $\underline{\gamma},\gamma^{\prime\prime}\leq$ $\widehat{\gamma}$, we
readily see that $a\leq a^{\prime}$ and $b\leq b^{\prime}$. Hence,
\begin{align*}
p  &  \geq k^{2}t\frac{A}{\sin\kappa a^{\prime}\sin\kappa b^{\prime}}%
=t\frac{A}{a^{\prime}b^{\prime}}\left[  1+\mathcal{O}\left(  \lambda
^{2}\right)  \right] \\
&  =\frac{x^{2}+\widetilde{y}^{2}-x\widetilde{y}\left(  \cos\overline{\gamma
}+\cos\gamma^{\prime}\right)  +\mathcal{O}\left(  t\lambda^{2}\right)
+\mathcal{O}\left(  \lambda^{4}\right)  }{x^{2}+\widetilde{y}^{2}%
-2x\widetilde{y}\cos\widehat{\gamma}}.
\end{align*}
So, by invoking the upper four point $\operatorname{cosq}_{K}$ condition,
(\ref{Ineq_11}) and (\ref{length_esta}), for sufficiently small $\lambda$ and
$t$, we get:
\begin{align*}
1  &  <1+\frac{\frac{x\widetilde{y}}{2}\sin\frac{\varepsilon_{0}}{4}\sin
\alpha_{av}}{x^{2}+\widetilde{y}^{2}-2x\widetilde{y}\cos\widehat{\gamma}}\\
&  \leq\frac{x^{2}+\widetilde{y}^{2}-x\widetilde{y}\left(  \cos\overline
{\gamma}+\cos\gamma^{\prime}\right)  +\mathcal{O}\left(  t\lambda^{2}\right)
+\mathcal{O}\left(  \lambda^{4}\right)  }{x^{2}+\widetilde{y}^{2}%
-2x\widetilde{y}\cos\widehat{\gamma}}\leq p\leq1,
\end{align*}
a contradiction. The case of negative $K$ is similar.

\textbf{III. Let }$\left(  \mathcal{M},\rho\right)  $\textbf{\ satisfy the
lower four point }$\operatorname{cosq}_{K}$\textbf{\ condition}. Set
\[
q=\operatorname{cosq}_{K}\left(  \overrightarrow{X\widetilde{Y}}%
,\overrightarrow{Y\widetilde{X}}\right)  .
\]
We have:%
\begin{align*}
q  &  =\frac{\cos\widehat{\kappa}f+\cos\widehat{\kappa}a\cos\widehat{\kappa}%
b}{\sin\widehat{\kappa}a\sin\widehat{\kappa}b}-\\
&  \frac{\left[  \cos\widehat{\kappa}a+\cos\widehat{\kappa}\left(  1-t\right)
\widetilde{y}\right]  \left[  \cos\widehat{\kappa}b+\cos\widehat{\kappa
}\left(  1-t\right)  x\right]  }{\left(  1+\cos\widehat{\kappa}d\right)
\sin\widehat{\kappa}a\sin\widehat{\kappa}b}.
\end{align*}
Approximating $q$ relative to $t$, we get $q=I/\left(  J\sin\widehat{\kappa
}a\sin\widehat{\kappa}b\right)  $, where%
\begin{align*}
I  &  =\left[  \cos\widehat{\kappa}f+\cos\widehat{\kappa}a+\mathcal{O}\left(
\lambda^{2}t^{2}\right)  \right]  \left[  1+\cos\widehat{\kappa}d\right]  -\\
&  \left[  \mu+\widehat{\kappa}t\widetilde{y}\sin\widehat{\kappa}\widetilde
{y}+\mathcal{O}\left(  \lambda^{2}t^{2}\right)  \right]  \left[  \nu
+\widehat{\kappa}\left(  tx\right)  \sin\widehat{\kappa}x+\mathcal{O}\left(
\lambda^{2}t^{2}\right)  \right]  ,\\
J  &  =\left(  1+\cos\widehat{\kappa}d\right)  ,
\end{align*}
and where we set $\mu=\cos\widehat{\kappa}a+\cos\widehat{\kappa}\widetilde{y}
$ and $\nu=1+\cos\widehat{\kappa}x$.

Let $K>0$. By recalling (\ref{diag1}) and (\ref{diag2}), we get:
\begin{align*}
I  &  =\left[  \mu+\kappa\left(  tx\right)  \sin\kappa\widetilde{y}\cos
\gamma^{\prime}\right]  \left[  \nu+\kappa\left(  t\widetilde{y}\right)
\sin\kappa x\cos\overline{\gamma}\right]  -\left[  \mu+\kappa t\widetilde
{y}\sin\widehat{\kappa}\widetilde{y}\right]  \times\\
&  \left[  \nu+\kappa\left(  tx\right)  \sin\kappa x\right]  +\mathcal{O}%
\left(  \lambda^{2}t^{2}\right)  ,J=\left(  1+\cos\kappa d\right)  .
\end{align*}
After simplifications, we have:%
\[
I=-kt\left[  \nu\sin k\widetilde{y}\left(  \widetilde{y}-x\cos\gamma^{\prime
}\right)  +\mu\sin\kappa x\left(  x-\widetilde{y}\cos\overline{\gamma}\right)
+\mathcal{O}\left(  \lambda^{2}t\right)  \right]  .
\]
By (\ref{diag2}), $J^{-1}=\nu^{-1}\left[  1+\mathcal{O}\left(  \lambda
^{2}t\right)  \right]  $. By the lower four point $\operatorname{cosq}_{K}$
condition,
\[
-q=kt\frac{\sin k\widetilde{y}\left(  \widetilde{y}-x\cos\gamma^{\prime
}\right)  +\frac{\mu}{\nu}\sin\kappa x\left(  x-\widetilde{y}\cos
\overline{\gamma}\right)  +\mathcal{O}\left(  \lambda^{2}t\right)  }%
{\sin\kappa a\sin\kappa b}\leq1\text{.}%
\]
So, we derived from the lower four point $\operatorname{cosq}$ condition
inequality (\ref{est_an_ex}). Hence, by using the arguments of part II, we see
that the lower four point $\operatorname{cosq}$ condition also implies
existence of Aleksandrov's angle. The case of negative $K$ is similar.

The proof of Proposition \ref{Prop_Exist_Angle} is complete.
\end{proof}

\subsection{Angle comparison theorem\label{Angle_Comparison}}

We begin with the following identity in the $K$-plane:

\begin{proposition}
\label{Prop_sper_ident}Let $K\neq0$ and $\mathcal{T}=ABC$ \ be a triangle in
$\mathbb{S}_{K}$. Set $x=AB,y=AC,z=BC$, ($x,y,z>0$), $\alpha=\measuredangle
BAC$ and $\beta=\measuredangle ABC$, as illustrated in Fig. \ref{fig14}. Then%
\begin{equation}
\sin\widehat{k}z=\frac{\cos\widehat{k}y+\cos\widehat{k}z}{1+\cos\widehat{k}%
x}\sin\widehat{k}x\cos\beta-\sin\widehat{k}y\cos\left(  \alpha+\beta\right)  .
\label{sper_ident}%
\end{equation}
In particular, if $K>0$, then
\[
\sin kz=\frac{\cos ky+\cos kz}{1+\cos kx}\sin kx\cos\beta-\sin ky\cos\left(
\alpha+\beta\right)
\]
and if $K<0$, then%
\[
\sinh kz=\frac{\cosh ky+\cosh kz}{1+\cosh kx}\sinh kx\cos\beta-\sinh
ky\cos\left(  \alpha+\beta\right)  .
\]%
%TCIMACRO{\FRAME{ftphFU}{1.2167in}{1.1605in}{0pt}{\Qcb{Sketch for Proposition
%\ref{Prop_sper_ident}}}{\Qlb{fig14}}{fig14.eps}%
%{\special{ language "Scientific Word";  type "GRAPHIC";
%maintain-aspect-ratio TRUE;  display "PICT";  valid_file "F";
%width 1.2167in;  height 1.1605in;  depth 0pt;  original-width 1.2748in;
%original-height 1.214in;  cropleft "0";  croptop "1";  cropright "1";
%cropbottom "0";  filename 'Fig14.eps';file-properties "XNPEU";}} }%
%BeginExpansion
\begin{figure}
[pth]
\begin{center}
\includegraphics[
natheight=1.214000in,
natwidth=1.274800in,
height=1.1605in,
width=1.2167in
]%
{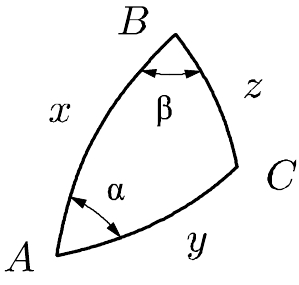}%
\caption{Sketch for Proposition \ref{Prop_sper_ident}}%
\label{fig14}%
\end{center}
\end{figure}
%EndExpansion

\end{proposition}

\begin{proof}
The following cases are possible:

(i) $C$ is between $A$ and $B$. Then $\alpha=\beta=0,x>y$ and $z=x-y$.

(ii) $A$ is between $B$ and $C$. Then $\alpha=\pi,\beta=0$ and $z=x+y$.

(iii) $B$ is between $A$ and $C$. Then $\alpha=0,\beta=\pi,y>x$ and $z=y-x $.

(iv) $\mathcal{T}$ $\ $is a non-degenerate triangle. Then $\alpha,\beta
\in\left(  0,\pi\right)  $.

For example, in case (i), the verification of (\ref{sper_ident}) reduces to
the direct verification of the elementary trigonometric identity%
\[
\sin\widehat{k}\left(  x-y\right)  =\frac{\cos\widehat{k}y+\cos\widehat
{k}\left(  x-y\right)  }{1+\cos\widehat{k}x}\sin\widehat{k}x-\sin\widehat
{k}y.
\]
Cases (ii) and (iii) are similar.

Now we consider case (iv). Let%
\[
I=\frac{\sin\widehat{k}y}{\sin\beta}\sin\left(  \alpha+\beta\right)
=\frac{\sin\widehat{k}y\sin\alpha\cos\beta}{\sin\beta}+\sin\widehat{k}%
y\cos\alpha.
\]
By the sine formula in $\mathbb{S}_{K}$,
\[
\frac{\sin\widehat{k}y\sin\alpha\cos\beta}{\sin\beta}=\sin\widehat{k}%
z\cos\beta.
\]
By the cosine formula in $\mathbb{S}_{K}$,
\[
\cos\beta=\frac{\cos\widehat{k}y-\cos\widehat{k}x\cos\widehat{k}z}%
{\sin\widehat{k}x\sin\widehat{k}z},
\]
whence%
\[
\frac{\sin\widehat{k}y}{\sin\beta}\sin\alpha\cos\beta=\frac{\cos\widehat
{k}y-\cos\widehat{k}x\cos\widehat{k}z}{\sin\widehat{k}x}.
\]
Again, by the cosine formula in $\mathbb{S}_{K}$,%
\[
\cos\alpha=\frac{\cos\widehat{k}z-\cos\widehat{k}x\cos\widehat{k}y}%
{\sin\widehat{k}x\sin\widehat{k}y},
\]
whence%
\[
\sin\widehat{k}y\cos\alpha=\frac{\cos\widehat{k}z-\cos\widehat{k}x\cos
\widehat{k}y}{\sin\widehat{k}x}.
\]
So,
\[
I=\frac{\left(  1-\cos\widehat{k}x\right)  \left(  \cos\widehat{k}%
y+\cos\widehat{k}z\right)  }{\sin\widehat{k}x}=\frac{\cos\widehat{k}%
y+\cos\widehat{k}z}{1+\cos\widehat{k}x}\sin\widehat{k}x,
\]
whence%
\[
\frac{\cos\widehat{k}y+\cos\widehat{k}z}{1+\cos\widehat{k}x}\sin\widehat
{k}x\cos\beta=\frac{\sin\widehat{k}y}{\sin\beta}\sin\left(  \alpha
+\beta\right)  \cos\beta.
\]
Hence, if $J$ denotes the right-hand side of (\ref{sper_ident}), then%
\[
J=\frac{\sin\widehat{k}y}{\sin\beta}\sin\left(  \alpha+\beta\right)  \cos
\beta-\sin\widehat{k}y\cos\left(  \alpha+\beta\right)  .
\]
Recall that by the sine formula in $\mathbb{S}_{K}$, $\sin\widehat{k}%
y=\sin\widehat{k}z\sin\beta/\sin\alpha$. So,%
\[
J=\frac{\sin\widehat{k}z}{\sin\alpha}\left[  \sin\left(  \alpha+\beta\right)
\cos\beta-\cos\left(  \alpha+\beta\right)  \sin\beta\right]  =\sin\widehat
{k}z,
\]
as needed.

The proof of Proposition \ref{Prop_sper_ident} is complete.
\end{proof}

Let $K\neq0$ and let $\left\{  A,B,C\right\}  $ be a triple of distinct points
in a metric space $\left(  \mathcal{M},\rho\right)  $ of diameter less than
$\pi/2\sqrt{K}$ if $K>0$. In what follows, we assume that the points $A $ and
$B$ can be joined by a shortest $\mathcal{L}=\mathcal{AB}$ and the points $A$
and $C$ can be joined by a shortest $\mathcal{N}=\mathcal{AC}$. By Proposition
\ref{Prop_Exist_Angle}, there exists an angle $\alpha$ between the shortests
$\mathcal{L}$ and $\mathcal{N}$. In what follows, we assume that $0<\alpha
\leq\pi$. Set $x=AB$ and $y=AC$.

To state our next lemma, we need \ the following notation. Let $K^{\prime}%
\in\left\{  0,K\right\}  $. Consider a geodesic triangle $\mathcal{T}%
^{K^{\prime}}=\widetilde{A}^{K^{\prime}}\widetilde{B}^{K^{\prime}}%
\widetilde{C}^{K^{\prime}}$ in $\mathbb{S}_{K^{\prime}}$ such that
$\widetilde{A}^{K^{\prime}}\widetilde{B}^{K^{\prime}}=x,\widetilde
{A}^{K^{\prime}}\widetilde{C}^{K^{\prime}}=y$ and $\alpha=\measuredangle
\widetilde{B}^{K^{\prime}}\widetilde{A}^{K^{\prime}}\widetilde{C}^{K^{\prime}%
}$. If $K^{\prime}=K$, set%
\[
\widetilde{A}^{K^{\prime}}=\widetilde{A},\text{ }\widetilde{B}^{K^{\prime}%
}=\widetilde{B},\text{ }\widetilde{C}^{K^{\prime}}=\widetilde{C},\text{
}\widetilde{B}\widetilde{C}=\widetilde{z}\text{ and }\widetilde{\beta
}=\measuredangle\widetilde{A}\widetilde{B}\widetilde{C},
\]
as illustrated in Fig. \ref{fig15}.%
%TCIMACRO{\FRAME{ftphFU}{1.2979in}{2.0119in}{0pt}{\Qcb{Sketch for Lemma
%\ref{Angle_conv_comp_th}}}{\Qlb{fig15}}{fig15.eps}%
%{\special{ language "Scientific Word";  type "GRAPHIC";
%maintain-aspect-ratio TRUE;  display "PICT";  valid_file "F";
%width 1.2979in;  height 2.0119in;  depth 0pt;  original-width 1.3625in;
%original-height 2.1281in;  cropleft "0";  croptop "1";  cropright "1";
%cropbottom "0";  filename 'Fig15.eps';file-properties "XNPEU";}} }%
%BeginExpansion
\begin{figure}
[pth]
\begin{center}
\includegraphics[
natheight=2.128100in,
natwidth=1.362500in,
height=2.0119in,
width=1.2979in
]%
{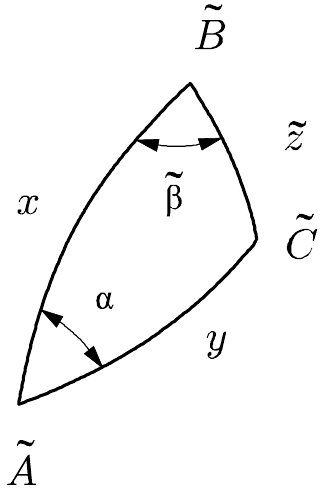}%
\caption{Sketch for Lemma \ref{Angle_conv_comp_th}}%
\label{fig15}%
\end{center}
\end{figure}
%EndExpansion
Suppose that for $t\in(0,1]$, points $\widehat{X}_{t}\in\mathcal{L}%
\backslash\left\{  A\right\}  $ and $\widehat{Y}_{t}\in\mathcal{N}%
\backslash\left\{  A\right\}  $ (in the metric space $\left(  \mathcal{M}%
,\rho\right)  $) have been selected. Consider the Euclidean triangle
$\widetilde{\mathcal{T}}_{t}^{0}=\widetilde{A}^{0}\widetilde{X}_{t}%
^{0}\widetilde{Y}_{t}^{0}$ such that $A\widehat{X}_{t}=\widetilde{A}%
^{0}\widetilde{X}_{t}^{0},A\widehat{Y}_{t}=\widetilde{A}^{0}\widetilde{Y}%
_{t}^{0}$ and $\measuredangle\widetilde{X}_{t}^{0}\widetilde{A}^{0}%
\widetilde{Y}_{t}^{0}=\alpha$. We claim that given small $t\in(0,1]$, there is
$s_{t}\in(0,1]$ such that if $A\widehat{X}_{t}=s_{t}x,A\widehat{Y}_{t}=ty$
(and $\measuredangle\widetilde{X}_{t}^{0}\widetilde{A}^{0}\widetilde{Y}%
_{t}^{0}=\alpha$), then $\measuredangle\widetilde{A}^{0}\widetilde{X}_{t}%
^{0}\widetilde{Y}_{t}^{0}=\widetilde{\beta}$, as illustrated in Fig.
\ref{fig16}.%
%TCIMACRO{\FRAME{ftphFU}{2.3495in}{1.4289in}{0pt}{\Qcb{Definition of $s_{t}$}%
%}{\Qlb{fig16}}{fig16.eps}{\special{ language "Scientific Word";
%type "GRAPHIC";  maintain-aspect-ratio TRUE;  display "PICT";
%valid_file "F";  width 2.3495in;  height 1.4289in;  depth 0pt;
%original-width 2.4888in;  original-height 1.5027in;  cropleft "0";
%croptop "1";  cropright "1";  cropbottom "0";
%filename 'Fig16.eps';file-properties "XNPEU";}} }%
%BeginExpansion
\begin{figure}
[pth]
\begin{center}
\includegraphics[
natheight=1.502700in,
natwidth=2.488800in,
height=1.4289in,
width=2.3495in
]%
{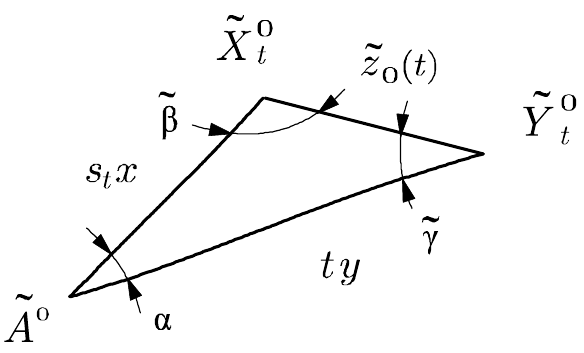}%
\caption{Definition of $s_{t}$}%
\label{fig16}%
\end{center}
\end{figure}
%EndExpansion
Indeed, if $\alpha=\pi$, then $\widetilde{\beta}=0$. Set $s_{t}=t$, and we are
done. Now let $\alpha\in\left(  0,\pi\right)  $. First, we remark that
$\alpha+\widetilde{\beta}<\pi$. \ It is sufficient to consider $K>0$. Let
$\delta=\measuredangle\widetilde{A}\widetilde{C}\widetilde{B}$. Because
$y,\widetilde{z}<\pi/2\sqrt{K}$, we can extend the shortests $\widetilde
{\mathcal{C}}\widetilde{\mathcal{A}}$ and $\widetilde{\mathcal{C}}%
\widetilde{\mathcal{B}}$ to the shortests $\widetilde{\mathcal{C}}%
\mathcal{A}^{\prime}$ and $\widetilde{\mathcal{C}}\mathcal{B}^{\prime}$ of the
lengths $\pi/2\sqrt{K}$. Consider the spherical triangle $\mathcal{T}^{\prime
}=\widetilde{\mathcal{C}}A^{\prime}B^{\prime}$. We have: $\measuredangle
\widetilde{\mathcal{C}}A^{\prime}B^{\prime}=\measuredangle\widetilde
{\mathcal{C}}B^{\prime}A^{\prime}=\pi/2$. Hence, by recalling the Gauss-Bonnet
theorem, we we see that
\[
\delta+\alpha+\widetilde{\beta}<\delta+\frac{\pi}{2}+\frac{\pi}{2}\text{,}%
\]
whence $\alpha+\widetilde{\beta}<\pi$ follows. In particular, $\alpha
\in\left(  0,\pi\right)  $, and setting $\widetilde{\gamma}=\pi-\alpha
-\widetilde{\beta}$, we see that $\widetilde{\gamma}\in\left(  0,\pi\right)
$. Hence, we select $s_{t}=ty\sin\widetilde{\gamma}/\left(  x\sin
\widetilde{\beta}\right)  $.

Finally, set
\begin{align*}
\widehat{\alpha}_{K^{\prime}}\left(  t\right)   &  =\measuredangle_{K^{\prime
}}\widehat{X}_{t}^{K^{\prime}}A^{K^{\prime}}\widehat{Y}_{t}^{K^{\prime}%
},\widehat{\beta}_{K^{\prime}}\left(  t\right)  =\measuredangle_{K^{\prime}%
}A^{K^{\prime}}\widehat{X}_{t}^{K^{\prime}}\widehat{Y}_{t}^{K^{\prime}},\\
\widehat{\gamma}_{K^{\prime}}\left(  t\right)   &  =\measuredangle_{K^{\prime
}}A^{K^{\prime}}\widehat{Y}_{t}^{K^{\prime}}\widehat{X}_{t}^{K^{\prime}}\text{
and }z\left(  t\right)  =\widehat{X}_{t}\widehat{Y}_{t}\text{,}%
\end{align*}
as shown in Fig. \ref{fig17}.%
%TCIMACRO{\FRAME{ftphFU}{1.7702in}{1.3099in}{0pt}{\Qcb{Sketch for Lemma
%\ref{Angle_conv_comp_th}}}{\Qlb{fig17}}{fig17.eps}%
%{\special{ language "Scientific Word";  type "GRAPHIC";
%maintain-aspect-ratio TRUE;  display "PICT";  valid_file "F";
%width 1.7702in;  height 1.3099in;  depth 0pt;  original-width 1.868in;
%original-height 1.3754in;  cropleft "0";  croptop "1";  cropright "1";
%cropbottom "0";  filename 'Fig17.eps';file-properties "XNPEU";}} }%
%BeginExpansion
\begin{figure}
[pth]
\begin{center}
\includegraphics[
natheight=1.375400in,
natwidth=1.868000in,
height=1.3099in,
width=1.7702in
]%
{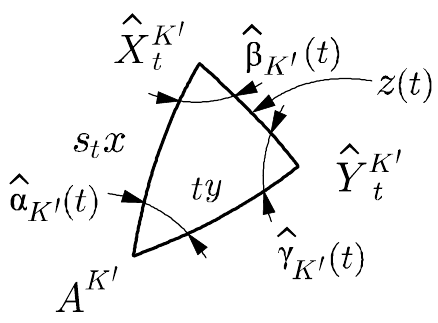}%
\caption{Sketch for Lemma \ref{Angle_conv_comp_th}}%
\label{fig17}%
\end{center}
\end{figure}
%EndExpansion

\begin{lemma}
\label{Angle_conv_comp_th}Let $K\neq0$. If $0<\alpha\leq\pi$, then%
\[
\lim_{t\rightarrow0+}\widehat{\beta}_{K}\left(  t\right)  =\widetilde{\beta}%
\]
(for the notation, see Fig. \ref{fig15} and Fig. \ref{fig17} for $K^{\prime
}=K$).
\end{lemma}

\begin{proof}
\textbf{I. }Let $\alpha=\pi$; then $\widetilde{\beta}=0$. We have:
$\lim_{t\rightarrow0+}\widehat{\alpha}_{0}\left(  t\right)  =\pi$, whence
$\lim_{t\rightarrow0+}\widehat{\beta}_{0}\left(  t\right)  =0$. Because
$\widehat{\beta}_{0}\left(  t\right)  -\widehat{\beta}_{K}\left(  t\right)
=\mathcal{O}\left(  A\widehat{X}_{t}\widehat{Y}_{t}\right)  =\mathcal{O}%
\left(  t^{2}\right)  $, we have: $\lim_{t\rightarrow0+}\widehat{\beta}%
_{K}\left(  t\right)  =0$, as needed.

\textbf{II. }Now let $\alpha\in\left(  0,\pi\right)  $. Then $\widetilde
{\beta},\widetilde{\gamma}\in\left(  0,\pi\right)  $, see Fig. \ref{fig16}. By
the Euclidean sine formula applied to the triangle $\widetilde{X}_{t}%
^{0}\widetilde{A}^{0}\widetilde{Y}_{t}^{0}$,
\[
\sin\widetilde{\beta}=\frac{ty\sin\alpha}{\widetilde{z}_{0}\left(  t\right)
}\text{.}%
\]
By the Euclidean sine formula applied to the triangle $\widehat{X}_{t}%
^{0}\widehat{A}^{0}\widehat{Y}_{t}^{0}$ (see Fig. \ref{fig17} for $K^{\prime
}=0$),%
\[
\sin\widehat{\beta}_{0}\left(  t\right)  =\frac{ty\sin\widehat{\alpha}%
_{0}\left(  t\right)  }{z\left(  t\right)  }.
\]
So, by recalling Proposition \ref{Prop_Exist_Angle} and because $\widehat
{\beta}_{0}\left(  t\right)  -\widehat{\beta}_{K}\left(  t\right)
=\mathcal{O}\left(  t^{2}\right)  $, all we have to do is to show that
$\lim_{t\rightarrow0+}t/z\left(  t\right)  =\lim_{t\rightarrow0+}%
t/\widetilde{z}_{0}\left(  t\right)  $ (in fact, $t/\widetilde{z}_{0}\left(
t\right)  =const$). Indeed, by the Euclidean cosine formula applied to the
triangle $\widetilde{X}_{t}^{0}\widetilde{A}^{0}\widetilde{Y}_{t}^{0}$ and
$\widehat{X}_{t}^{0}\widehat{A}^{0}\widehat{Y}_{t}^{0}$, and by recalling that
$s_{t}=ty\sin\widetilde{\gamma}/\left(  x\sin\widetilde{\beta}\right)  $, we
get:%
\begin{align}
\frac{t}{\widetilde{z}_{0}\left(  t\right)  }  &  =\frac{1}{y}\frac
{\sin\widetilde{\beta}}{\sqrt{\left(  \sin\widetilde{\beta}-\sin
\widetilde{\gamma}\right)  ^{2}+4\sin\widetilde{\beta}\sin\widetilde{\gamma
}\sin^{2}\frac{\alpha}{2}}},\nonumber\\
\frac{t}{z\left(  t\right)  }  &  =\frac{1}{y}\frac{\sin\widetilde{\beta}%
}{\sqrt{\left(  \sin\widetilde{\beta}-\sin\widetilde{\gamma}\right)
^{2}+4\sin\widetilde{\beta}\sin\widetilde{\gamma}\sin^{2}\frac{\widehat
{\alpha}_{0}\left(  t\right)  }{2}}}. \label{ineq_z_over_t}%
\end{align}

By Proposition \ref{Prop_Exist_Angle}, $\lim_{t\rightarrow0+}\widehat{\alpha
}_{0}\left(  t\right)  =\alpha$. Also recall that $\alpha,\widetilde{\beta
},\widetilde{\gamma}\in\left(  0,\pi\right)  $. Hence, $\lim_{t\rightarrow
0+}t/z\left(  t\right)  $ and $\lim_{t\rightarrow0+}t/\widetilde{z}_{0}\left(
t\right)  $ exist and they are equal.

The proof of Lemma \ref{Angle_conv_comp_th} is complete.
\end{proof}

\begin{proposition}
\label{Prop_angle_comp}Let $K\neq0$ and let $\left\{  A,B,C\right\}  $ be a
triple of distinct points in a metric space $\left(  \mathcal{M},\rho\right)
$ such that the points $A$ and $B$ can be joined by a shortest $\mathcal{L}%
=\mathcal{AB}$ and the points $A$ and $C$ can be joined by a shortest
$\mathcal{N}=\mathcal{AC}$, and $AB,AC\leq\pi/\left(  6\sqrt{K}\right)  $ if
$K>0$. If $\left(  \mathcal{M},\rho\right)  $ satisfies the one-sided four
point $\operatorname{cosq}_{K}$ condition, then $\measuredangle BAC\leq
\measuredangle_{K}BAC$.
\end{proposition}

\begin{remark}
In the hypothesis of Proposition \ref{Prop_angle_comp}, we do not require that
$\left(  \mathcal{M},\rho\right)  $ be a geodesically connected metric space.
Also, the bound on $AB$ and $AC$ is not sharp.
\end{remark}

\begin{proof}
Let $\alpha=\measuredangle BAC$ and $\alpha_{K}=\measuredangle_{K}BAC$. There
is no restriction in assuming that $\alpha\in(0,\pi]$. To prove the inequality
$\alpha\leq\alpha_{K}$, we consider a geodesic triangle $\widehat{\mathcal{T}%
}_{t}^{K}=A^{K}\widehat{B}_{t}^{K}\widehat{C}_{t}^{K}$ in $\mathbb{S}_{K}$
such that $A^{K}\widehat{B}_{t}^{K}=x,A^{K}\widehat{C}_{t}^{K}=y$ and
$\measuredangle\widehat{B}_{t}^{K}A^{K}\widehat{C}_{t}^{K}=\widehat{\alpha
}_{K}\left(  t\right)  $. Set $\widehat{z}_{K}\left(  t\right)  =\widehat
{B}_{t}^{K}\widehat{C}_{t}^{K}$, as illustrated in Fig. \ref{fig18}.%
%TCIMACRO{\FRAME{ftphFU}{4.1299in}{2.0073in}{0pt}{\Qcb{Sketch for Proposition
%\ref{Prop_angle_comp}}}{\Qlb{fig18}}{fig18.eps}%
%{\special{ language "Scientific Word";  type "GRAPHIC";
%maintain-aspect-ratio TRUE;  display "PICT";  valid_file "F";
%width 4.1299in;  height 2.0073in;  depth 0pt;  original-width 4.3974in;
%original-height 2.1226in;  cropleft "0";  croptop "1";  cropright "1";
%cropbottom "0";  filename 'Fig18.eps';file-properties "XNPEU";}} }%
%BeginExpansion
\begin{figure}
[pth]
\begin{center}
\includegraphics[
natheight=2.122600in,
natwidth=4.397400in,
height=2.0073in,
width=4.1299in
]%
{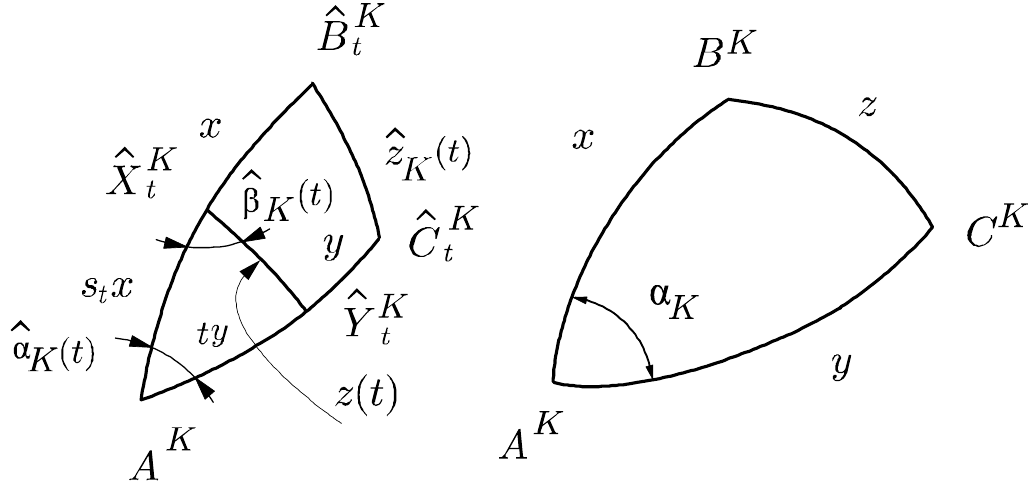}%
\caption{Sketch for Proposition \ref{Prop_angle_comp}}%
\label{fig18}%
\end{center}
\end{figure}
%EndExpansion
It is readily seen that $\alpha\leq\alpha_{K}$ if and only if $\widetilde
{z}=\lim_{t\rightarrow0+}\widehat{z}_{K}\left(  t\right)  \leq z$ (for the
notation, see Fig. \ref{fig15} and Fig. \ref{fig17} for $K^{\prime}=K$). So,
our goal is to derive the inequality $\widetilde{z}\leq z$ .

By Proposition \ref{Prop_Exist_Angle}, $\alpha=\lim_{t\rightarrow0+}%
\widehat{\alpha}_{K}\left(  t\right)  $. It is readily seen that if
$\alpha=\pi$, then $z\left(  t\right)  /t=x+y$, i.e., it is bounded above and
below by positive constants. Let $\alpha\in\left(  0,\pi\right)  $. Because
$\widehat{\alpha}_{0}\left(  t\right)  \rightarrow\alpha$ as $t\rightarrow0+
$,%
\[
\sin\frac{\widehat{\alpha}_{0}\left(  t\right)  }{2}\geq\frac{1}{2}\sin
\frac{\alpha}{2}%
\]
for small $t$. Then by recalling (\ref{ineq_z_over_t}), it is not difficult to
see that
\[
\frac{t}{z\left(  t\right)  }\leq\frac{1}{2y\sqrt{\sin\widetilde{\gamma}}%
\sin\frac{\widehat{\alpha}_{0}\left(  t\right)  }{2}}\leq\frac{1}{y\sqrt
{\sin\widetilde{\gamma}}\sin\frac{\alpha}{2}}<+\infty\text{.}%
\]
So, the hypotheses of Corollary \ref{Coroll_growth_K} are satisfied.

Let $K>0$. By Corollary \ref{Coroll_growth_K},
\begin{align*}
&  \frac{\cos\kappa y+\cos\kappa z}{1+\cos\kappa x}\sin\kappa x\cos
\widehat{\beta}_{K}\left(  t\right)  -\sin\kappa y\cos\left(  \widehat{\alpha
}_{K}\left(  t\right)  +\widehat{\beta}_{K}\left(  t\right)  \right) \\
&  \leq\sin\kappa z+\mathcal{O}\left(  t\right)  ,
\end{align*}
By Proposition \ref{Prop_Exist_Angle}, $\lim_{t\rightarrow0+}\widehat{\alpha
}_{K}\left(  t\right)  =\alpha$ and by Lemma \ref{Angle_conv_comp_th},
$\lim_{t\rightarrow0+}\widehat{\beta}_{K}\left(  t\right)  =\widetilde{\beta}%
$. Let $K>0$. By letting $t\rightarrow0+$, we get%
\begin{align*}
&  \sin\kappa z-\frac{\cos\kappa z}{1+\cos\kappa x}\sin\kappa x\cos
\widetilde{\beta}\\
&  \geq\frac{\cos\kappa y}{1+\cos\kappa x}\sin\kappa x\cos\widetilde{\beta
}-\sin\kappa y\cos\left(  \alpha+\widetilde{\beta}\right)  ,
\end{align*}
By Proposition \ref{Prop_sper_ident},
\[
\sin\kappa\widetilde{z}-\frac{\cos\kappa\widetilde{z}}{1+\cos\kappa x}%
\sin\kappa x\cos\widetilde{\beta}=\frac{\cos\kappa y}{1+\cos\kappa x}%
\sin\kappa x\cos\widetilde{\beta}-\sin\kappa y\cos\left(  \alpha
+\widetilde{\beta}\right)  ,
\]
whence%
\begin{align}
&  \sin\kappa z-\frac{\cos\kappa z}{1+\cos\kappa x}\sin\kappa x\cos
\widetilde{\beta}\nonumber\\
&  \geq\sin\kappa\widetilde{z}-\frac{\cos\kappa\widetilde{z}}{1+\cos\kappa
x}\sin\kappa x\cos\widetilde{\beta}. \label{ineq_pos}%
\end{align}
By the triangle inequality, $z,\widetilde{z}\leq\pi/\left(  3\kappa\right)  $.
By Corollary \ref{Cor_uniqueness}, there is no restriction in assuming that
$z>0$. So, we can also assume that $\widetilde{z}$ is also positive. Consider
the function%
\[
f\left(  u\right)  =\sin\kappa u-\frac{\cos\kappa u}{1+\cos\kappa x}\sin\kappa
x\cos\widetilde{\beta},u\in(0,\frac{\pi}{3\kappa}]\text{.}%
\]
It is readily seen that $f\left(  u\right)  $ is a strictly increasing
function if $u\in(0,\pi/\left(  3\kappa\right)  ]$. So, the inequality
$\widetilde{z}\leq z$ for positive $K$ follows from inequality (\ref{ineq_pos}%
), as needed.

In a similar way, for $K<0$, we have:%
\begin{align}
&  \sinh\kappa z-\frac{\cosh\kappa z}{1+\cosh\kappa x}\sinh\kappa
x\cos\widetilde{\beta}\nonumber\\
&  \geq\sinh\kappa\widetilde{z}-\frac{\cosh\kappa\widetilde{z}}{1+\cosh\kappa
x}\sinh\kappa x\cos\widetilde{\beta}. \label{ineq_neg}%
\end{align}
It is easy to see that the function
\[
g\left(  u\right)  =\sinh\kappa u-\frac{\cosh\kappa u}{1+\cosh\kappa x}%
\sinh\kappa x\cos\widetilde{\beta},u\in\left(  0,+\infty\right)
\]
is an increasing function if $u\in\left(  0,+\infty\right)  $. Hence,
(\ref{ineq_neg}) implies the inequality $\widetilde{z}\leq z$ for negative
$K$, as claimed.

The proof of Proposition \ref{Prop_angle_comp} is complete.
\end{proof}

\begin{corollary}
\label{cor_final}Let $K>0$ and let $\left(  \mathcal{M},\rho\right)  $ be a
geodesically connected metric space such that $\operatorname{diam}\left(
\mathcal{M}\right)  \leq\pi/\left(  2\sqrt{K}\right)  $ when $K>0$. If
$\left(  \mathcal{M},\rho\right)  $ satisfies the one-sided four point
$\operatorname{cosq}_{K}$ condition, then it is an $\Re_{K}$ domain with the
same diameter restriction.
\end{corollary}

\begin{proof}
Theorem 9 in \cite[\S \ 3]{A1957a} states that a metric space $\left(
\mathcal{M},\rho\right)  $ such that \newline(i) $\left(  \mathcal{M}%
,\rho\right)  $ is geodesically connected, \newline(ii) the perimeter of every
geodesic triangle in $\left(  \mathcal{M},\rho\right)  $ is less than
$2\pi/\sqrt{K^{\prime}}$ if $K^{\prime}>0$, \newline(iii) every point of
$\left(  \mathcal{M},\rho\right)  $ has a neighborhood which is an
$\Re_{K^{\prime}}$ domain, \newline(iv) shortests in $\left(  \mathcal{M}%
,\rho\right)  $ depend continuously on their end points\newline is an
$\Re_{K^{\prime}}$ domain.

By the hypothesis of Corollary \ref{cor_final}, (i) and (ii) for $K^{\prime
}=K$ are satisfied; (iii) for $K^{\prime}=K$ is satisfied by Proposition
\ref{Prop_angle_comp}, and (iv) is satisfied by Lemma \ref{Lemma_cont_geod}.
Hence, $\left(  \mathcal{M},\rho\right)  $ is an $\Re_{K}$ domain.
\end{proof}

Finally, Theorem \ref{MainTh} follows from Theorem \ref{ThRKBound},
Proposition \ref{Prop_angle_comp} ($K<0$) and Corollary \ref{cor_final} ($K>0
$).

\section{Proof of Theorem \ref{ThExtr} \label{SecThExtr}}

In this section, we consider an extremal case of Theorem \ref{MainTh} when
$\left\vert \operatorname{cosq}_{K}\right\vert =1$. We will need a rigidity
lemma on geodesic convex hulls of quadruples.

In \cite[\S \ 4, Theorem 6]{A1957a}, Aleksandrov established the following
rigidity result: if $\mathcal{T=ABC}$ is a triangle in an $\Re_{K}$ domain and
$\measuredangle ABC=\measuredangle_{K}ABC$, then $BX=B^{K}X^{K}$ for every
$X\in\mathcal{AC}$ and $X^{K}\in\mathcal{A}^{K}\mathcal{C}^{K}$ such that
$AX=A^{K}X^{K}$. Aleksandrov's proof also implies the converse: if
$BX_{0}=B^{K}X_{0}^{K}$ for at least one point $X_{0}\in\mathcal{AC}%
\backslash\left\{  A,C\right\}  $, then $\measuredangle ABC=\measuredangle
_{K}ABC$. In \cite[Proposition 2.9]{BrH1990}, Bridson and Haefliger slightly
improved Aleksandrov's theorem by proving isometry of the convex hulls of the
triangles (see also (1) and (2) of Sec. 2.10 in \cite{BrH1990}). The following
rigidity lemma is close to Aleksandrov's rigidity theorem in its spirit and in
the method of the proof. For completeness, we include the rigidity lemma and
its proof.

\begin{lemma}
\label{Lemma_hull}Let $K\in%
%TCIMACRO{\U{211d} }%
%BeginExpansion
\mathbb{R}
%EndExpansion
$ and let $\mathfrak{Q=}\left\{  A,P,Q,B\right\}  $ be a quadruple of distinct
points in an $\Re_{K}$ domain. Let $\mathcal{R}$ be a convex quadrangle in
$\mathbb{S}_{K}$ bounded by the closed polygonal curve $\mathcal{L}^{\prime
}=\mathcal{A}^{\prime}\mathcal{P}^{\prime}\mathcal{Q}^{\prime}\mathcal{B}%
^{\prime}\mathcal{A}^{\prime}$ with the vertices at $A^{\prime},$ $P^{\prime
},$ $Q^{\prime}$ and $B^{\prime}$. Suppose that there is an isometry $f$ from
$\mathfrak{Q}$ onto the quadruple $\mathfrak{Q}^{\prime}\mathfrak{=\{}%
A^{\prime},$ $P^{\prime},$ $Q^{\prime},$ $B^{\prime}\}$ such that $f\left(
A\right)  =A^{\prime},$ $f\left(  P\right)  =P^{\prime},$ $f\left(  Q\right)
=Q^{\prime}$ and $f\left(  B\right)  =B^{\prime}$. Then the geodesic convex
hull of $\mathfrak{Q}$ is isometric to $\mathcal{R}$.
\end{lemma}

\begin{proof}
The proof will be done in a series of steps.

Let\textit{\ }$\mathcal{L}$ be a polygonal curve $\mathcal{A}_{1}%
\mathcal{A}_{2}\ldots\mathcal{A}_{n}$ in $\Re_{K}$ and $\mathcal{L}^{\prime}$
be a polygonal curve $\mathcal{A}_{1}^{\prime}\mathcal{A}_{2}^{\prime}%
\ldots\mathcal{A}_{n}^{\prime}$ in $\mathbb{S}_{K}$ such that $A_{j}%
A_{j+1}=A_{j}^{\prime}A_{j+1}^{\prime}$ for every $j\in\left\{
1,2,...,n-1\right\}  $. Let $\mathfrak{g}_{al,\mathcal{L}}$, $\mathfrak{g}%
_{al,\mathcal{L}^{\prime}}$ denote the arc length parametrizations of
$\mathcal{L}$ and $\mathcal{L}^{\prime}$ relative to $A_{1}$ and
$A_{1}^{\prime}$, respectively (for the notation, see Sec.
\ref{Al_upper_curv_cond}). Define $\varphi_{\mathcal{L},\mathcal{L}^{\prime}%
}:\mathcal{L}^{\prime}\rightarrow\mathcal{L}$ as follows. If $X^{\prime}%
\in\mathcal{L}^{\prime}$ and $X^{\prime}=\mathfrak{g}_{al,\mathcal{L}^{\prime
}}\left(  t_{0}\right)  $, then set $X=\varphi_{\mathcal{L},\mathcal{L}%
^{\prime}}\left(  X^{\prime}\right)  =\mathfrak{g}_{al,\mathcal{L}}$ $\left(
t_{0}\right)  \in\mathcal{L}$.

\textbf{I. }Let\textit{\ }$\mathcal{T}=ABC$\textit{\ }be a geodesic triangle
in $\Re_{K}$\ of perimeter less than $2\pi/\sqrt{K}$ if $K>0$ and let
$\mathcal{T}^{\prime}=A^{\prime}B^{\prime}C^{\prime}$\ be its isometric copy
in $\mathbb{S}_{K}$\textit{. }If $X$ is a point on the side $\mathcal{A}%
\mathcal{B}$, then $X^{\prime}$ denotes the point on the side $\mathcal{A}%
^{\prime}\mathcal{B}^{\prime}$\textit{\ }such that $BX=B^{\prime}X^{\prime}$.
The point $Y^{\prime}\in\mathcal{B}^{\prime}\mathcal{C}^{\prime}$
corresponding to a point $Y\in\mathcal{BC}$ is defined in a similar way. We
begin with the following corollary of \cite{BrH1990}, Proposition 2.9 and (1),
(2) of Sec. 2.10: \textit{The convex hull }$\mathcal{G}\left[  A,B,C\right]  $
\textit{in} $\Re_{K}$ \textit{is isometric to the convex hull} $\mathcal{G}%
\left[  A^{\prime},B^{\prime},C^{\prime}\right]  $\textit{\ in }%
$\mathbb{S}_{K}$ \textit{if and only if there is }$X\in\mathcal{AB}%
\backslash\left\{  B\right\}  $ and $Y\in\mathcal{BC}\backslash\left\{
B\right\}  $ \textit{such that} $XY=X^{\prime}Y^{\prime}$ \textit{where either
}$X\neq A$ or $Y\neq C$\textit{. }

\textbf{II. }Let $O^{\prime}$ be the point of intersection of the shortests
$\mathcal{A}^{\prime}\mathcal{Q}^{\prime}$ and $\mathcal{B}^{\prime
}\mathcal{P}^{\prime}$. Let $u=A^{\prime}O^{\prime}$ and $v=O^{\prime
}Q^{\prime}$. Because $AQ=A^{\prime}Q^{\prime}$, we can select $O\in
\mathcal{AQ}$ such that $AO=u$ and $OQ=v$. By the triangle inequality, $PB\leq
PO+OB$. By $K$-concavity (Theorem 2 in \cite[\S \ 3]{A1957a}), $PO\leq
P^{\prime}O^{\prime}$ and $OB\leq O^{\prime}B^{\prime}$. Hence, we have:%
\[
PB\leq PO+OB\leq P^{\prime}O^{\prime}+O^{\prime}B^{\prime}=P^{\prime}%
B^{\prime}=PB,
\]
whence $PB=PO+OB$ follows. By the uniqueness property of shortests in $\Re
_{K}$, the polygonal curve $\mathcal{POB}$ coincides with the shortest
$\mathcal{PB}$. We also have: $PO=P^{\prime}O^{\prime}$ and $OB=O^{\prime
}B^{\prime}$.

\textbf{III. }Consider the closed polygonal curves
\[
\mathcal{L}=\mathcal{PQBAP}\text{ and }\mathcal{L}^{\prime}=\mathcal{P}%
^{\prime}\mathcal{Q}^{\prime}\mathcal{B}^{\prime}\mathcal{A}^{\prime
}\mathcal{P}^{\prime},
\]
and set $\varphi_{P^{\prime}}=\varphi_{\mathcal{L},\mathcal{L}^{\prime}}$.

III$_{\text{a}}$. Let $E^{\prime}\in\mathcal{B}^{\prime}\mathcal{Q}^{\prime}$.
Set $E=\varphi_{P^{\prime}}\left(  E^{\prime}\right)  $. Then $AE=A^{\prime
}E^{\prime}$. Indeed, consider triangle $AQB$. By II, $BO=B^{\prime}O^{\prime
}$. Then by I, $AE=A^{\prime}E^{\prime}$, as needed. In a similar way, all
distances from a point of $\mathfrak{Q}$ to a point on one of the shortests
$\mathcal{AP}$, $\mathcal{PQ}$, $\mathcal{QB}$ and $\mathcal{AB}$ are the same
as the corresponding distances in $\mathbb{S}_{K}$.

III$_{\text{b }}$. Now, let $E^{\prime}\in\mathcal{A}^{\prime}\mathcal{P}%
^{\prime}$ (we can assume that $E^{\prime}\neq P^{\prime}$), $F^{\prime}%
\in\mathcal{P}^{\prime}\mathcal{Q}^{\prime}$, $E=\varphi_{P^{\prime}}\left(
E^{\prime}\right)  $ and $F=\varphi_{P^{\prime}}\left(  F^{\prime}\right)  $.
Consider the triangle $E^{\prime}Q^{\prime}P^{\prime}$. Let $G^{\prime}%
\in\mathcal{P}^{\prime}\mathcal{E}^{\prime}\backslash\left\{  P^{\prime
},E^{\prime}\right\}  $ and $G=\varphi_{P^{\prime}}\left(  G^{\prime}\right)
$. By III$_{\text{a }}$, $QG=Q^{\prime}G^{\prime}$, whence by I,
$EF=E^{\prime}F^{\prime}$ follows.

III$_{\text{c}}$. Next, let $E^{\prime}\in\mathcal{A}^{\prime}\mathcal{P}%
^{\prime}$, $F^{\prime}\in\mathcal{B}^{\prime}\mathcal{Q}^{\prime}$,
$E=\varphi_{P^{\prime}}\left(  E^{\prime}\right)  $ and $F=\varphi_{P^{\prime
}}\left(  F^{\prime}\right)  $. Let $O^{\prime}$ be the point of intersection
of the shortest $\mathcal{A}^{\prime}\mathcal{Q}^{\prime}$ and $\mathcal{E}%
^{\prime}\mathcal{B}^{\prime}$. Recall that by III$_{\text{a}}$, $E^{\prime
}Q^{\prime}=EQ$ and $\mathcal{EB=E}^{\prime}\mathcal{B}^{\prime}$. There is
$O\in\mathcal{EB}$ such that $EO=E^{\prime}O^{\prime}$ and $OB=O^{\prime
}B^{\prime}$. By employing arguments similar to those of II, we see that
$OQ=O^{\prime}Q^{\prime}$. Hence, by I, applied to triangle $BEQ$, we have:
$EF=E^{\prime}F^{\prime}$.

So, by III, $\varphi_{P^{\prime}}$ is an isometry in \ from $\mathcal{L}%
^{\prime}$ onto $\mathcal{L}$.

\textbf{IV. }The isometry $\varphi_{P^{\prime}}$ from $\mathcal{L}^{\prime}$
onto $\mathcal{L}$ can be extended to an isometry from $\mathcal{R}$ into
$\mathcal{GC}\left[  \mathfrak{Q}\right]  $. Indeed, let $X^{\prime}%
,Y^{\prime}\in\mathcal{R}$. For definiteness, suppose that there are
$D^{\prime}\in\mathcal{A}^{\prime}B^{\prime}$ and $F^{\prime}\in
\mathcal{B}^{\prime}Q^{\prime}$ such that $X^{\prime}\in\mathcal{P}^{\prime
}D^{\prime}$ and $Y^{\prime}\in\mathcal{P}^{\prime}\mathcal{F}^{\prime}$. Let
$D=\varphi_{P^{\prime}}\left(  D^{\prime}\right)  $ and $F=\varphi_{P^{\prime
}}\left(  F^{\prime}\right)  $. By III$_{\text{a}}$, $P^{\prime}D^{\prime}=PD$
and $P^{\prime}F^{\prime}=PF$. Hence, we can select $X\in\mathcal{PD}$ such
that $\mathcal{P}^{\prime}X^{\prime}=PX$ and $X^{\prime}D^{\prime}=XD$. Point
$Y\in\mathcal{PF}$ is selected in a similar way so that $P^{\prime}Y^{\prime
}=PY$, as illustrated in Fig. \ref{fig19}.
%TCIMACRO{\FRAME{ftphFU}{3.5939in}{1.3634in}{0pt}{\Qcb{Sketch for part IV of
%Lemma \ref{Lemma_hull}}}{\Qlb{fig19}}{fig19.eps}%
%{\special{ language "Scientific Word";  type "GRAPHIC";
%maintain-aspect-ratio TRUE;  display "PICT";  valid_file "F";
%width 3.5939in;  height 1.3634in;  depth 0pt;  original-width 3.8227in;
%original-height 1.4307in;  cropleft "0";  croptop "1";  cropright "1";
%cropbottom "0";  filename 'Fig19.eps';file-properties "XNPEU";}} }%
%BeginExpansion
\begin{figure}
[pth]
\begin{center}
\includegraphics[
natheight=1.430700in,
natwidth=3.822700in,
height=1.3634in,
width=3.5939in
]%
{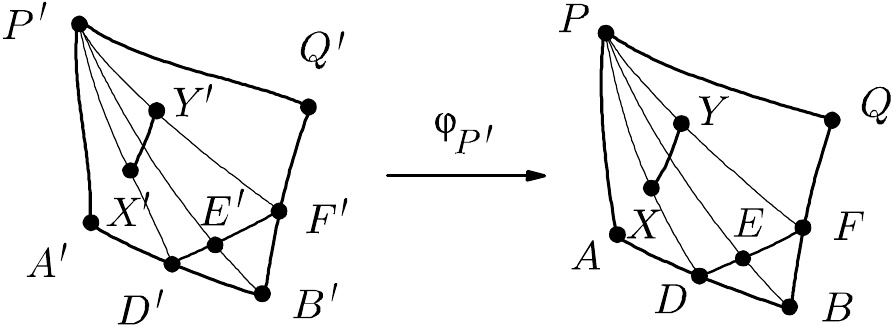}%
\caption{Sketch for part IV of Lemma \ref{Lemma_hull}}%
\label{fig19}%
\end{center}
\end{figure}
%EndExpansion
Set $\varphi_{P^{\prime}}\left(  X^{\prime}\right)  =X$ and $\varphi
_{P^{\prime}}\left(  Y^{\prime}\right)  =Y$. We claim that $X^{\prime
}Y^{\prime}=XY$. Indeed, by III$_{\text{b}}$, $D^{\prime}F^{\prime}=DF$. Let
$E^{\prime}$ be the point of intersection of the shortest $\mathcal{P}%
^{\prime}\mathcal{B}^{\prime}$ and $\mathcal{D}^{\prime}\mathcal{F}^{\prime}$.
Because $DF=D^{\prime}F^{\prime}$, we can select $E\in\mathcal{DF}$ such that
$DE=D^{\prime}E^{\prime}$ and $EF=E^{\prime}F^{\prime}$. By using arguments
similar to those of II, we see that $PE=P^{\prime}E^{\prime}$. Hence, by I,
$\mathcal{GC}\left[  \left\{  D^{\prime},P^{\prime},F^{\prime}\right\}
\right]  $ is isometric to $\mathcal{GC}\left[  \left\{  D,P,F\right\}
\right]  $, and $XY=X^{\prime}Y^{\prime}$ follows. Thus, $\varphi_{P^{\prime}%
}$ is an isometry from $\mathcal{R}$ into $\mathcal{GC}\left[  \mathfrak{Q}%
\right]  $.

\textbf{V}. $\varphi_{P^{\prime}}$ is a surjection. Because $\mathcal{R}$ is
convex it is sufficient to prove the following claim $\mathfrak{P}\left(
n\right)  $: \textit{the isometry} $\varphi_{P^{\prime}}$\textit{\ from}
$\mathcal{G}^{n}\left[  \mathfrak{O}^{\prime}\right]  $ \textit{into}
$\mathcal{G}^{n}\left[  \mathfrak{O}\right]  $ \textit{is a surjection for
every }$n=0,$ $1,2,...$. Indeed, clearly $\mathfrak{P}\left(  0\right)  $ is
true. Suppose that $\mathfrak{P}\left(  n\right)  $ is true. Let $Z\in$
$\mathcal{G}^{n+1}\left[  \mathfrak{O}\right]  $. Then, there are $X,Y\in$
$\mathcal{G}^{n}\left[  \mathfrak{O}\right]  $ such that $Z\in\mathcal{XY}$.
Because $\varphi_{P^{\prime}}:\mathcal{G}^{n}\left[  \mathfrak{O}^{\prime
}\right]  \rightarrow\mathcal{G}^{n}\left[  \mathfrak{O}\right]  $ is a
bijection, there are (unique) $X^{\prime}=\varphi_{P^{\prime}}^{-1}\left(
X\right)  $, $Y^{\prime}=\varphi_{P^{\prime}}^{-1}\left(  Y\right)  $
satisfying $XY=X^{\prime}Y^{\prime}$. Without loss of generality, we can
assume that there are $D^{\prime}\in\mathcal{A}^{\prime}\mathcal{B}^{\prime}$
and $F^{\prime}\in\mathcal{B}^{\prime}\mathcal{Q}^{\prime}$ such that
$X^{\prime}\in\mathcal{P}^{\prime}\mathcal{D}^{\prime}$ and $Y^{\prime}%
\in\mathcal{P}^{\prime}\mathcal{F}^{\prime}$. Then, by the definition of
$\varphi_{P^{\prime}}$, we see that $X\in\mathcal{PD}$, $Y\in\mathcal{PF}$,
where $D=$ $\varphi_{P^{\prime}}\left(  D^{\prime}\right)  $ and
$F=\varphi_{P^{\prime}}\left(  F^{\prime}\right)  $, and $P^{\prime}D^{\prime
}=PD,$ $P^{\prime}F^{\prime}=PF$, as illustrated in Fig. \ref{fig20}.%
%TCIMACRO{\FRAME{ftphFU}{3.6354in}{1.5959in}{0pt}{\Qcb{Sketch for part V of
%Lemma \ref{Lemma_hull}}}{\Qlb{fig20}}{fig20.eps}%
%{\special{ language "Scientific Word";  type "GRAPHIC";
%maintain-aspect-ratio TRUE;  display "PICT";  valid_file "F";
%width 3.6354in;  height 1.5959in;  depth 0pt;  original-width 3.8688in;
%original-height 1.6817in;  cropleft "0";  croptop "1";  cropright "1";
%cropbottom "0";  filename 'Fig20.eps';file-properties "XNPEU";}} }%
%BeginExpansion
\begin{figure}
[pth]
\begin{center}
\includegraphics[
natheight=1.681700in,
natwidth=3.868800in,
height=1.5959in,
width=3.6354in
]%
{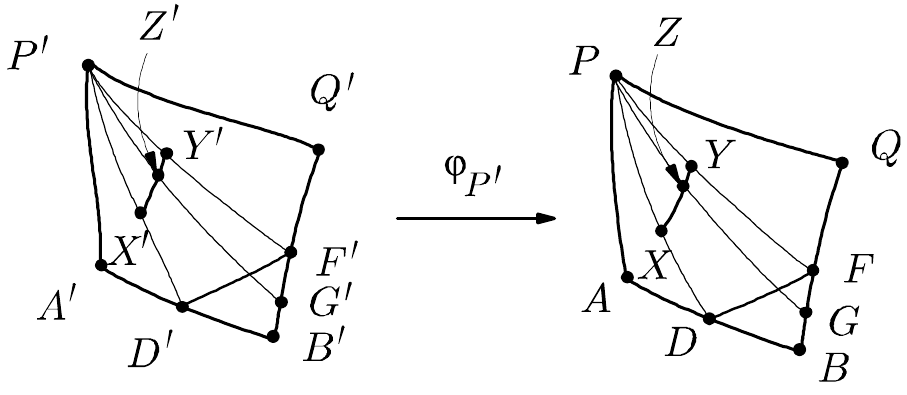}%
\caption{Sketch for part V of Lemma \ref{Lemma_hull}}%
\label{fig20}%
\end{center}
\end{figure}
%EndExpansion
For definiteness, suppose that $Z^{\prime}\in\mathcal{P}^{\prime}%
\mathcal{G}^{\prime}\backslash\left\{  P,G\right\}  $ where $G\in
\mathcal{B}^{\prime}\mathcal{Q}^{\prime}$. Set $G=\varphi_{P^{\prime}}\left(
G^{\prime}\right)  $. Then $PG=P^{\prime}G^{\prime}$. By using arguments of
II, we see that the polygonal curve $PZG$ coincides with the shortest joining
$P$ to $G$. Hence, $Z=\varphi_{P^{\prime}}\left(  Z^{\prime}\right)  $. Thus,
$\varphi_{P^{\prime}}:\mathcal{G}^{n+1}\left[  \mathfrak{O}^{\prime}\right]
\rightarrow\mathcal{G}^{n+1}\left[  \mathfrak{O}\right]  $ is a surjection.

The proof of Lemma \ref{Lemma_hull} is complete.
\end{proof}

Finally, we complete the proof of Theorem \ref{ThExtr}. By Theorem
\ref{MainTh}, $\left(  \mathcal{M},\rho\right)  $ is an $\Re_{K}$ domain. Let
$\operatorname{cosq}_{K}\left(  \overrightarrow{AP},\overrightarrow
{BQ}\right)  =1$. Because $\operatorname{diam}\left(  A,P,Q,B\right)
<\pi/\left(  2\sqrt{K}\right)  $ if $K>0$, we have: $AP+PQ+BQ+AB<2\pi/\sqrt
{K}$, and Reshetnyak's majorization theorem is applicable to the closed curve
$\mathcal{L=APQBA}$. So, as in the proof of Theorem \ref{ThRKBound}, consider
the closed polygonal curve $\mathcal{L}$ and a convex domain $\mathcal{V}%
\subseteq\mathbb{S}_{K}$ ($\partial\mathcal{V=}\mathcal{L}^{\prime
}=\mathcal{A}^{\prime}\mathcal{P}^{\prime}\mathcal{Q}^{\prime}\mathcal{B}%
^{\prime}\mathcal{A}^{\prime}$) majorizing the curve $\mathcal{L}$ and
satisfying (\ref{diag_cond}). Then, as we showed in the proof of Theorem
\ref{ThRKBound},
\[
\operatorname{cosq}_{K}\left(  \overrightarrow{AP},\overrightarrow{BQ}\right)
\leq\operatorname{cosq}_{K}\left(  \overrightarrow{A^{\prime}P^{\prime}%
},\overrightarrow{B^{\prime}Q^{\prime}}\right)  \leq1.
\]
If either $d=PB<d^{\prime}=P^{\prime}B^{\prime}$, or $f=AQ<f^{\prime
}=A^{\prime}Q^{\prime}$, then $1=\operatorname{cosq}_{K}\left(
\overrightarrow{AP},\overrightarrow{BQ}\right)  $ $<\operatorname{cosq}%
_{K}\left(  \overrightarrow{A^{\prime}P^{\prime}},\overrightarrow{B^{\prime
}Q^{\prime}}\right)  $, a contradiction. So, $f=f^{\prime}$ and $d=d^{\prime}$ follows.

Let $\operatorname{cosq}_{K}\left(  \overrightarrow{AP},\overrightarrow
{BQ}\right)  =-1 $. By the hypothesis, Reshetnyak's majorization theorem is
applicable to the closed curve $\mathcal{N=AQBPA}$. The reader should follow
the proof of Theorem 4.2 in \cite{BergNik2007a} to arrive at the same
conclusion $f=f^{\prime}$ and $d=d^{\prime}$.

So, if $\left\vert \operatorname{cosq}_{K}\left(  \overrightarrow
{AP},\overrightarrow{BQ}\right)  \right\vert =1$, then the quadruple $\left\{
A,P,B,Q\right\}  $ in $\left(  \mathcal{M},\rho\right)  $ is isometric to the
quadruple $\left\{  A^{\prime},P^{\prime},B^{\prime},Q^{\prime}\right\}  $ in
$\mathbb{S}_{K}$. Hence, the statement of Theorem \ref{ThExtr} follows from
Lemma \ref{Lemma_hull}.

\begin{example}
\label{Ex_toExtr_th}Theorem \ref{ThExtr} need not be true if we allow
$\operatorname{diam}\left(  \mathcal{M}\right)  =\pi/2$. Indeed, consider the
metric space $\left(  \mathcal{M},\rho\right)  =\left(  \mathcal{M}%
_{\varepsilon},\rho_{\varepsilon}\right)  $ of Example \ref{Ex_counter_1} for
$\varepsilon=0$. Notice that $\left(  \mathcal{M},\rho\right)  $ is an
$\Re_{1}$ domain, $\operatorname{diam}\left(  \mathcal{M}\right)  =\pi/2$ and
$\operatorname{cosq}_{1}\left(  \overrightarrow{PO},\overrightarrow
{BQ}\right)  =1$, whereas $\mathcal{GC}\left[  \left\{  B,Q,O,P\right\}
\right]  =\mathcal{M}$ cannot be isometric to a convex domain in the
half-sphere $\mathbb{S}_{1}$.
\end{example}

\section{Proof of Theorem \ref{Thsemimetr} \label{SecThsemimetr}}

In this section, we extend Theorem \ref{MainTh} to complete weakly convex
semimetric spaces satisfying the one-sided four point $\operatorname{cosq}%
_{K}$ condition. We begin with the following

\begin{lemma}
\label{Lemma_semimetr}Let $K\neq0$ and let $\left(  \mathcal{M},\rho\right)  $
be a semimetric space such that $p\left(  \mathcal{T}\right)  <2\pi/\sqrt{K} $
if $K>0$ for every triple \ of distinct points $\mathcal{T=}\left\{
A,B,C\right\}  $ in $\mathcal{M}$. If $\left(  \mathcal{M},\rho\right)  $
satisfies the one-sided four point $\operatorname{cosq}_{K}$ condition, then
$\left(  \mathcal{M},\rho\right)  $ is a metric space.
\end{lemma}

\begin{proof}
Set $a=BC,$ $b=AC$ and $c=AB$. We have to prove the triangle inequality for
$\rho$.

\textbf{I}. \textbf{Let }$\left(  \mathcal{M},\rho\right)  $\textbf{\ satisfy
the upper four point }$\operatorname{cosq}_{K}$\textbf{\ condition}. Then%
\[
\operatorname{cosq}_{K}\left(  \overrightarrow{CA},\overrightarrow{CB}\right)
=\frac{\cos\widehat{\kappa}c-\cos\widehat{\kappa}a\cos\widehat{\kappa}b}%
{\sin\widehat{\kappa}a\sin\widehat{\kappa}b}\leq1,
\]
whence%
\begin{equation}
\frac{\cos\widehat{\kappa}c-\cos\widehat{\kappa}\left(  b-a\right)  }%
{\sin\widehat{\kappa}a\sin\widehat{\kappa}b}\leq0\text{.} \label{ineq_sym}%
\end{equation}

If $K>0$, we get: $\cos\kappa c-\cos\kappa\left(  b-a\right)  \leq0$, whence
$b\leq a+c$. If $K<0$, we get $\cosh\kappa c-\cosh\kappa\left(  b-a\right)
\geq0$, whence $b\leq a+c$. Verification of remaining triangle inequalities
for $\mathcal{T}$ is similar.

\textbf{II}. \textbf{Let }$\left(  \mathcal{M},\rho\right)  $\textbf{\ satisfy
the lower four point }$\operatorname{cosq}_{K}$\textbf{\ condition}. Then%
\[
\operatorname{cosq}_{K}\left(  \overrightarrow{CA},\overrightarrow{BC}\right)
=\frac{\cos\widehat{\kappa}a\cos\widehat{\kappa}b-\cos\widehat{\kappa}c}%
{\sin\widehat{\kappa}a\sin\widehat{\kappa}b}\geq-1,
\]
whence, (\ref{ineq_sym}) follows. As in I, this implies the triangle
inequality for $\rho$.

The proof of Lemma \ref{Lemma_semimetr} is complete.
\end{proof}

By Lemma \ref{Lemma_semimetr}, $\left(  \mathcal{M},\rho\right)  $ is a metric
space. Next, we show that $\left(  \mathcal{M},\rho\right)  $ is a (complete)
geodesically connected metric space. Let $A,B\in\mathcal{M}$, $A\neq B$. By
weak convexity, there is $\lambda\in\left(  0,1\right)  $\ such that, for
every $n=1,2,...$, there is a point $C_{n}\in\mathcal{M}$%
\ satisfying\textit{\ }%
\[
\left\vert \rho\left(  A,C_{n}\right)  -\lambda\rho\left(  A,B\right)
\right\vert <1/n\text{ \textit{and} }\left\vert \rho\left(  B,C_{n}\right)
-\left(  1-\lambda\right)  \rho\left(  A,B\right)  \right\vert <1/n.
\]

We claim that\textit{\ }$\left\{  C_{n}\right\}  _{n=1,2,...}$\textit{\ }is a
Cauchy sequence.\textit{\ }The proof uses no new ideas beside those \ of the
proof of Lemma \ref{Lemma_cont_geod}. Indeed, in the proof of Lemma
\ref{Lemma_cont_geod}, for $m\neq n$, take $A_{n}:=A,$ $B_{n}:=B,$ $P:=C_{n}$,
$P_{n}:=C_{m}$, see Fig. \ref{fig10}. Set $l=AB$, and $\overline{\delta}%
_{m,n}=\overline{\lim}_{m,n\rightarrow\infty}C_{n}C_{m}$ and%
\begin{align*}
AC_{n}-\lambda AB  &  =\varepsilon_{n}^{\prime}\rightarrow0\text{ as
}n\rightarrow\infty,\\
BC_{n}-\left(  1-\lambda\right)  AB  &  =\varepsilon_{n}^{\prime\prime
}\rightarrow0\text{ as }n\rightarrow\infty\text{ }.
\end{align*}

If\textbf{\ }$\left(  \mathcal{M},\rho\right)  $\textbf{\ }satisfies the upper
four point $\operatorname{cosq}_{K}$\textbf{\ }condition, then by
(\ref{cont_g_c}),
\begin{align*}
\overline{\lim}_{m,n\rightarrow\infty}\operatorname{cosq}\left(
\overrightarrow{AC_{n}},\overrightarrow{C_{m}B}\right)   &  =\\
1+\frac{\left[  1-\cos\left(  \widehat{\kappa}\overline{\delta}_{mn}\right)
\right]  \left[  \cos\left(  \widehat{\kappa}\left(  1-\lambda\right)
l\right)  +\cos\left(  \widehat{\kappa}l\right)  \right]  }{\left(  \left[
1+\cos\left(  \widehat{\kappa}\lambda l\right)  \right]  \right)  \left(
\sin\left(  \widehat{\kappa}\lambda l\right)  \right)  \sin\left(
\widehat{\kappa}\left(  1-\lambda\right)  l\right)  }  &  \leq1,
\end{align*}
whence, as in the proof of Lemma \ref{Lemma_cont_geod}, $\overline{\delta
}_{mn}=0$ follows. The case of the lower four point $\operatorname{cosq}_{K}%
$\textbf{\ }condition is similar. Thus, we showed that \textit{\ }$\left\{
C_{n}\right\}  _{n=1,2,...}$\textit{\ }is a Cauchy sequence. By part (b) of
the hypothesis of Theorem \ref{Thsemimetr}, $\left(  \mathcal{M},\rho\right)
$ is a complete metric space. Hence, the sequence $\left\{  C_{n}\right\}
_{n=1,2,...}$ converges to a point $C\in\mathcal{M}$ such that $AC=\lambda AB$
and $BC=\left(  1-\lambda\right)  AB$. We readily see that $AB=AC+CB$. So,
every pair $A,B$ of distinct points of $\mathcal{M}$ has a point $C$ between
them. By Menger's theorem \cite{Bl1970}, Theorem 14.1, a complete convex
metric space is geodesically connected. Finally, Theorem \ref{Thsemimetr}
follows from Theorem \ref{MainTh}.

The proof of Theorem \ref{Thsemimetr} is complete.

\section{$K$-quadrilateral inequality condition\label{euler_ineq}}

In this section, we derive $K$-Euler's inequality, a generalization of a
familiar Euler's inequality \cite[Corollary 4]{Eul750} (also known as Enflo's
$2$-roundness condition \cite{Enf1969}) for $\Re_{K}$ domains for non-zero
$K$, and study the case of equality in $K$-Euler's inequality.

\subsection{$K$-Euler inequality in $\mathbb{S}_{K}$\label{SecSK}}

\begin{proposition}
\label{Prop_Euler_Eq_SK}Let $K\neq0$ and let $\mathcal{R}$ be a convex
quadrangular domain in $\mathbb{S}_{K}$ bounded by a closed polygonal curve
$\mathcal{L=ABCDA}$. Let $O_{1}$ be the midpoint of the shortest diagonal
$\mathcal{BD}$ and $O_{2}$ be the midpoint of the shortest diagonal
$\mathcal{AC}$. Set
\[
a=AB,\text{ }b=BC,\text{ }c=CD,\text{ }d=AD,\text{ }e=BD,\text{ }f=AC\text{
and }g=O_{1}O_{2},
\]
as illustrated in Fig. \ref{fig21}.%
%TCIMACRO{\FRAME{ftphFU}{1.547in}{1.547in}{0pt}{\Qcb{$K$-Euler's inequality in
%$\QTR{Bbb}{S}_{K}$}}{\Qlb{fig21}}{fig21.eps}%
%{\special{ language "Scientific Word";  type "GRAPHIC";
%maintain-aspect-ratio TRUE;  display "PICT";  valid_file "F";  width 1.547in;
%height 1.547in;  depth 0pt;  original-width 1.63in;  original-height 1.63in;
%cropleft "0";  croptop "1";  cropright "1";  cropbottom "0";
%filename 'Fig21.eps';file-properties "XNPEU";}} }%
%BeginExpansion
\begin{figure}
[pth]
\begin{center}
\includegraphics[
natheight=1.630000in,
natwidth=1.630000in,
height=1.547in,
width=1.547in
]%
{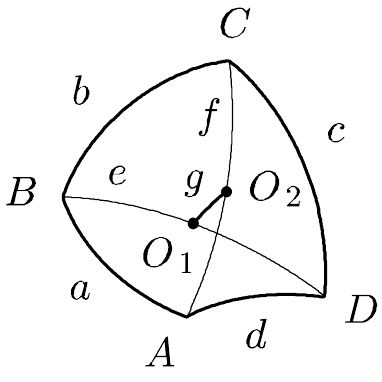}%
\caption{$K$-Euler's inequality in $\mathbb{S}_{K}$}%
\label{fig21}%
\end{center}
\end{figure}
%EndExpansion
Then the following equality, called $K$-Euler's equality holds:%
\[
\cos\widehat{\kappa}a+\cos\widehat{\kappa}b+\cos\widehat{\kappa}c+\cos
\widehat{\kappa}d=4\cos\widehat{\kappa}\frac{e}{2}\cos\widehat{\kappa}\frac
{f}{2}\cos\widehat{\kappa}g.
\]
In particular, if $K>0$, then
\[
\cos\kappa a+\cos\kappa b+\cos\kappa c+\cos\kappa d=4\cos\kappa\frac{e}{2}%
\cos\kappa\frac{f}{2}\cos\kappa g,
\]
and if $K<0$, then%
\[
\cosh\kappa a+\cosh\kappa b+\cosh\kappa c+\cosh\kappa d=4\cosh\kappa\frac
{e}{2}\cosh\kappa\frac{f}{2}\cosh\kappa g.
\]

\end{proposition}

\begin{proof}
Let $O$ be the point of intersection of the shortests $\mathcal{BD}$ and
$\mathcal{AC}$. Set $x=BO,$ $y=DO,$ $z=AO$, $w=OC$. There is no restriction in
assuming that $x\geq y$ and $w\geq z$. Set $\alpha=\measuredangle BOC$. By the
cosine formula in $\mathbb{S}_{K}$,%
\begin{align*}
\cos\widehat{\kappa}a  &  =\cos\widehat{\kappa}x\cos\widehat{\kappa}%
z-\sin\widehat{\kappa}x\sin\widehat{\kappa}z\cos\alpha,\\
\cos\widehat{\kappa}b  &  =\cos\widehat{\kappa}x\cos\widehat{\kappa}%
w+\sin\widehat{\kappa}x\sin\widehat{\kappa}w\cos\alpha,\\
\cos\widehat{\kappa}c  &  =\cos\widehat{\kappa}y\cos\widehat{\kappa}%
w-\sin\widehat{\kappa}y\sin\widehat{\kappa}w\cos\alpha,\\
\cos\widehat{\kappa}d  &  =\cos\widehat{\kappa}y\cos\widehat{\kappa}%
z+\sin\widehat{\kappa}y\sin\widehat{\kappa}z\cos\alpha,
\end{align*}
whence%
\begin{align*}
&  \cos\widehat{\kappa}a+\cos\widehat{\kappa}b+\cos\widehat{\kappa}%
c+\cos\widehat{\kappa}d\\
&  =\cos\widehat{\kappa}x\cos\widehat{\kappa}z+\cos\widehat{\kappa}%
x\cos\widehat{\kappa}w+\cos\widehat{\kappa}w\cos\widehat{\kappa}y+\cos
\widehat{\kappa}y\cos\widehat{\kappa}z+\\
&  \left(  -\sin\widehat{\kappa}x\sin\widehat{\kappa}z+\sin\widehat{\kappa
}x\sin\widehat{\kappa}w-\sin\widehat{\kappa}y\sin\widehat{\kappa}%
w+\sin\widehat{\kappa}y\sin\widehat{\kappa}z\right)  \cos\alpha.
\end{align*}
Notice that%
\begin{align*}
&  \cos\widehat{\kappa}x\cos\widehat{\kappa}z+\cos\widehat{\kappa}%
x\cos\widehat{\kappa}w+\cos\widehat{\kappa}w\cos\widehat{\kappa}y+\cos
\widehat{\kappa}y\cos\widehat{\kappa}z\\
&  =4\cos\widehat{\kappa}\frac{z+w}{2}\cos\widehat{\kappa}\frac{w-z}{2}%
\cos\widehat{\kappa}\frac{x+y}{2}\cos\widehat{\kappa}\frac{x-y}{2}\\
&  =4\cos\widehat{\kappa}\frac{e}{2}\cos\widehat{\kappa}\frac{f}{2}%
\cos\widehat{\kappa}\frac{x-y}{2}\cos\widehat{\kappa}\frac{w-z}{2},
\end{align*}
and
\begin{align*}
&  -\sin\widehat{\kappa}x\sin\widehat{\kappa}z+\sin\widehat{\kappa}%
x\sin\widehat{\kappa}w-\sin\widehat{\kappa}y\sin\widehat{\kappa}w+\sin
\widehat{\kappa}y\sin\widehat{\kappa}z\\
&  =4\sin\widehat{\kappa}\frac{x-y}{2}\cos\widehat{\kappa}\frac{x+y}{2}%
\sin\widehat{\kappa}\frac{w-z}{2}\cos\widehat{\kappa}\frac{z+w}{2}\\
&  =4\cos\widehat{\kappa}\frac{e}{2}\cos\widehat{\kappa}\frac{f}{2}%
\sin\widehat{\kappa}\frac{x-y}{2}\sin\widehat{\kappa}\frac{w-z}{2}.
\end{align*}
We have:
\[
OO_{1}=x-\frac{x+y}{2}=\frac{x-y}{2},\text{ }OO_{2}=w-\frac{z+w}{2}=\frac
{w-z}{2}.
\]
So, by the cosine formula in $\mathbb{S}_{K}$,
\begin{align*}
&  \cos\widehat{\kappa}a+\cos\widehat{\kappa}b+\cos\widehat{\kappa}%
c+\cos\widehat{\kappa}d\\
&  =4\cos\widehat{\kappa}\frac{f}{2}\cos\widehat{\kappa}\frac{e}{2}\left(
\cos\widehat{\kappa}\frac{w-z}{2}\cos\widehat{\kappa}\frac{x-y}{2}%
+\sin\widehat{\kappa}\frac{w-z}{2}\sin\widehat{\kappa}\frac{x-y}{2}\cos
\alpha\right) \\
&  =4\cos\widehat{\kappa}\frac{f}{2}\cos\widehat{\kappa}\frac{e}{2}%
\cos\widehat{\kappa}g,
\end{align*}
as needed.

The proof of Proposition \ref{Prop_Euler_Eq_SK} is complete.
\end{proof}

\begin{corollary}
[$K$-Euler's inequality in $\mathbb{S}_{K}$]\label{Cor_eul}Under the
hypothesis of Proposition \ref{Prop_Euler_Eq_SK}, the following inequalities
hold:\newline(a) If $K>0$, then
\[
\cos\kappa a+\cos\kappa b+\cos\kappa c+\cos\kappa d\leq4\cos\kappa\frac{e}%
{2}\cos\kappa\frac{f}{2},
\]
(b) If $K<0$, then%
\[
\cosh\kappa a+\cosh\kappa b+\cosh\kappa c+\cosh\kappa d\geq4\cosh\kappa
\frac{e}{2}\cosh\kappa\frac{f}{2}.
\]

\end{corollary}

\subsection{$K$-Euler's inequality in $\Re_{K}$}

\begin{theorem}
\label{Th_quad_neq}Let $K\neq0$ and $\mathfrak{Q=}\left\{  A,B,C,D\right\}  $
be a quadruple of distinct points in an $\Re_{K}$ domain. If $AB+BC+CD+AD<2\pi
/\sqrt{K}$ if $K>0$, then the following inequalities (called $K$-Euler's, or
$K$-quadrilateral inequalities) hold:\newline(a) If $K>0$, then
\begin{equation}
\cos\kappa a+\cos\kappa b+\cos\kappa c+\cos\kappa d\leq4\cos\kappa\frac{e}%
{2}\cos\kappa\frac{f}{2}, \label{Eul_ineqKpos}%
\end{equation}
(b) If $K<0$, then%
\begin{equation}
\cosh\kappa a+\cosh\kappa b+\cosh\kappa c+\cosh\kappa d\geq4\cosh\kappa
\frac{e}{2}\cosh\kappa\frac{f}{2}, \label{Eul_ineqKneg}%
\end{equation}
where we use the notation of Sec. \ref{SecSK}.
\end{theorem}

\begin{proof}
Consider the closed polygonal curve $\mathcal{L=ABCDA}$ in $\Re_{K}$. We are
given that the length of $\mathcal{L}$ is less than $2\pi/\sqrt{K}$ if $K>0$.
By Reshetnyak's majorization theorem, there is a convex domain
$\mathcal{V\subseteq}\mathbb{S}_{K}$ bounded by a polygonal curve
$\mathcal{L}^{\prime}=\mathcal{A}^{\prime}\mathcal{B}^{\prime}\mathcal{C}%
^{\prime}\mathcal{D}^{\prime}\mathcal{A}^{\prime}$ such that
\begin{align*}
a  &  =AB=a^{\prime}=A^{\prime}B^{\prime},\text{ }b=BC=b^{\prime}=B^{\prime
}C^{\prime},\\
c  &  =CD=c^{\prime}=C^{\prime}D^{\prime},\text{ }d=AD=d^{\prime}=A^{\prime
}D^{\prime},\\
e  &  =BD\leq e^{\prime}=B^{\prime}D^{\prime}\text{ and }f=AC\leq f^{\prime
}=A^{\prime}C^{\prime}.
\end{align*}
Let $K>0$. By invoking Corollary \ref{Cor_eul}, we get:%
\begin{align*}
&  \cos\kappa a+\cos\kappa b+\cos\kappa c+\cos\kappa d\\
&  =\cos\kappa a^{\prime}+\cos\kappa b^{\prime}+\cos\kappa c^{\prime}%
+\cos\kappa d^{\prime}\leq\\
4\cos\kappa\frac{e^{\prime}}{2}\cos\kappa\frac{f^{\prime}}{2}  &  \leq
4\cos\kappa\frac{e}{2}\cos\kappa\frac{f}{2},
\end{align*}
as needed. The case of negative $K$ is treated in a similar way.

The proof of Theorem \ref{Th_quad_neq} is complete.
\end{proof}

\subsection{Extremal theorem for $K$-Euler's inequality}

The following theorem extends the second part of Theorem 6 in
\cite{BergNik2008} to the case of non-zero $K$.

\begin{theorem}
\label{Th_eul_extr}Let $K\neq0$ and $\mathfrak{Q=}\left\{  A,B,C,D\right\}  $
be a quadruple of distinct points in an $\Re_{K}$ domain. Suppose that
$AB+BC+CD+AD<2\pi/\sqrt{K}$ if $K>0$. Then the equality sign in $K$-Euler's
inequality (\ref{Eul_ineqKpos}) \ for positive $K$ and in (\ref{Eul_ineqKneg})
for negative $K$ holds if and only if the geodesic convex hull of
$\mathfrak{Q}$ is isometric to a parallelogramoidal domain $\mathcal{V}$ in
$\mathbb{S}_{K}$, i.e., a segment of straight line or a closed domain bounded
by a closed polygonal curve $\mathcal{L}^{\prime}=\mathcal{A}^{\prime
}\mathcal{B}^{\prime}\mathcal{C}^{\prime}\mathcal{D}^{\prime}\mathcal{A}%
^{\prime}$ such that $\operatorname{cosq}_{K}\left(  \overrightarrow
{A^{\prime}D^{\prime}},\overrightarrow{C^{\prime}B^{\prime}}\right)  =-1$ and
$x=y$ (and thereby, $a=b$).
\end{theorem}

\begin{proof}
We can assume that $\mathfrak{Q}$ is not isometric to a quadruple of points in
$%
%TCIMACRO{\U{211d} }%
%BeginExpansion
\mathbb{R}
%EndExpansion
$. Set
\[
a=AB,\text{ }y=BC,\text{ }b=CD,\text{ }x=AD,\text{ }d=AC\text{ and }f=BD.
\]
Let $K>0$.

\textbf{I. } Suppose that%
\begin{equation}
\cos\kappa a+\cos\kappa b+\cos kx+\cos\kappa y=4\cos\kappa\frac{f}{2}%
\cos\kappa\frac{d}{2}. \label{Eul_eq}%
\end{equation}
By Reshetnyak's majorization theorem, there is a convex domain $\mathcal{V}$
bounded by closed polygonal curve $\mathcal{L}^{\prime}=\mathcal{A}^{\prime
}\mathcal{B}^{\prime}\mathcal{C}^{\prime}\mathcal{D}^{\prime}$ in
$\mathbb{S}_{K}$ majorizing the closed polygonal curve $\mathcal{L=ABCD}$. Let
$d^{\prime}=A^{\prime}C^{\prime}$ and $f^{\prime}=B^{\prime}D^{\prime}$, as
illustrated in Fig. \ref{fig22}.%
%TCIMACRO{\FRAME{ftphFU}{2.8726in}{1.3219in}{0pt}{\Qcb{Sketch for Theorem
%\ref{Th_eul_extr}}}{\Qlb{fig22}}{fig22.eps}%
%{\special{ language "Scientific Word";  type "GRAPHIC";
%maintain-aspect-ratio TRUE;  display "PICT";  valid_file "F";
%width 2.8726in;  height 1.3219in;  depth 0pt;  original-width 3.0497in;
%original-height 1.3874in;  cropleft "0";  croptop "1";  cropright "1";
%cropbottom "0";  filename 'Fig22.eps';file-properties "XNPEU";}} }%
%BeginExpansion
\begin{figure}
[pth]
\begin{center}
\includegraphics[
natheight=1.387400in,
natwidth=3.049700in,
height=1.3219in,
width=2.8726in
]%
{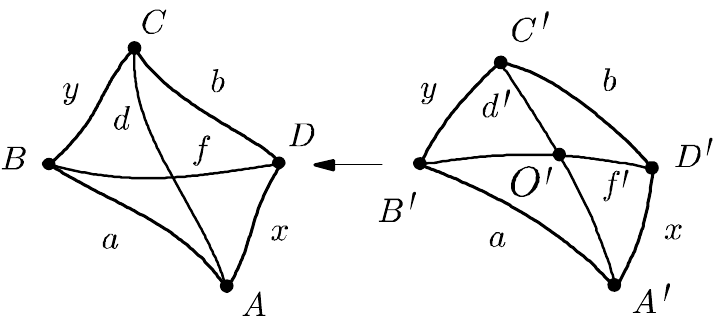}%
\caption{Sketch for Theorem \ref{Th_eul_extr}}%
\label{fig22}%
\end{center}
\end{figure}
%EndExpansion
We have: $f\leq f^{\prime} $ and $d\leq d^{\prime}$, whence by recalling
Corollary \ref{Cor_eul}, we get:
\begin{align*}
\cos\kappa a+\cos\kappa b+\cos kx+\cos\kappa y  &  =4\cos\kappa\frac{f}{2}%
\cos\kappa\frac{d}{2}\geq4\cos\kappa\frac{f^{\prime}}{2}\cos\kappa
\frac{d^{\prime}}{2}\geq\\
&  \cos\kappa a+\cos\kappa b+\cos kx+\cos\kappa y.
\end{align*}
Hence, by Proposition \ref{Prop_Euler_Eq_SK},
\[
4\cos\kappa\frac{f^{\prime}}{2}\cos\kappa\frac{d^{\prime}}{2}\cos\kappa
g^{\prime}=4\cos\kappa\frac{f^{\prime}}{2}\cos\kappa\frac{d^{\prime}}{2},
\]
where $g^{\prime}$ is the distance between the midpoints of the shortests
$\mathcal{A}^{\prime}\mathcal{C}^{\prime}$ and $\mathcal{B}^{\prime
}\mathcal{D}^{\prime}$. So, $g^{\prime}=0$, that is, the shortests
$\mathcal{A}^{\prime}\mathcal{C}^{\prime}$ and $\mathcal{B}^{\prime
}\mathcal{D}^{\prime} $ intersect at their common midpoint $O^{\prime}$. By
recalling the geometric interpretation of $\operatorname{cosq}_{K}$ in
$\mathbb{S}_{K}$ in Sec. \ref{cosqk_in_SK} (see Fig. \ref{fig2} where
$A:=A^{\prime},$ $B:=C^{\prime}, $ $P:=D^{\prime}$, $P^{\prime}:=B^{\prime}$
and $O:=O^{\prime} $), we readily see that $\operatorname{cosq}_{K}\left(
\overrightarrow{A^{\prime}D^{\prime}},\overrightarrow{C^{\prime}B^{\prime}%
}\right)  =-1$, $x=y $ and $a=b$.

Now we show that $f=f^{\prime}$ and $d=d^{\prime}$. Indeed, recall that $f\leq
f^{\prime}$ and $d\leq d^{\prime}$. If, say, $f<f^{\prime}$, then because $f,$
$d,$ $f^{\prime},$ $d^{\prime}\in(0,\pi)$,%
\[
\cos\kappa a+\cos\kappa b+\cos kx+\cos\kappa y=4\cos\kappa\frac{f^{\prime}}%
{2}\cos\kappa\frac{d^{\prime}}{2}<4\cos\kappa\frac{f}{2}\cos\kappa\frac{d}%
{2},
\]
a contradiction because of (\ref{Eul_eq}). Hence, the quadruple $\mathfrak{Q}%
=\left\{  A,\text{ }B,\text{ }C,\text{ }D\right\}  $ is isometric to the
quadruple $\mathfrak{Q}^{\prime}=\left\{  A^{\prime},\text{ }B^{\prime},\text{
}C^{\prime},\text{ }D^{\prime}\right\}  $, whence, by Lemma \ref{Lemma_hull},
$\mathcal{GC}\left[  \mathfrak{Q}\right]  $ is isometric to the
parallelogramoidal domain $\mathcal{V}$, as claimed.

\textbf{II. }Let $f$ be an isometry from $\mathcal{GC}\left[  \mathfrak{Q}%
\right]  $ onto a parallelogramoidal domain $\mathcal{V}$ in $\mathbb{S}_{K}$
bounded by the closed polygonal curve $\mathcal{L}^{\prime}=\mathcal{A}%
^{\prime}\mathcal{B}^{\prime}\mathcal{C}^{\prime}\mathcal{D}^{\prime}$ such
that $f\left(  A\right)  =A^{\prime},$ $f\left(  B\right)  =B^{\prime},$
$f\left(  C\right)  =C^{\prime}$ and $f\left(  D\right)  =D^{\prime}$. As we
mentioned in I, the shortests $\mathcal{A}^{\prime}\mathcal{C}^{\prime}$ and
$\mathcal{B}^{\prime}\mathcal{D}^{\prime}$ intersect at the common midpoint
$O^{\prime}$, i.e., $g^{\prime}=0$. Hence, by Proposition
\ref{Prop_Euler_Eq_SK},
\[
\cos\kappa a+\cos\kappa b+\cos kx+\cos\kappa y=4\cos\kappa\frac{f^{\prime}}%
{2}\cos\kappa\frac{d^{\prime}}{2}=4\cos\kappa\frac{f}{2}\cos\kappa\frac{d}%
{2},
\]
as needed.

The case of negative $K$ is similar.

The proof of Theorem \ref{Th_eul_extr} is complete.
\end{proof}

\section{Remarks\label{Remarks}}

In Sec. 7, part I, Example 21 in \cite{BergNik2008}, we showed that, for an
individual quadruple of points, the four point $\operatorname{cosq}_{0}$
condition need not imply $0$-concavity, Berestovskii's embeddability condition
or Reshetnyak's majorization condition for $K=0$. It is not difficult to
construct a similar example for non-zero $K$.

\begin{example}
\label{concl_remarks}Let $\mathfrak{Q}\mathcal{=}\left\{  A,B,C,O\right\}  $
be a four element set. The six (symmetric) distances between the pairs of
points in $\mathfrak{Q}$ are given by
\begin{align*}
\rho\left(  A,B\right)   &  =0.8,\text{ }\rho\left(  B,C\right)  =1,\text{
}\rho\left(  C,O\right)  =0.95,\\
\rho\left(  A,O\right)   &  =0.4,\text{ }\rho\left(  B,O\right)  =0.4\text{
and }\rho\left(  A,C\right)  =1.
\end{align*}
It is easy to see that $\rho$ is a metric. If we take $A:=A,P:=B,$ $B:=O$ and
$Q:=C$, then in the notation of Sec. \ref{Counter_Examples} all $12$ main
(approximate) values of $\operatorname{cosq}_{1}$ and $\operatorname{cosq}%
_{-1}$ for the four point metric space $\left(  \mathcal{M},\rho\right)  $ are
given in Tables \ref{table6} and \ref{table7}.\newline%
%TCIMACRO{\TeXButton{TeX field}{\begin{table}[h]
%\begin{center}}}%
%BeginExpansion
\begin{table}[h]
\begin{center}%
%EndExpansion
$%
\begin{tabular}
[c]{|l|l|l|l|l|l|l|}\hline
Case & I & II & III & IV & V & VI\\\hline
$\operatorname{cosq}_{1}$ & $0.0012$ & $0.2048$ & $-0.2865$ & $0.6466$ &
$-0.2865$ & $0.2841$\\\hline
Case & VII & VIII & IX & X & XI & XII\\\hline
$\operatorname{cosq}_{1}$ & $0.0012$ & $0.2048$ & $0.6466$ & $0.2841$ &
$-0.4756$ & $-0.4756$\\\hline
\end{tabular}
\ $%
%TCIMACRO{\TeXButton{TeX field}{\end{center}
%\caption{Example \ref{concl_remarks}, $K=1$}
%\label{table6}
%\end{table}}}%
%BeginExpansion
\end{center}
\caption{Example \ref{concl_remarks}, $K=1$}
\label{table6}
\end{table}%
%EndExpansion
\newline\newline%
%TCIMACRO{\TeXButton{TeX field}{\begin{table}[h]
%\begin{center}}}%
%BeginExpansion
\begin{table}[h]
\begin{center}%
%EndExpansion
$%
\begin{tabular}
[c]{|l|l|l|l|l|l|l|}\hline
Case & I & II & III & IV & V & VI\\\hline
$\operatorname{cosq}_{-1}$ & $-0.0106$ & $-0.1647$ & $-0.6208$ & $0.3287$ &
$-0.6208$ & $0.6406$\\\hline
Case & VII & VIII & IX & X & XI & XII\\\hline
$\operatorname{cosq}_{-1}$ & $-0.0106$ & $-0.1647$ & $0.3287$ & $0.6406$ &
$-0.4887$ & $-0.4887$\\\hline
\end{tabular}
\ $%
%TCIMACRO{\TeXButton{TeX field}{\end{center}
%\caption{Example \ref{concl_remarks}, $K=-1$}
%\label{table7}
%\end{table}}}%
%BeginExpansion
\end{center}
\caption{Example \ref{concl_remarks}, $K=-1$}
\label{table7}
\end{table}%
%EndExpansion
\newline\newline Hence, $\left(  \mathfrak{Q},\rho\right)  $ satisfies the
upper four point $\operatorname{cosq}_{K}$ condition and the lower four point
$\operatorname{cosq}_{K}$ condition for $K=\pm1$. Notice, that $\mathfrak{Q}$
is a triangular quadruple: $O$ is between $A$ and $B$. The quadruple
$\mathfrak{O}$ is not a non-rectilinear quadruple satisfying Case A in
\cite{Ber1986}, as is required in Theorem 5 in \cite[Sec. 3.]{Ber1986}. Let
$\mathcal{T}_{+}^{\prime}=A_{+}^{\prime}B_{+}^{\prime}C_{+}^{\prime}$ be a
triangle in $\mathbb{S}_{1}$ and $\mathcal{T}_{-}^{\prime}=A_{-}^{\prime}%
B_{-}^{\prime}C_{-}^{\prime}$ in $\mathbb{S}_{-1}$ be such that the triple
$\left\{  A,B,C\right\}  $ is isometric to $\left\{  A_{+}^{\prime}%
,B_{+}^{\prime},C_{+}^{\prime}\right\}  $ and $\left\{  A_{-}^{\prime}%
,B_{-}^{\prime},C_{-}^{\prime}\right\}  $. Let $O_{+}^{\prime}$ be the
midpoint of the shortest $\mathcal{A}_{+}^{\prime}\mathcal{B}_{+}^{\prime}$
and $O_{-}^{\prime}$ be the midpoint of the shortest $\mathcal{A}_{-}^{\prime
}\mathcal{B}_{-}^{\prime}$. By Lemma \ref{LemmaBruhat} and similar formula for
$K=-1$, both approximate values for $C_{+}^{\prime}O_{+}^{\prime}$ and
$C_{-}^{\prime}O_{-}^{\prime}$ are easy to calculate:
\begin{align*}
C_{+}^{\prime}O_{+}^{\prime}  &  =\arccos\left(  \frac{\cos1}{\cos0.4}\right)
\approx0.943\,9<0.95=CO\text{ and}\\
C_{-}^{\prime}O_{-}^{\prime}  &  =\operatorname{arccosh}\left(  \frac{\cosh
1}{\cosh0.4}\right)  \approx0.8944<0.95=CO\text{.}%
\end{align*}
Thus, the $K$-concavity condition fails for the triangular quadruple
$\mathfrak{Q}$, and, as a corollary, both Berestovskii's embeddability
condition and Reshetnyak's majorization condition for $K=\pm1$ fail.
\end{example}

In (c) of Part I in \cite[Sec. 7]{BergNik2008}, we erroneously omitted the
condition that the triangular quadruple cannot be rectilinear and it cannot
satisfy case A in \cite{Ber1986}. We thank Professor Berestovskii for pointing
this out in a personal communication.

\end{document}